\documentclass{article}
\usepackage{amsmath,amssymb,amsthm,amsfonts,amscd}
\usepackage[hypertex]{hyperref}
\usepackage[numeric]{amsrefs}
\usepackage{rotating} 
\input colordvi
\usepackage{graphicx}

\newtheorem{thm}{Theorem}
\newtheorem{cor}{Corollary}
\newtheorem{lem}{Lemma}

\newcommand{\thmref}[1]{theorem~\ref{#1}}

\newcommand{\secref}[1]{section~\ref{#1}}

\numberwithin{equation}{section}

\renewcommand\a{\alpha}

\newcommand\g{\gamma}
\renewcommand\d{\delta}
\newcommand\e{\varepsilon}
\renewcommand\l{\lambda}

\newcommand\D{\Delta}

\renewcommand\D{\Delta}
\newcommand\G{\Gamma}
\newcommand\f{\frac}

\newcommand{\Z}{{\mathbb{Z}}}
\newcommand{\R}{{\mathbb{R}}}

\newcommand{\C}{{\mathbb{C}}}
\newcommand{\A}{{\mathbb{A}}}
\newcommand{\Q}{{\mathbb{Q}}}

\newcommand{\U}{{\mathbb{H}}}

\renewcommand\i{^{-1}}
\renewcommand\({\left(}
\renewcommand\){\right)}
\newcommand{\ttwo}[4]{
\(\begin{smallmatrix}{#1} & {#2}
\\ {#3} & {#4} \end{smallmatrix}\)}
\newcommand{\tthree}[9]{
\(\begin{smallmatrix}{#1} & {#2} & {#3}
\\ {#4} & {#5} & {#6} \\ {#7} & {#8} & {#9} \end{smallmatrix}\)}

\newcommand{\sgn}{\operatorname{sgn}}

\newcommand\srel[2]{\begin{smallmatrix} {#1} \\ {#2} \end{smallmatrix}}
\newcommand\nrel[2]{\begin{matrix} {#1} \\ {#2} \end{matrix}}
\newcommand{\gobble}[1]{}
  \newcommand{\rangeref}[2]{%
    \ref{#1}--\afterassignment\gobble\fam 0\ref{#2}%
  }

\makeatletter
\def\imod#1{\allowbreak\mkern5mu({\operator@font mod}\,#1)}
\makeatother

\begin{document}

\title{Fourier coefficients of automorphic forms, character variety orbits, and small representations}

\date{April 15, 2012}

\author{Stephen D. Miller\thanks{Supported by NSF grant DMS-0901594.}~  and Siddhartha Sahi\\
Rutgers University\\
\tt{miller@math.rutgers.edu, sahi@math.rutgers.edu}}

\maketitle

\begin{abstract}

We consider the Fourier expansions of automorphic forms on general Lie groups, with a particular emphasis on exceptional groups.  After describing some principles underlying known results on $GL(n)$, $Sp(4)$, and $G_2$, we perform an analysis of the expansions on split real forms of $E_6$ and $E_7$ where simplifications take place for automorphic realizations of real representations which have small Gelfand-Kirillov dimension.  Though the character varieties are more complicated for exceptional groups, we explain how the nonvanishing Fourier coefficients for small representations behave analogously to Fourier coefficients on $GL(n)$.  We use this mechanism, for example, to show that the minimal representation of either $E_6$ or $E_7$ never occurs in the cuspidal automorphic spectrum.  We also give a complete description of the {\em internal Chevalley modules} of all complex Chevalley groups -- that is, the orbit decomposition of the Levi factor of a maximal parabolic  on  its unipotent radical.  This generalizes classical results on trivectors and in particular includes a full description of the complex character variety orbits for all maximal parabolics. The results of this paper have been applied in the string theory literature to the study of BPS instanton contributions to graviton scattering \cite{GMV}.

\end{abstract}

\section{Introduction}

The most common way to dissect a modular form is to take its Fourier expansion.  Any smooth function $f$ on the complex upper half plane $\U = \{x+i y | y >0 \}$ that is periodic in $x$ with period $r$ has an absolutely convergent Fourier series expansion
\begin{equation}\label{fourierseries1}
    f(x+iy) \ \ = \ \ \sum_{n\,\in\,\Z} a_n(y)\,e^{2\pi i n x/r}\,,
\end{equation}
with coefficients that depend on $y$.  The holomorphy of $f$ further demands that each coefficient $a_n(y)$ satisfy a first order differential equation, whose general solution is a scalar multiple of $e^{-2\pi i ny/r}$.
 The property that modular forms are bounded as $y\rightarrow\infty$ requires this solution to vanish if $n<0$, so $f$ furthermore has the form
 \begin{equation}\label{fourierseries2}
    f(x+iy) \ \ = \ \ \sum_{n\,\ge\,0} a_n\,e^{2\pi i n (x+iy)/r}
 \end{equation}
for some  coefficients $a_n \in \C$.
This is the well-known {\em $q$-expansion} of a classical holomorphic modular form.
Any expression of the form (\ref{fourierseries2}) is of course periodic; the modularity of $f$ is deeper and comes from identities satisfied by the $a_n$.  A similar argument applies to the non-holomorphic Maass forms, which are instead eigenfunctions of the non-euclidean laplacian $\D=-y^2(\f{\partial^2}{\partial x^2}+\f{\partial^2}{\partial y^2})$.  In this case a differential equation and boundedness condition is again used to pin down the coefficient $a_n(y)$ as a Bessel function times a scalar  coefficient $a_n$.

The Fourier expansions of classical, holomorphic modular forms reveal a tremendous amount of arithmetic information, such as Hecke eigenvalues and point counts of elliptic curves over varying finite fields.  They also play a crucial analytic role, as they completely determine the form  and provide the Dirichlet series coefficients for its $L$-functions.  Thanks to the work of Hecke \cite{Hecke}, Maass \cite{Maass}, Jacquet-Langlands \cite{jacquetlanglands}, and Atkin-Lehner \cite{atkinlehner}, there is now a very complete theory of  Fourier expansions for all $GL(2)$ automorphic forms, that is, the classical holomorphic modular forms, the non-holomorphic Maass forms, and Eisenstein series.

The expansion (\ref{fourierseries1}-\ref{fourierseries2}) can be interpreted group theoretically as follows. Suppose $F$ is now a function on a Lie group $G$ that is left invariant under a discrete subgroup $\G$.  This includes the case of classical modular forms for $G=SL(2,\R)$ or $GL(2,\R)$, by setting $F(g)=f(g\cdot i)$, where $g\cdot i=(\sgn\det g)\f{ai+b}{ci+d}$ is the point in the upper half plane mapped from $i$ by the fractional linear transformation corresponding to $g=\ttwo abcd$.
Suppose that $U$ is an abelian subgroup for which $\G\cap U\backslash U$ is compact  and consequently has finite volume (normalized to be 1) under its Haar measure $du$.  We may then   expand
\begin{equation}\label{fourierseries3}
    F(g) \ \ = \ \ \sum_{\chi \, \in \, {\mathfrak C}(\G\cap U)} F_\chi(g) \,,
\end{equation}
where ${\mathfrak C}(\G\cap U)$ is the group\footnote{Even though the notation does not explicitly reflect the ambient group $U$, we will primarily work with arithmetic subgroups $\G$ and unipotent groups $U$, making it possible to recover  $U$   by taking the Zariski closure of $\G\cap U$.}  of characters on $U$ which are trivial on $\G\cap U$, and
\begin{equation}\label{fourierseries4}
    F_\chi(g) \ \ = \ \ \int_{\G\cap U\backslash U} F(ug)\,\chi(u)\i\,du\,.
\end{equation}  Note that each term $F_\chi$ satisfies the transformation law
\begin{equation}\label{fourierseries5}
    F_\chi(ug) \ \ = \ \ \chi(u)\,F_\chi(g) \ \, , \ \ \ \ u\,\in\,U\,,\ g\,\in\,G\,,
\end{equation}
and is so  determined by its restriction to the quotient $U\backslash G$ -- much  like the functions $a_n(y)$ in (\ref{fourierseries1}) depend  only on a single real variable.  Thus (\ref{fourierseries4}) is the Fourier expansion of $F(ug)$ evaluated at $u=e$, and each $F_\chi$ is a Fourier coefficient.

In our example of the hyperbolic upper half plane $\U$, $U$ corresponds to the subgroup of $2\times 2$ unit upper triangular matrices, which is isomorphic to the real line $\R$;  $\G\cap U$ corresponds to $r\Z$ under this isomorphism.  The ability to write $F_\chi$ there in terms of an unknown scalar multiple of an explicit special function depended on a differential equation; in more modern terms, it has to do with dimensionality of certain functionals on a representation space.  In any event, it is a special circumstance that does not occur for every choice of abelian subgroup $U$:~for example,  (\ref{fourierseries3}) contains no information at all when $U$ is the trivial group.  This reflects the tension that -- even if one somehow relaxes the requirement that $U$ be abelian --  the larger $U$ is, the easier it may be to pin down functions satisfying (\ref{fourierseries5}) and a relevant differential equation, but the more complicated $U$ may become.  Of course (\ref{fourierseries3}) fails to hold without the assumption that $U$ is abelian:~in general the sum on the righthand side of (\ref{fourierseries3}) must be augmented by other terms coming from higher dimensional representations of $U$, even though each $F_\chi$ from (\ref{fourierseries4}) still makes sense (see (\ref{charexp1})).

As a manifestation of this tension, an expansion like (\ref{fourierseries2}) is hard to come by in most situations.
  A famous theorem, proven independently by Piatetski-Shapiro \cite{psmult} and Shalika \cite{shalika}, gives a type of Fourier expansion for cusp forms on $GL(n)$ by an inductive argument involving  abelian subgroups of the maximal unipotent subgroup $N=N_n=\{n\times n$~unit upper triangular matrices$\}$.  Their result is quite general but is somewhat cumbersome to state for congruence subgroups and number fields.  For that reason let us consider a cusp form $F$ for $GL(n,\Z)\backslash GL(n,\R)$, and let $N'=[N,N]$ denote the derived subgroup of $N$.  All characters of $N(\R)$ are trivial on $N'(\R)$, and so
  \begin{equation}\label{psshalexpn2}
    V(g) \ \ = \  \ \int_{N'(\Z)\backslash N'(\R)}F(ng)\,dn\
\end{equation}
represents the sum (\ref{fourierseries3}) over all  characters of $N_n(\R)$ that are trivial on $N_n(\Z)$.
This Fourier expansion is entirely analogous  to the $GL(2)$ expansion (\ref{fourierseries2}), but with coefficients indexed by $n-1$ integral parameters and a different special function (a ``Whittaker'' function) which we shall say more about in section~\ref{sec:PSshal}.  As such, $V(g)$ can be thought of as the contribution of the ``abelian'' terms in the Fourier expansion of $F(g)$ -- the ones that come from the abelianization of $N$.  However, unlike the case of $n=2$, in general the integration in (\ref{psshalexpn2}) loses information about $F(g)$.  Piatetski-Shapiro and Shalika proved that nevertheless
$F(g)$
   can be reconstructed as the sum of translates of $V$,
\begin{equation}\label{psshalexpn1}
    F(g) \ \ = \ \ \sum_{\g \in N_{n-1}(\Z)\backslash GL(n-1,\Z)} V\(\ttwo{\g}{}{}{1} g\)\,.
\end{equation}
Section~\ref{sec:PSshal} contains more details about the proof of this formula and its relation to Whittaker functions.

   The Piatetski-Shapiro/Shalika expansion has been extremely useful in the analytic theory of automorphic forms on $GL(n)$, perhaps most famously because it allows one to reconstruct a form in terms of its abelian Fourier coefficients -- in particular, coefficients which have a direct arithmetic nature, e.g., are directly related to  $L$-function data.   Such a result is a particularly friendly special feature of  the general linear group that is absent in the general situation; this is because their argument  relies on an essential special fact about character variety orbits in $GL(n)$.
However, even though their statement does not generalize and the orbit structure becomes considerably more complicated, one can still derive important pieces of the Fourier expansion of an automorphic form from their approach. In this paper we give a generalization of the Piatetski-Shapiro/Shalika expansion to automorphic forms on arbitrary reductive algebraic groups.  This generalization is rarely as precise as (\ref{psshalexpn1}), but we indicate some conditions (on both the group and the cusp form) under which it simplifies to have a comparable form, or at least a useful enough form for some applications.  More  detailed Fourier expansions have been given in a number of particular examples by earlier authors, which we try to review in \secref{sec:characterexpansions}.  The study of Fourier coefficients of automorphic forms is a very rich subject which is too broad to recount here.  Our main goal is to say something general for a broad range of groups, including details for exceptional groups.  We would like to mention the recent paper \cite{ginzhund}, which takes a complementary approach on exceptional groups.

  In a different direction, we apply results of Matumoto \cite{matumoto} to show that many of the Fourier coefficients $F_\chi$ from (\ref{fourierseries4})  vanish if the archimedean component of its automorphic representation is ``small'' in the sense of having a small wavefront set.  These results are analogs of a related nonarchimedean vanishing theorem of M\oe glin-Waldspuger \cite{Mowa}.  We give a detailed analysis for maximal parabolic subgroups of exceptional groups in section~\ref{sec:tables}.  These results are in turn used in \cite[\S6]{GMV} to verify string-theoretic conjectures about the vanishing of certain Fourier coefficients of automorphic forms (without having to explicitly compute them).
  Indeed, providing background results for the investigation in \cite{GMV} was a primary motivation for writing this paper.
  However, we also pursue some more general statements that are perhaps of wider interest internally to  automorphic forms.  For example, our methods show the following:

  \begin{thm}\label{ournewtheorem} Let $G$ denote a split Chevalley group  of type $E_6$ or $E_7$, and let  $\pi$ be an (adelic) automorphic representation of $G$ over a number field $k$ for which at least one component $\pi_v$ is a minimal representation of $G(k_v)$  -- that is, the wavefront set of $\pi_v$ is the closure of the smallest nontrivial coadjoint nilpotent orbit.  Fix a choice of positive roots and let $N$ be the maximal unipotent $k$-subgroup of $G$ generated by their root vectors. Then the vectors $F$ in the representation space for $\pi$ are completely determined by the degenerate Whittaker integrals
\begin{equation}\label{degenwhitint}
    \int_{N(k)\backslash N(\A_k)} F(ng)\,\psi(n)\i\,dn\,,
\end{equation}
in which   $\psi:N(\A_k)\rightarrow\C$  is trivial on $N(k)$ and on all but at most a single one-parameter subgroup corresponding to a simple positive root.
  \end{thm}
In fact, the argument gives a formula for $F$ analogous to (\ref{psshalexpn1}) for these automorphic realizations of minimal representations (see \cite{kazhpoli} for a different formula, which also extends to $E_8$).  If the nonarchimedean multiplicity one results from \cite{savin} were generalized to archimedean fields, it furthermore would give a global multiplicity one theorem for these automorphic minimal representations. Each character $\psi$ in the statement of the theorem is  trivial on the unipotent radical of a proper parabolic subgroup, namely one which contains $N$ and whose Levi component contains the one-parameter subgroup that $\psi$ does not vanish on.  The integration in (\ref{degenwhitint}) then factors over this unipotent radical; by definition, it   vanishes when $F$ is a cusp form.  Therefore we conclude:
\begin{cor}\label{nocuspidalminimalaut}
There are no cuspidal automorphic representations of Chevalley groups of type $E_6$ or $E_7$ which have a minimal local component.  In particular, the Gelfand-Kirillov dimension (which is half the dimension of the wavefront set) of any component of a cuspidal automorphic representation of a Chevalley group of type $E_6$ or $E_7$ must be at least 16 or 26, respectively.
\end{cor}

Piatetski-Shapiro raised the   question as to what the smallest Gelfand-Kirillov dimension can be for a cuspidal automorphic representation of a given group (see \cite{jsli} for results in the symplectic case).   A folklore conjecture asserts that the wavefront set of any component must be the closure of a {\em distinguished} orbit:~one which does not intersect any proper Levi subgroup.  This would replace the lower bounds in the corollary by 21 and 33, respectively.
 A related analysis can be provided for slightly larger representations, but with weaker conclusions.

 One of the main ingredients in our results is the full complex orbit structure of the adjoint action of the  Levi component of a maximal parabolic subgroup on the Lie algebra of its nilradical.  These Levi actions are known as {\em internal Chevalley modules} and were classified completely when the unipotent radical is abelian \cite{rrs} or Heisenberg   \cite{rohrle}.  They are well-known for classical groups through classical rank theory.   For the general exceptional group case, we used computer programs of Littelmann \cite{littelmann} and de Graaf \cite{deG} to arrive at the following:

  \begin{thm}\label{thm:internal}
  Let $\frak g$ be a complex exceptional  simple Lie algebra, $G$ a complex Lie group with Lie algebra $\frak g$, and $P$ a maximal parabolic subgroup with unipotent radical $U$.  Decompose the Lie algebra $\frak u$ of $U$ as a direct sum ${\frak u}=\oplus{\frak u}_i$ of irreducible subspaces for the adjoint action of a Levi component $L$ of $P$.  Then all complex orbits of $L$ on ${\frak u}$ are explicitly listed in the tables of \secref{sec:tables}, along with the adjoint nilpotent orbits of $G$ in $\frak g$ in which they are contained.
  \end{thm}

Section~\ref{sec:characterexpansions} gives a general framework for Fourier expansions, and shows how several known results can be seen as specializations.  In section~\ref{sec:matumotosec} we explain local results of Matumoto and M\oe glin-Waldspurger on the vanishing of Fourier coefficients for certain types of representations.
Theorem~\ref{ournewtheorem} is proven in section~\ref{sec:newthmpf}, along with information from the tables in section~\ref{sec:tables}.  Section~\ref{sec:tables} also contains the list of orbits in theorem~\ref{thm:internal}.

It is a pleasure to acknowledge Manjul Bhargava, Roe Goodman, Dmitry Gourevitch,  Michael Green, Dick Gross, Roger Howe, Joseph Hundley, Erez Lapid, Ross Lawther, Peter Littelmann,  Colette M\oe glin, Gerhard R\"orhle, Gordan Savin, Wilfried Schmid, Takashi Taniguchi, Pierre Vanhove, Jean-Loup Waldspurger, and Nolan Wallach for their valuable conversations.  In particular we would like to thank Peter Littelmann and Pierre Vanhove for their assistance in generating the tables in \secref{sec:tables}.

\section{Character Expansions}\label{sec:characterexpansions}

Let $G_{\text{lin}}$ denote the real points of a connected reductive linear algebraic group defined over $\Q$, and let  $G$ be a finite cover of $G_{\text{lin}}$  which is  a central extension of $G_{\text{lin}}$ by a finite abelian group.
Let $\G\subset G$ be an arithmetic subgroup, meaning that it is commensurate with the inverse image of $G_{\text{lin}}\cap GL(N,\Z)$ in $G$, where $G_{\text{lin}}$ is realized as a subgroup of $GL(N)$ (compatibly with its natural $\Q$-structure).
 Likewise, let $G_\Q$ denote the inverse image of $G_{\text{lin}}(\Q)=G_{\text{lin}}\cap GL(N,\Q)$ in $G$.   All unipotent subgroups of $G_{\text{lin}}$ split over the finite cover and so can be identified with subgroups of $G$.  If $U$ is defined over $\Q$, as we shall now assume, then the arithmetic subgroup $\G$ intersects both $U$  and its derived subgroup $U'=[U,U]$  in cocompact lattices.

 Let ${\mathfrak C}(\Gamma\cap U)$ denote the group of characters of $U$ which are trivial on $\Gamma\cap U$, which we refer to as the integral points of the {\em character variety} of $U$.  Any function $F\in C^\infty(\G\backslash G)$ has the Fourier expansion
\begin{equation}\label{charexp1}
    \Pi_UF(g) \ \ = \ \ \sum_{\chi\,\in\,{\mathfrak C}(\G\cap U)} F_\chi(g)
\end{equation}
generalizing (\ref{fourierseries3}), where $F_\chi(g)$ is defined exactly as in (\ref{fourierseries4})   and $\Pi_U:C^\infty((\G\cap U)\backslash G)\rightarrow C^\infty(((\G\cap U)\cdot U')\backslash G)$ denotes the projection operator
\begin{equation}\label{charexp2}
    \Pi_UF(g) \ \ = \ \ \int_{\G\cap U'\backslash U'}F(ug)\,du\,.
\end{equation}
This is because the Fourier series expansion is valid on functions  on the quotient $U/[U,U]$, the abelianization of $U$.

Let $H$ denote the normalizer of $U$ in $G$,
\begin{equation}\label{charexp3}
    H \ \ = \ \ \{ g \,\in\,G \, | \, gU=Ug \}\,.
\end{equation}
Then $H$ is defined over $\Q$  and $\G\cap H$ is an arithmetic subgroup of $H$. The group $\G\cap H$ acts on $U$ by conjugation, and dually on  ${\mathfrak C}(\G\cap U)$ by
\begin{equation}\label{actiononchars}
h:\chi(u) \ \ \mapsto \ \  \chi(h\i uh)\,.
\end{equation}
 Let ${\mathcal O}(\G\cap H,\G\cap U)$ denote the set of orbits of this action, and ${\mathcal C}(\G\cap H,\G\cap U)$ denote a fixed set of base points for these orbits.  Note that if $\chi$ and $\chi'(u)=\chi(h\i uh)$ are in the same orbit, then
 \begin{equation}\label{charexp4}
 \aligned
    F_{\chi'}(g) \ \ & =  \ \ \int_{\G\cap U\backslash U} F(ug)\,\chi(h\i uh)\i\,du \\
    & = \ \  \int_{\G\cap U\backslash U} F(h uh\i g)\,\chi(u)\i\,du \\
    & = \ \ F_{\chi}(h\i g)\,.
    \endaligned
 \end{equation}
 Here we have used the invariance of $F$ under $h\in \G\cap H$,  as well as the fact that conjugation by an arithmetic subgroup of $H$ leaves the $du$ invariant.
 In particular, this computation shows that $F_\chi(g)$ is automorphic under $(\G\cap H)_\chi$, the stabilizer of $\chi$ in $\G\cap H$.  Furthermore,   the stabilizers
  \begin{equation}\label{Hchistabdef}
    H_\chi \ \ = \ \ \{  \,h\,\in\,H \,\mid \, \chi(u)\,=\,\chi(h\i u h)\,\} \ \ \subset \ \ H
  \end{equation}    of characters $\chi\in {\mathfrak C}(\G\cap U)$    are also defined over $\Q$, and have $(\G\cap H)_\chi$ as arithmetic subgroups.

  We may hence rewrite (\ref{charexp1}) as
  \begin{equation}\label{charexp5}
   \Pi_UF(g) \ \ = \ \ \sum_{\chi\, \in \,{\mathcal C}(\G\cap H,\G\cap U)   } \sum_{\ h\in (\G\cap H)_\chi\backslash (\G\cap H)} F_\chi(hg)\,.
  \end{equation}
  Let ${\mathcal O}(H,\G\cap U)$ denote the  equivalence classes  of characters in  ${\mathcal O}(\G\cap H,\G\cap U)$  under the action (\ref{actiononchars}) of the complexification $H(\C)$ of $H$.  We refer to   ${\mathcal O}(H,\G\cap U)$ as  {\em complex orbits} and ${\mathcal O}(\G\cap H,\G\cap U)$ as  {\em integral orbits}.   A complex orbit groups  integral orbits into characters which have a similar algebraic nature, even though they may not be equivalent under the action of the discrete subgroup $\G\cap H$.  The expression
  \begin{equation}\label{charexp6}
 \Pi_UF(g) \ \ = \ \ \sum_{{\mathfrak o}\,\in\,{\mathcal O}(H,\G\cap U)}
 \sum_{\ \chi\,\in\,{\mathfrak o} \, }
     \sum_{\,h\in (\G\cap H)_\chi\backslash (\G\cap H)} F_\chi(hg)
  \end{equation}
  packages the terms more usefully, because certain  properties of the Fourier coefficients (e.g., vanishing) are often controlled by the complex orbits ${\mathcal O}(H,\G\cap U)$ rather than the individual orbits $\frak o$.  Furthermore, the Fourier coefficients within an orbit ${\frak o} \in {\mathcal O}(H,\G\cap U)$ are sometimes related by an external mechanism, such as the action of Hecke operators.   The complex orbits  ${\mathcal O}(H,\G\cap U)$ have been classified in many cases, and are often finite in number; this is in particular the case in the important example where $U$ is the unipotent radical of a maximal parabolic subgroup of a split Chevalley group.  This classification is well known for classical groups and given in section~\ref{sec:tables} for the five exceptional groups.

The derivation of  formula (\ref{charexp6}) can be iterated further by   using the fact that $F_\chi(hg)$ is an automorphic function on $(\G\cap H)_\chi\backslash H_\chi$.   This gives a further refinement, though it will necessarily  lose some information if nonabelian unipotent subgroups of $H$ are used.
  In the rest of this  section we describe  some important instances of (\ref{charexp6}) and this iterative refinement that have appeared in the literature, in particular covering the rank 2 Chevalley groups.

 \subsection{Example:~The Piatetski-Shapiro/Shalika expansion on $GL(n)$}\label{sec:PSshal}

We begin by stating a motivating geometric fact:~for any field $K$ and integer $m\ge 1$,
\begin{equation}\label{glnorbitfact}
    \text{$GL(m,K)$ acts on $K^m$ with two orbits, namely $\{0\}$ and $K^m-\{0\}$.}
\end{equation}
In more pedestrian terms,  any nonzero vector can be extended to a basis of $K^m$.  As we shall see, this furnishes a rare situation where the set of complex orbits ${\mathcal O}(H,\G\cap U)$ has only two elements, which   makes for an elegant Fourier expansion on $GL(n)$.  The Fourier expansion on a general group often includes similar terms in addition to more complicated ones not present here.

We shall now explain how Piatetski-Shapiro and Shalika derived (\ref{psshalexpn1}) from an iteration of the principle behind (\ref{charexp6}).
Let $G=SL(n,\R)$ and $\G=SL(n,\Z)$.  Let $P=P_n$ be the standard parabolic subgroup of $G$ of type $(n-1,1)$
and $U=U_n\simeq \R^{n-1}$ its unipotent radical
\begin{equation}\label{Udef}
    U \ \ = \ \ \left\{\, u \ = \
\ttwo{I_{n-1}}{\srel{u_1}{\srel{u_2}{\srel{\vdots}{u_{n-1}}}}}{0}{1} \ | \ u_1,\ldots,u_{n-1}\,\in\,\R\,
    \right\}.
\end{equation}
The normalizer $H$ of $U$ coincides with $P$.  Since $P$ factors as $LU$, where $L\simeq GL(n-1,\R)$ is its Levi component, $(\G\cap H)_\chi\backslash (\G\cap H)$ is in bijective correspondence with $L_\chi\backslash L$, where $L_\chi$ is the stabilizer of $\chi$ in $L$.  Since ${\cal C}(\G\cap U)$ is isomorphic to the lattice $\Z^m$,
the orbit statement (\ref{glnorbitfact}) indicates that  the complexification of $L$, $GL(n-1,\C)$, breaks up  ${\mathcal O}(H,\G\cap U)$ into two complex orbits:~the orbit consisting only of the trivial character, and its complement.  The $\Z$-structure in this case is also  easy to work out using the theory of elementary divisors -- $GL(m,\Z)$ acts on $\Z^m$ with orbits of the form $\{(v_1,\ldots,v_m)\in\Z^m|\gcd(v_1,\ldots,v_m)=d\}$, indexed by $d\in\Z_{\ge 0}$.  Putting this together we see that ${\mathcal O}(\G\cap H,\G\cap U)$ has orbit representatives given by the  characters  $\chi_k(u)=e^{2\pi i k u_{n-1}}$ (in terms of the parameterization (\ref{Udef})), one for each $k\in \Z_{\ge 0}$.

The character $\chi_0$ is trivial and $L_{\chi_0}=L$.    However, for $k\neq 0$ the stabilizer of $L_\chi$ is isomorphic to the quotient of the parabolic subgroup $P_{n-2}\subset GL(n-1,\R)$ by its center.
Then (\ref{charexp6}) specializes to
\begin{equation}\label{charexpgln1}
    F(g) \ \ = \ \ \int_{U(\Z)\backslash U(\R)}F(ug)\,du \ + \
    \sum_{k\,=\,1}^\infty \sum_{\g \,\in\,P_{n-2}(\Z)\backslash P_{n-1}(\Z)}V_k(\g g)\,,
\end{equation} where
$P_{n-2}(\Z)$ is embedded in the top left corner of matrices in $P_{n-1}(\Z)$  and  $V_k(g)$ is the period
\begin{equation}\label{charexpgln2}
    V_k(g) \ \ = \ \ \int_{U(\Z)\backslash U(\R)}F(ug)\,\chi_k(u)\i\,du\,.
\end{equation}
Note that the first term on the right hand side of (\ref{charexpgln1}), which equals $V_0(g)$, vanishes by definition when $F$ is {\em cuspidal}.  The second summand on the right hand side of (\ref{charexpgln1}) corresponds to the large orbit of $GL(n-1,\Q)$ on $\Q^n$ from (\ref{glnorbitfact}), which is responsible for the tautologically equivalent phrasing of (\ref{charexpgln1}) in  the adelic terminology originally used by Piatetski-Shapiro \cite{psmult} and Shalika \cite{shalika}.

 Of course $V_k(hg)$ too is an automorphic function,  under $P_{n-2}(\Z)$.  Hence we may repeat this discussion with $n$ replaced by $n-1$  and get a nested expansion, involving sums of translated periods over $N=\{$unit upper triangular matrices$\}\subset G$ of the form
 \begin{equation}\label{charexpgln3}
    \int_{N(\Z)\backslash N(\R)}F\(\left(
                                 \begin{smallmatrix}
                                   1 & x_1 & \star & \star & \star \\
                                   0 & 1 & x_2 & \star & \star \\
                                   0 & 0 & 1 & \ddots & \star \\
                                    \vdots & \vdots & \ddots & \ddots & x_{n-1} \\
                                   0 & 0 & \hdots &0 & 1 \\
                                 \end{smallmatrix}
                               \right)
    g\)e(-k_1x_1-k_2x_2-\cdots-k_{n-1}x_{n-1})\,dn\,
 \end{equation}
 with $k_1,\ldots,k_{n-1} \in \Z_{\ge 0}$.  These latter integrals are Whittaker integrals (in particular, ``degenerate'' Whittaker integrals if some $k_i=0$), and have been widely studied.  Assume now that $F$ is a cusp form, so that the   degenerate Whittaker integrals  vanish automatically.
When $F$ is an eigenfunction of the ring of invariant differential operators and of moderate growth,  a uniqueness principle allows one to write the nondegenerate Whittaker integrals in terms of  multiples of a  special function $W_{(k_1,\ldots,k_{n-1})}(g)$.
  The uniqueness principle further relates these to each other by the formula
\begin{equation}\label{Wksimplerform}
\aligned
    W_{(k_1,\ldots,k_{n-1})}( g) \ \ &  = \ \ W_{(1,1,\ldots,1)}(\D_k\, g) \,, \\
   \text{with~~} \D_k \ \ & = \ \ \(
    \begin{smallmatrix}
 k_1k_2\cdots k_{n-1} &   &  & &   &  \\
  &  k_2\cdots k_{n-1} &   &  & &  \\
  &   &  k_3\cdots k_{n-1} & &  &  \\
  &   &   &  \ddots &  & \\
  &   &   &   & k_{n-1} & \\
  &   &   &   &  & 1 \\
    \end{smallmatrix}
    \)
  \endaligned
\end{equation}
(both sides transform identically under left translation by $N$).  Thus we can write  $V(g)$ from (\ref{psshalexpn2}) as
\begin{equation}\label{psexpansformula}
    V(g) \ \ = \ \ \sum_{k_1,\ldots,k_{n-1}\,\in\,\Z_{> 0}} a_{(k_1,\ldots,k_{n-1})}\,W_{(k_1,\ldots,k_{n-1})}( g)\ , \ \ a_{(k_1,\ldots,k_{n-1})}\,\in\,\C\,,
\end{equation}
a relation   entirely analogous to (\ref{fourierseries2}).  The left translations by $\g\in P_{n-1}(\Z)$  that came from (\ref{charexpgln1}) and its nested descendants result  in  the Piatetski-Shapiro/Shalika expansion  (\ref{psshalexpn1}).

It should be noted that  in the particular case of $n=3$, the iteration stage of the argument becomes much simpler  because $P_{n-2}(\Z)$ contains $N(\Z)$ as a finite index subgroup.  The remaining invariance under the subgroup
 $\left\{\tthree{1}{\star}{0}{0}{1}{0}{0}{0}{1}\right\}$ leads to a Fourier expansion in the starred entry.  This is also the situation for the final step of the iteration for general $n$.  Though the Fourier expansion (\ref{psshalexpn1}) is most frequently used for cusp forms, its analog  for Eisenstein series -- and in particular the degenerate Whittaker coefficients contained therein --  is still important for a number of applications.  See \cite[\S7]{bumpgl3} for complete details of the Fourier expansions for Borel Eisenstein series on $SL(3)$.

\subsection{Example:~Jiang's expansion on $Sp(4)$}

Let us now consider the split Lie group $G=Sp(4,\R)$, defined as
\begin{equation}\label{sp4expn1}
 G \ \ =  \ \ \{\,g \in SL(4,\R) \, \mid \, gJg^t \,=\,J\, \} \, , \ \ \   J \ \ = \ \ \(\begin{smallmatrix} 0 & 0 & 0 & 1 \\ 0 & 0 & 1 & 0 \\ 0 & -1 & 0 & 0 \\ -1 & 0 & 0 & 0
    \end{smallmatrix}\).
\end{equation}  The root system of $G$ has 8 roots.  With respect to the fixed maximal torus
\begin{equation}\label{sp4maxtorus}
    T \ \ = \ \ \left\{ \(\begin{smallmatrix} t_1 & 0 & 0 & 0 \\ 0 & t_2 & 0 & 0 \\ 0 & 0 & t_2\i & 0 \\ 0 & 0 & 0 & t_1\i
    \end{smallmatrix}\) \, | \  t_1, t_2 \,\neq\,0\right\},
\end{equation}
root vectors with respect to the four positive roots are given by
\begin{equation}\label{npluselts}
\gathered
    X_{\a_1} \ \ = \ \ \left[\begin{smallmatrix}0 & 1&0 & 0 \\ 0&0 & 0 &0 \\ 0&0  & 0& -1 \\ 0 &0 &0 &0
\end{smallmatrix}\right] \, , \,
    X_{\a_2} \ \ = \ \ \left[\begin{smallmatrix} 0&0 &0 &0  \\ 0& 0& 1 &0 \\0 & 0 & 0& 0\\  0    &0 & 0&
0 \end{smallmatrix}\right] \, , \,
    X_{\a_1+\a_2} \ \ = \ \ \left[\begin{smallmatrix} 0&0 & 1& 0 \\ 0&0 &  0& 1\\0 &0  &0 & 0\\0  & 0& 0  &0
\end{smallmatrix}\right]\, , \\
\ \ \ \ \qquad\qquad\qquad\qquad\qquad\qquad \ \text{and } \
    X_{2\a_1+\a_2} \ \ = \ \ \left[\begin{smallmatrix} 0&0 &0 & 1 \\ 0& 0& 0 &0 \\ 0& 0 &0 &0 \\ 0 &0 &0 &0
\end{smallmatrix}\right]
\,,
\endgathered
\end{equation}
where $\a_1$ is the short simple positive root and $\a_2$ is the long simple positive root.  A root vector for the negative of any of these roots is given by the transpose  of the corresponding matrix.

Let $P$ denote the standard ``Klingen'' parabolic subgroup of $G$,
\begin{equation}\label{Klingendef}
    P \ \ = \ \  \left\{  \, \(\begin{smallmatrix}
                             \star & \star & \star & \star \\
                             0 & \star & \star & \star \\
                            0 & \star & \star & \star \\
                            0 & 0 & 0 & \star \\
                           \end{smallmatrix}
    \)  \ \in \ G \, \right\}.
\end{equation}
Its unipotent radical $U$ is  a 3-dimensional Heisenberg group, with  Lie algebra $\frak u$   spanned by $X_{\a_1}$, $X_{\a_1+\a_2}$, and $X_{2\a_1+\a_2}$.  Its center is $[U,U]=\{e^{t X_{2\a_1+\a_2}}|t\in\R\}$.  The Levi factor of this parabolic has semisimple part  $SL(2,\R)$, and acts on $U/[U,U]\simeq \R^2$ by the same action   as in the previous $SL(3,\R)$ example.\footnote{In fact, both this and the $SL(3,\R)$ expansion can be proven directly via harmonic analysis of  Heisenberg groups (see \cite[\S3]{voronoi}), without the translation and induction argument just presented.} Hence the projection $\Pi_UF(g)$ has an expansion essentially identical to the one there.  The resulting formula was discovered by Dihua Jiang \cite[Lemma 2.1.1]{jiang}, who used it as an important tool in deriving an integral representation for the degree 16 tensor product $L$-function on $GSp(4)\times GSp(4)$.
A very similar expansion exists for the split  rank 2 group $SO(3,2;\R)$, which is a quotient of $Sp(4,\R)$.

\subsection{Example:~Siegel's expansion on $Sp(4)$}\label{sec:siegelsexpn}

Instead of  the Klingen parabolic (\ref{Klingendef}), consider the standard ``Siegel'' parabolic subgroup of $G=Sp(4,\R)$,
\begin{equation}\label{Siegeldef}
    P \ \ = \ \  \left\{  \, \(\begin{smallmatrix}
                             \star & \star & \star & \star \\
                             \star & \star & \star & \star \\
                            0 & 0 & \star & \star \\
                            0 & 0 & \star & \star \\
                           \end{smallmatrix}
    \)  \ \in \ G \, \right\}.
\end{equation}
Its unipotent radical $U$ is now a 3-dimensional abelian group, whose lie algebra $\frak u$ is spanned by $X_{\a_2}$, $X_{\a_1+\a_2}$, and $X_{2\a_1+\a_2}$.  The Levi component of $P$ is isomorphic to $GL(2,\R)$, and acts by the 3-dimensional symmetric square action, or alternatively the similarity action on the upper right $2\times 2$ matrix block in $U$.  The  complex orbits ${\mathcal O}(H,\G\cap U)$ can be classified in terms of the rank of this block (see section~\ref{sec:Cn}).

The arithmetic structure of the integral orbits in ${\mathcal C}(\G\cap H,\G\cap U)$ is quite subtle, as it is described in terms of ideal classes and class numbers in quadratic fields.  This is the basis for classical expansions of genus 2 holomorphic forms. The generalization of the Siegel parabolic to $Sp(2n,\R)$ or $O(n,n;\R)$ consists of matrices whose lower left $n\times n$ block vanishes.  Since its unipotent radical is abelian, this type of expansion  naturally generalizes to give a Fourier expansion of an automorphic form $F$ on either of these two groups.

\subsection{Expansions on exceptional groups}\label{sec:matumotosec}

Gan-Gross-Savin \cite{ggs} give a theory of Fourier expansions for particular types of automorphic forms on  $G$ = the split real form of $G_2$, namely those  whose archimedean component is a quaternionic discrete series representation.  They consider a maximal parabolic subgroup $P=LU$  whose unipotent radical $U$ is a $5$-dimensional Heisenberg group; the semisimple part of its Levi component $L$ is an $SL(2)$ subgroup determined by a short root.    The restriction to quaternionic discrete series representations is made to apply a uniqueness principle  that  pins down their coefficients as scalar multiples of particular special functions (as in (\ref{fourierseries2})), as well as to avoid coefficients from smaller orbits.  However,  some  aspects of their theory  apply to more general representations.  In this regard it is similar to Siegel's study of holomorphic forms for $Sp(2n,\R)$, in which the coefficients from (\ref{sec:siegelsexpn}) must be augmented by Whittaker coefficients for generic representations.

Brandon Bate \cite{bate} considers the general automorphic form on   $G$, and in particular   a maximal parabolic subgroup determined by a long root.  He finds an explicit version of  (\ref{charexp1}) similar to the Piatetski-Shapiro/Shalika and Jiang expansions, and applies it to obtain the functional equation of the degree 7 $L$-function on $G_2$ (the first explicit functional equation on this group, because the Langlands-Shahidi method only applies to groups whose Dynkin diagram is  part of a larger Dynkin diagram).  Hence all rank 2 groups have an essentially identical piece of their Fourier expansions of the same type:~an average over an embedded $SL(2)$ determined by a long root.

The maximal parabolics of the split real forms of larger exceptional groups have particularly rich structures.    We give a listing of the complex orbits in section~\ref{sec:tables}.  We remark that classifying the integral orbits can be extremely subtle, as it is already in the case of $G_2$ (see, for example, the Fourier expansion in \cite{jiang-rallis}).  Recently Bhargava \cite{bhargava},  Krutelevich \cite{krutelevich}, and Savin-Woodbury \cite{savinwoodbury} have made major progress on some of these group actions.  This subtlety is apparently more striking for groups of uneven root length than it is for the simply laced groups such as $E_6$, $E_7$, and $E_8$, where it is nevertheless very intricate.

\section{Vanishing of coefficients for certain small automorphic representations}

From now until the end of the paper we take $G_{\text{lin}}$ to be the real points of a Chevalley group, defined compatibly with the Chevalley $\Z$-basis.
We shall also suppose that the automorphic form $F \in C^\infty(\Gamma \backslash G)$ is a smooth vector for an automorphic representation.

 After fixing a maximal torus and choice of positive root system for $G$ with respect to this torus, let $\Sigma^+$ denote the positive simple roots.  Let $S$ be an arbitrary subset of $\Sigma^+$ and $P=P(S)$ the standard  parabolic subgroup associated to $S$:~$P$ contains the one-parameter subgroups generated by root vectors $X_{\a}$ of all positive roots, as well as the negative roots such that $-\a \in S$.  It has a Levi decomposition $P=LU$, where  $L$ is a maximal reductive subgroup of $P$ (containing the one-parameter subgroups generated by the root vectors $X_\a$, $X_{-\a}$ of roots $\a \in S$),
    and $U$ is its unipotent radical (containing all  one-parameter subgroups generated by root vectors for positive roots, aside from the ones that are contained in $L$).  Since $\G$ is arithmetic,  $\G\cap [L,L]$ and $\G\cap U$ are  arithmetic subgroups of $[L,L]$ and $U$, respectively.

 The Lie algebra ${\frak u}$ of the unipotent radical $U$ decomposes as the direct sum
\begin{equation}\label{lieV}
    {\frak u} \ \ = \ \ \bigoplus_{i\,=\,1}^m\, {\frak u}_i \,,
\end{equation}
 where ${\frak u}_i$ is the span of the root vectors
for all positive roots $\sum_{\a\in\Sigma^+} n_\a \a$ such that  $\sum_{\a\in S}n_\a = i$.
 In the case that $P$ is a maximal parabolic (that is, $S$ has exactly one element), the adjoint action of $L$ acts irreducibly on each ${\frak u}_i$, actions that are known as  {\em internal Chevalley modules}.  It is furthermore known that each of these actions has finitely many complex orbits.  We shall give an enumeration of these later in section~\ref{sec:tables}.

Let $U\supset U^{(1)}\supset U^{(2)}\supset \cdots$ be the lower central series of $U$, i.e., $U^{(j)}:=[U,U^{(j-1)}]$.
Notice that the Lie algebra of  $U^{(j-1)}$ of $U$ is just  ${\frak u}^{(j-1)}=\oplus_{i\ge j}{\frak u}_i$.
Suppose now that $\l$ is a nontrivial linear functional on some ${\frak u}_j$, which we extend trivially to the rest of ${\frak u}^{(j-1)}$,
 and then exponentiate to  a character $\chi$ of  $U^{(j-1)}$.  We can now consider the Fourier expansion of the type (\ref{charexp6}), but with the subgroup $U$ replaced by $U^{(j-1)}$ and the projection operator $\Pi_U$ from (\ref{charexp2}) replaced by integration over $\G\cap U^{(j)}\backslash U^{(j)}$.   Consider the linear functional on an automorphic representation which maps an automorphic form $F$ to the Fourier coefficient $F_\chi$ defined in  (\ref{fourierseries4}).  This map commutes with the automorphic representation's right translation by the adele group of $G$, and thus gives a global linear functional which is $\chi$-equivariant with respect to $U^{(j-1)}$.  By restriction, it thus gives nonvanishing $\chi$-equivariant local linear functionals for each completion of the ground field.

  For an intricately defined character $\chi$ and a relatively simple automorphic form $F$, the
Fourier coefficient $F_\chi(g)$ from (\ref{fourierseries4}) may vanish identically in $g$; for example, this happens if $F$ is constant, but $\chi$ is not.  The following theorem of Matumoto gives a condition that often ensures this vanishing for all automorphic forms in an automorphic representation with archimedean component $\pi$. It is the archimedean analog of a more well-known theorem of M\oe glin-Waldspurger \cite{Mowa}.  Both results work with the nonvanishing equivariant local linear functionals of the previous paragraph. The
statement involves the complexified wavefront set $WF(\pi)_{\mathbb{C}}$ in
the dual Lie algebra $\mathfrak{g^{\star }}$ to $\mathfrak{g}$, which has
several different definitions:~for example, it is the associated variety of
the annihilator ideal of $\pi $, and it can also be computed in terms of the
support of the Fourier transform of the character of $\pi $. It is always
the closure of a unique coadjoint nilpotent orbit in $\mathfrak{g}$ \cite{joseph,borho}.

\begin{thm}(Matumoto \cite{matumoto})\label{matumoto}
Consider $\l$ as an element of $\frak g^\star$ by  trivially extending it to the rest of ${\frak g}$, and assume that  $\l\notin WF(\pi)_{\mathbb{C}}$.  Then $F_\chi\equiv 0$ for all vectors $F$ in any automorphic representation that has $\pi$ as an archimedean component.
\end{thm}

\noindent The paper \cite{Mowa} by M\oe glin-Waldspurger  contains the same assertion, but with $\pi$ a nonarchimedean representation.\footnote{Though \cite{Mowa} makes a restriction that the field have residual characteristic greater than 2, the authors have informed us that its use on pp.~429 and 431 of their paper can be avoided, and hence the restriction removed.}

\section{Abelian unipotent radicals and small representations}
\label{sec:newthmpf}

Let us now consider a standard maximal parabolic subgroup $P=P(\{\a\})=LU$ with abelian  unipotent radical $U$.  This is the case precisely when no  root has a coefficient of $\a$ greater than one when expanded in terms of the basis of positive simple roots.

According to the tables in \secref{sec:tables}, the action of the complexification $L(\C)$ on $U$ will in general have more than two orbits.  In general the smallest orbit is always the trivial orbit, while the next biggest orbit corresponds to a character which is sensitive to a single root vector in the Lie algebra $\frak u$ of  $U$.   An orbit representative can be furnished by restricting the generic character $\psi$ to $U$; recall $\psi$ is the   character of the unipotent radical $N$ of the minimal parabolic $P(\Sigma^+)$  which satisfies
\begin{equation}\label{psidef}
    \psi(e^{t X_\a}) \ \ = \ \ e^{2\pi i t} \ \ \ \ \text{for all} \ \, \a \,\in\,\Sigma^+\,,
\end{equation}
and which is trivial on the one parameter subgroups generated by all other positive root vectors.  Not only is $\chi$ $L(\C)$-equivalent to $\psi|_U$, but it is furthermore $L(k)$-equivalent if $\chi$ is defined over the number field $k$ and $G$ is simply laced (see \cite{rrs} and \cite[pp.~759-760]{savinwoodbury}).

Now suppose that the archimedean component $\pi$ of an automorphic representation is a minimal representation.  According to Matumoto's theorem~\ref{matumoto} and the tables in section~\ref{sec:tables}, functions $F$ in this automorphic representation will have nonzero Fourier coefficients $F_\chi$ only for $\chi$ in these smallest two orbits  -- i.e., $\chi$ must be trivial, or $L(\C)$-equivalent to $\psi|_U$. As a result, the Fourier expansion for automorphic realizations of minimal representations behaves very similarly to the $GL(n)$ case, with a formula analogous to (\ref{psshalexpn1}).  Since this logic breaks down when $U$ is nonabelian, we  restrict to the cases in theorem~\ref{ournewtheorem} (where this complication does not occur).

\begin{lem}\label{lem:lemma1}
Consider the $E_n$ Dynkin diagrams as numbered in figure~\ref{EnDynkins} and the chain of Levi components of maximal parabolics $P(\{\a_n\})$ of types
\begin{equation}\label{lemmachain}
   E_7 \ \ \supset \ \  E_6 \ \ \supset  \ \ D_5 \ \ \supset \ \  A_4 \ \ \supset  \ \ A_2  \times A_1 \ \ \supset \ \ A_1\times A_1 \ \ \supset \ \ A_1
\end{equation}
formed by successively deleting the highest numbered node.  The unipotent radicals $U$ of each of these parabolics is abelian.  Decompose the Lie algebra $\frak n$ of $N$ as the direct sum of $\frak n'$ and the Lie algebra $\frak u$ of $U$, where $\frak n'$ is spanned by root vectors $X_\a$ of positive roots whose coefficient of $\a_n$ is zero.  Then if $X$ is a nonzero element of $\frak n'$, the sum $X+X_{\a_n}$ cannot lie in the minimal adjoint nilpotent orbit $\cal O$ of $\frak g$.
\end{lem}
\begin{proof}
The unipotent radicals are abelian because the highest root of each of these root systems has coefficient 1 of the last simple root.  For the second statement, we note that $X_{\a_n}$ is an element of $\cal O$.  In the minimal (56-dimensional) representation of $\frak e_7$, $X_{\a_7}$ is a matrix which squares to zero, and therefore $p(x)=x^2$ is the minimal polynomial of any element of $\cal O$ in this representation.  However, when one writes $X=\sum c_\a X_\a$ as a linear combination of the 36 root vectors $X_\a\in {\frak n}'$ (corresponding to the positive roots of the embedded $\frak e_6$ in (\ref{lemmachain})), the condition that $p(X+X_{\a_n})=0$ forces each of the 36 coefficients $c_\a$ to vanish.  This proves the lemma for $n=7$.  The analogous argument applies to the case of $n=6$ in its 27-dimensional representation, but with a linear combination of 20 root vectors (corresponding to the positive roots of the embedded $\frak{so}(5,5)$).

The case of $n=5$ is slightly different:~in the standard 10-dimensional representation of $\frak{so}(5,5)$ the above argument does not rule out $c_\a\neq 0$ for the roots $\a=\a_1$, $\a_2$, or $\a_1+\a_2$.  However, the matrices for $c_{\a_1}X_{\a_1}+c_{\a_2}X_{\a_2}+c_{\a_1+\a_2}X_{\a_1+\a_2}+X_{\a_5}$ have rank at least 4 unless $c_{\a_1}=c_{\a_2}=c_{\a_1+\a_2}=0$.  The matrix for $X_{\a_5}$ has rank 2 and hence so must all elements of $\cal O$, proving the lemma for $n=5$.  The cases of smaller $n$ can be handled directly in terms of Jordan canonical form.
\end{proof}

\begin{proof}[Proof of theorem~\ref{ournewtheorem}]
Let $G$ be a split Chevalley group of type $E_7$ and let $P=P(\{\a_7\})$ denote the standard maximal parabolic subgroup of $G$ associated with the last node (in the numbering of figure~\ref{EnDynkins}).  Its unipotent radical is 27 dimensional, with four complex character variety orbits of dimensions 0, 17, 26, and 27 (see section~\ref{sec:E7node7}).  According to the table there, the 26- and 27-dimensional orbits lie in coadjoint nilpotent orbits strictly larger than the minimal coadjoint nilpotent orbit in the usual closure ordering.  Hence Matumoto's   and M\oe glin-Waldspurger's theorems imply that $F$ can be written as a sum (\ref{fourierseries3}) in which
the only characters which contribute are either
trivial  or $L(k)$-equivalent to the restriction of $\psi$ to $U$.  Since the adelic automorphic form $F$ is left invariant under $G(k)\supset L(k)$, in considering these Fourier coefficients we may furthermore assume that either $\chi$ is trivial  or equal to $\psi|_U$.

We now separately examine these two types of contributions.  First, assume that $\chi=\psi|_U$.  The coefficients $F_\chi$ are automorphic under the stabilizer of $\chi$ inside the Levi component $L$.  This stabilizer is a maximal $k$-parabolic of $L$ associated with node 6 in the $E_6$ Dynkin diagram, and contains $N\cap L$.  In particular it has a 16-dimensional unipotent radical $V$ in which we may take a Fourier expansion of $F_\chi$.    Together $V$ and $U$ generate the unipotent radical of the nonmaximal parabolic subgroup  $P(\{\a_6,\a_7\})$ in $G$.
However,  lemma~\ref{lem:lemma1} and Matumoto's   and M\oe glin-Waldspurger's theorems imply no nontrivial characters on $V$ can contribute to this expansion (essentially because the minimal orbit has already been ``used up'' by the nontrivial character $\chi$ on $U$).  Thus $F_\chi$ is trivial under left translation by $V$,  and is consequently automorphic on the Levi component of the stabilizer, of type $D_5$.

In the case that $\chi$  trivial, $F_\chi$ is  automorphic on $\G\cap L$ by dint of the fact that the stabilizer of the trivial character is the full group.  Thus it is automorphic on a Chevalley group of type $E_6$.  An automorphic function on $E_6$, be it $F_\chi$ or  an automorphic form for a Chevalley group of type $E_6$, can then be expanded in the $P(\{\a_6\})$ parabolic as above.  Thus in all cases we can look at a further Fourier expansion of an automorphic function on a smaller group in the chain (\ref{lemmachain}).
Proceeding downwards and using the fact that the unipotent radicals in lemma~\ref{lem:lemma1} are always abelian, we see that $F$ is a sum of   translates of Fourier coefficients of the form (\ref{degenwhitint}).

\end{proof}

\section{Orbit structure of internal Chevalley modules}\label{sec:tables}

In this section, we list the full complex orbit structure for all internal Chevalley modules of maximal parabolic subgroups.  Recall that these are the actions of the Levi component on the individual graded pieces $\frak u_i$ from (\ref{lieV}).
  We also give examples for some low rank classical groups, noting  that those with rank $\le 2$ have been discussed earlier in section~\ref{sec:characterexpansions}.  Papers \cite{rohrle,rohrleinternal,rrs,MR1219660} give a complete discussion in a number of important cases; see also \cite{djokovic,igusa,haris,popov} for some historically important examples.  Peter Littelmann's computer software\footnote{We are grateful to Peter Littelmann and Pierre Vanhove for helping us to get this software to run in modern computing environments.} \cite{littelmann} computes the orbits in first graded piece $\frak u_1$.  Because of the observation (stated precisely in each case below) that each higher graded piece $\frak u_i$, $i>1$, occurs as the first graded piece of another internal Chevalley module, the software thus handles all cases over $\C$.  This observation was previously used in \cite{shahidi} as part of an induction that establishes analytic properties in the Langlands-Shahidi method.

  A few comments are in order about covers.  First of all, the orbit structure of internal Chevalley modules is unaffected by taking a central extension:~this is because the center acts trivially on the Lie algebra under the adjoint action.  Therefore in working out the examples for the Lie algebras of various types below, it is sufficient to calculate with a particular semisimple Lie group having that Lie algebra.
  Furthermore,  the action of the Levi $L$ is essentially pinned down by that of $[L,L]$, since the action of the center of $L$ on each $\frak u_i$ can be easily described in terms of the structure of the root system.  The action of $[L,L]$ can itself be identified using the Weyl character formula, which is slightly more difficult but still straightforward.

  Though the tables here compute the orbits on the graded pieces $\frak u_i$ only for $i>0$, the orbits for $\frak u_{-i}$ are related using the Cartan involution.  In particular, the character variety $\frak u_{-1}$'s orbits are identical to those of $\frak u_1$.

For the sake of compact notation, we often say that a subalgebra of a Lie algebra ``contains a root'' $\a$ when it contains a root vector $X_\a$ for $\a$; likewise, we apply this same terminology to a subgroup that contains the one-parameter subgroup generated by $X_\a$.  We shall also sometimes write a nonsimple root by stringing together its coefficients when expanded a sum of the positive simple roots (for example, the root $\a_1+2\a_2+\a_3+\a_4$ of $D_4$ could be more concisely written as 1211).  Furthermore, we will indicate a basepoint of an orbit is a sum of root vectors by formally adding these abbreviated labels of the respective roots.  (In each  case, linear combinations with nontrivial coefficients of these root vectors gives a basepoint of the same orbit, so omitting coefficients is harmless.)  We also write the basepoint of the trivial orbit as $00\cdots00$. In the tables below we list basepoints for each orbit, their dimensions, and the adjoint nilpotent orbit they are contained in.  Since these are isomorphic to coadjoint nilpotent orbits, we will label its column as such (in deference to the commonly recognized terminology).  Such orbits will be listed in terms of their marked Dynkin diagrams, which are also strings of nonnegative integers.

Before giving the list of orbits we shall first make some remarks pertinent to the classical group cases.
As we remarked in section~\ref{sec:PSshal}, the orbit structure for groups of type $A_n$ crucially depends on the fact that any linearly independent set of vectors can be extended to a basis.  We now describe the  analogous vector space statements needed for the other classical groups.  These are somewhat more complicated and involve  Witt's theorem (which we  recall below).  We shall keep the statements here flexible enough to cover the case of type $A_n$, which we will be able to study directly using elementary linear algebra.

Suppose that $V$ is a complex {\em  bilinear} space, that is, a  vector space equipped with a symmetric or skew-symmetric bilinear form
$\left\langle .,.\right\rangle _{V}$.  Let
$r_{V}$ and $n_{V}$ denote the rank and nullity of $\left\langle .,.\right\rangle _{V}$, respectively, and let   $\epsilon_{_{V}}=1,-1,0$ according as $\left\langle .,.\right\rangle_{V}$ is symmetric, skew-symmetric, or both (i.e., identically $0$).  These three cases correspond to groups of type $B_n$ or $D_n$, groups of type $C_n$, and groups of type $A_n$.

Two complex bilinear spaces $V$ and $W$ are {\em isometric} when there exists  an {\em isometry} between them, i.e., an invertible  linear map $s:V\rightarrow W$ satisfying $\left\langle sv_{1},sv_{2}\right\rangle _{W}=\left\langle v_{1},v_{2}%
\right\rangle _{V}$ for all $v_{1},v_{2}\in V$.
Witt's theorem \cite[Theorem~3.9]{artin} asserts that
this is the case precisely when  $(r_V,n_V,\e_{_V})=(r_W,n_W,\e_{_{W}})$.  Moreover, when this condition holds and furthermore the nullities $n_V=n_W=0$, any isometry between subspaces of $V$ and $W$ extends to one between all of  $V$ and $W$.

The set of isometries from $V$ to
itself constitutes its isometry group $I\left(  V\right)  \subset GL\left(
V\right)$.   Witt's theorem implies that two linear maps $a,b$ from a vector space $X$ into a bilinear space $V$ have isometric images if and only if $b$ can be written as $s\i a c$ for some $c\in GL\left(  X\right)  $ and $s\in I\left(  V\right)$.

We now record the translations of these statements in terms of matrices, via the bilinear form on $\C^n$ defined from an $n\times n$ matrix by the formula $\left\langle v_{1},v_{2}\right\rangle =v_{1}^{t}%
Mv_{2}$.  This form is symmetric or skew-symmetric according as $M$ is symmetric or skew-symmetric,
and its rank and nullity are the rank and nullity of $M$, respectively.
An isometry between bilinear spaces corresponding to $n\times n$ matrices $M_{1},M_{2}$ is a matrix $S\in
GL(n,\C)$ such that $S^{t}M_{1}S=M_{2}$.
  The isometry group
for the bilinear space is%
\[
I\left(  M\right)  \ \  = \ \ \left\{  S\in GL(n,\C) \, \mid \,  S^{t}MS=M\right\} ,
\]
and the rank of the column space of an   $n\times m$ matrix  $A$  is
$\operatorname{rank}\left(  A^{t}MA\right)$.
The consequences of Witt's theorem mentioned above can be restated as follows:

\begin{lem}
\label{M12}
\enumerate
\item[(1)]  Let $M_{1}$ and $M_{2}$ be two complex $n\times n$  matrices, both
symmetric or both skew-symmetric. Then
$\operatorname{rank}\left(  M_{1}\right)  =\operatorname{rank}\left(  M_{2}\right)
$ if and only if
 there exists a matrix $S\in GL(n,\C)$ such that
$S^{t}M_{1}S=M_{2}$.
\item[(2)]
Let $M$ be a complex $m\times m$ nonsingular symmetric or skew-symmetric  matrix, and let $A,B$
be  complex $m\times n$  matrices. Then
\[
\operatorname{rank}\left(  A\right)  =\operatorname{rank}\left(  B\right)  \text{  ~~ \emph{and}~~ }\operatorname{rank}\left(
A^{t}MA\right)  =\operatorname{rank}\left(  B^{t}MB\right)
\]
if and only if there exist $C\in GL(n,\C)$ and $S\in I\left(
M\right)  $ such that $B=S^{-1}AC$.
\end{lem}

\subsection{Type $A_n$ : $SL(n+1)$}\label{sec:An}

Owing to the Dynkin diagram symmetry, there are essentially $\lceil\f n2\rceil$ cases here.   A standard maximal parabolic subgroup is block upper triangular according to a decomposition $n_1+n_2=n+1$, with Levi component $L$ isomorphic to the subgroup of $GL(n_1)\times GL(n_2)$ defined by $\{(g,g') \in GL(n_1)\times GL(n_2)|\det(g)\det(g')=1\}$.  The unipotent radical is abelian in each case, so ${\frak u}={\frak u}_1$ is isomorphic to $n_1 \times n_2$ matrices. The action of $L$ on $\frak u$ is the tensor product action of the standard representation of $GL(n_1)$ on  $n_1$-dimensional vectors, with the contragredient representation of $GL(n_2)$ on $n_2$-dimensional column vectors.

This action has been classically studied  and has $\min(n_1,n_2)+1$ complex orbits, classified by the rank of the  $n_1 \times n_2$  matrix.  Indeed, the general orbit classification for unipotent radicals in classical groups goes by the name of ``classical rank theory'' because of its similarity to this prototypical case.  Representatives, accordingly, are given by matrices which are zero except for an $r\times r$ identity matrix in, say, their top right hand corner, $0 \le r \le \min(n_1,n_2)$.

Besides \thmref{matumoto}, a very complete analysis of vanishing Fourier  coefficients for representations of $GL(n)$ is given in \cite{Gurevich-Sahi}.

\subsection{Type $B_n$ : $SO(n+1,n)$}\label{sec:Bn}

Consider
the standard parabolic subgroup $P_{\a_k}$ of $B_n$, where $1\le k \le n$.
In this case the Levi factor $L$ has $[L,L]$ of type $A_{k-1}\times B_{n-k}$.  The unipotent radical is abelian when $k=1$, but is otherwise a two-step nilpotent group.  Here the action of $L$ on $\frak u_1$ is the tensor product of the standard (vector) actions of $SL(k)$ and $SO(n-k+1,n-k)$.  The action on $\frak u_2$ is the exterior square action of $SL(k)$ on antisymmetric $n$-tensors and arises for $D_k$ node $k$; its orbits are described in  section~\ref{sec:Dnnodebig}.

The action on $\frak u_1$ can be thought of as $SL(k)\times SO(n-k+1,n-k)$ acting on $k\times (2n-2k+1)$ matrices $A$.  Here $\C^{2n-2k+1}$ is viewed as a bilinear space, equipped with a symmetric bilinear form given by a $(2n-2k+1)\times (2n-2k+1)$ matrix $J$ that defines the orthogonal group. According to lemma~\ref{M12}, orbits of this action have a fixed value of the rank of $A$, and a fixed value of the rank of $A^t JA$.  Because $\ker\left(  A^{t}JA\right)  \supset\ker\left(  A\right)$, the ranks must also satisfy the inequality
\begin{equation}\label{Bnineq1}
    \operatorname{rank}(A^t J A) \ \ \le  \ \  \operatorname{rank}(A)\,.
\end{equation}
Furthermore,
\begin{equation}\label{Bnineq2}
    2\, \operatorname{rank}(A) \ - \ (2n-2k+1) \ \ \le \ \   \operatorname{rank}(A^t J A) \,.
\end{equation}
To see this, let $X=\ker(A^t JA)$ and consider the subspace $AX\subset \C^{2n-2k+1}$, which is the radical of the image of $A$ when thought of as a bilinear subspace.  The orthogonal complement of $AX$ in $\C^{2n-2k+1}$ contains the image of $A$.  Because $J$ is nondegenerate, $\C^{2n-2k+1}$ is the direct sum of $AX$ and its orthogonal complement, which gives the inequality  $\dim(AX)\le 2n-2k+1-\operatorname{rank}(A)$.  At the same time, $\dim(AX)=\operatorname{rank}(A)-\operatorname{rank}(A^t JA)$ because of the ``rank plus nullity equals dimension'' formula for the image of $A$.  Inequality (\ref{Bnineq2}) immediately follows.

Finally,
\begin{equation}\label{Bnineq3}
     \operatorname{rank}(A^t J A)\,, \  \operatorname{rank}(A) \ \ \le \ \ \min\{k,2n-2k+1\}
\end{equation}
because the rank of a matrix is always bounded by its row and column size.

The complex orbits on $\frak u_1$ not only satisfy (\ref{Bnineq1})-(\ref{Bnineq3}), but exist for each possibility.  This can be  directly seen by listing $k\times(2n-2k+1)$ matrix representatives of the form
\begin{equation}\label{Bnrepresenative}
  A \ \ = \ \ A_{s,p} \ \ = \ \   \(
    \begin{matrix}
    w_1+iw_2 \\
    w_3+iw_4 \\
    \vdots\\
    w_{2s-1}+iw_s \\
    w_{2s+1}\\
    w_{2s+2}\\
    \vdots\\
    w_{2s+p}\\
    0\\
    0\\
    \vdots\\
    0
    \end{matrix}
    \),
\end{equation}
where $\{w_1,\ldots,w_{2n-2k+1}\}$ is a basis of $\C^n$ satisfying $w_j^t J w_\ell=\d_{j=\ell}$.
Indeed, $A_{s,p}$ has rank $p+s$ and $(A_{s,p})^t J A_{s,p}$ has rank $p$.
Part (2) of Lemma~\ref{M12} states that the
  matrices $A_{s,p}$, where $0\le p\le \min\{k,2n-2k+1\}$, $0\le p+s\le \min\{k,2n-2k+1\}$, $2s+p\le k$, and $2(p+s)-(2n-2k+1)\le p$, furnish a complete set of orbit representatives for the action of $GL(k,\C)\times O(n-k+1,n-k)$.  A short computation with stabilizers shows that they are furthermore a complete set of  basepoints for the complex orbits of $L$ on $\frak u_1$.

\subsection{Type $C_n:Sp(2n)$}\label{sec:Cn}

Because most of the statements in section~\ref{sec:Bn} did not distinguish between symmetric and skew-symmetric bilinear forms,
Witt's theorem again applies in a very similar way.  The main difference is a distinction between the standard maximal parabolic subgroups $P_{\a_k}$ for $k<n$  and $k=n$.

\subsubsection{$C_n$ node $k<n$}\label{sec:Cnnodenotn}

Here the semisimple part of $L$ has type $A_{k-1}\times C_{n-k}$ and
the unipotent radical  $U$ is a two-step nilpotent group with $\frak u=\frak u_1\oplus \frak u_2$.
 The action on $\frak u_1$ is   the tensor product of the standard actions of $SL(k)$ and $Sp(2n-2k)$, while the action on $\frak u_2$ is the symmetric square action of $SL(k)$ on symmetric $k$-tensors. The latter occurs for $C_k$ node $k$ and is described in section~\ref{sec:Cnnoden}.

Witt's theorem applies here nearly exactly as it does for $B_n$ in  section~\ref{sec:Bn}.  The main difference is that the rank of $A^t JA$ must be even because $J$ is now skew-symmetric.  Aside from this,  the conditions (\ref{Bnineq1})-(\ref{Bnineq3}) stand after replacing $2n-2k+1$ with $2n-2k$.
Fix a basis $\{w_1,\ldots,w_{n-k},z_1,\ldots,z_{n-k}\}$ of $\C^{2n-2k}$ satisfying
\begin{equation}\label{Cnbasis}
\left\langle x_{i},y_{j}\right\rangle  \ \ = \ \ \delta_{ij}\,,\text{\quad\ }\left\langle
x_{i},x_{j}\right\rangle \ \  = \ \ \left\langle y_{i},y_{j}\right\rangle \  \ = \ \ 0\,.
\end{equation}
Let $A_{p,s}$ denote the $k\times (2n-2k)$ matrix whose first $2p+s$ rows are $x_1,\ldots,x_p$, $y_1,\ldots,y_{p+s}$, and the rest all zeros.  Then $A_{p,s}$ has rank $2p+s$ while $(A_{p,s})^t JA_{p,s}$ has rank $2p$.  This shows that each rank inequality is met.  Another application of lemma~\ref{M12}
and a short stabilizer computation show that
the matrices $A_{p,s}$ for $p,q$ satisfying $0\le 2p \le   \min\{k,2n-2k\}$,  $0\le 2p+s \le   \min\{k,2n-2k\}$, and  $p+s\le n-k$ then furnish a complete set of orbit representatives for $L(\C)$ on $\frak u_1$.

\subsubsection{$C_n$ node $n$}\label{sec:Cnnoden}

Here $[L,L]$ is $SL(n)$ and $U$ is abelian of dimension $\f{n(n+1)}{2}$.  The action on $\frak u$ is the symmetric square action of $SL(n)$ on symmetric $n$-tensors.  This can be naturally viewed as the action of $g\in SL(n)$ on symmetric $n\times n$ matrices $X$ given by $g:X\mapsto gXg^t$.
  Because symmetric matrices can be orthogonally diagonalized, this action has $n+1$ orbits, each of which is represented by a  matrix which has zero entries except for precisely $k$ ones on its diagonal, where $k$ ranges from $0$ to $n$ (see lemma~\ref{M12}).

\subsection{Type $D_4$ : $SO(4,4)$}\label{sec:D4}

Before giving the general theory, we give some detailed examples for special cases that were important in  \cite{GMV} and in section~\ref{sec:newthmpf}.

\begin{center}
\begin{figure}
  \includegraphics[scale=.7]{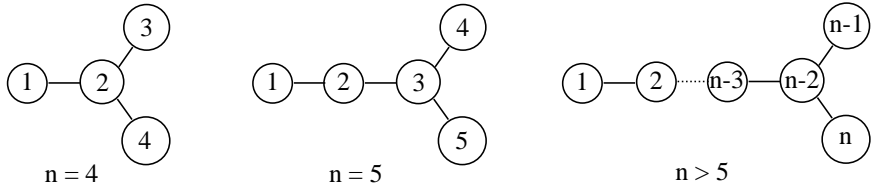}\\
  \caption{$D_n$ Dynkin diagrams}\label{DnDynkins}
\end{figure}
\end{center}

The triality makes $D_4$ exceptional among classical groups.  Let us first consider a maximal parabolic $P=LU$ associated to one of its 3 terminal nodes (numbered 1, 3, or 4 in figure~\ref{DnDynkins}).  Then $L$ is of type $A_3$ and $U$ is 6 dimensional.  Because the largest root of the $D_4$ root system is $\a_1+2\a_2+\a_3+\a_4$, $U$ is abelian and $\frak u=\frak u_1$.  The action of  $L$  on $\frak u$ is the 6-dimensional representation of $SL(4)$ on anti-symmetric tensors.  The general theory of this action is well understood and described in comments below in section~\ref{sec:Dnnodebig}.  The action of $L$ carves the 6-dimensional vector space into 3 orbits:~a zero orbit, one whose basepoint is a root vector for any positive simple root in $\frak u$, and an open dense one whose basepoint is a sum of root vectors for $\a_1+\a_2+\a_4$ and $\a_1+\a_2+\a_3$.

Next, let us consider $P=LU$ associated to the central node (numbered 2).  Then the semisimple part of $L$ has type $SL(2,\R)\times SL(2,\R)\times SL(2,\R)$  and $U$ is 9 dimensional, breaking up as ${\frak u}={\frak u}_1\oplus {\frak u}_2$, where $\frak u_1$ is the 8-dimensional triple tensor product representation of $SL(2,\R)\times SL(2,\R)\times SL(2,\R)$  and $\frak u_2$ is the one dimensional span of $X_{1211}$.   The action on ${\frak u}_1$ breaks up into 7 orbits under $L(\C)$:

\begin{center}\begin{tabular}{||p{3.81cm}p{2.41cm}||p{2.78cm}||}
\hline
 \multicolumn{1}{|p{3.81cm}|}{\centering Orbit Basepoint} &  \multicolumn{1}{p{2.41cm}|}{\centering Dimension} &  \multicolumn{1}{p{2.78cm}|}{\centering Coadjoint orbit intersected} \\
\hline
 \multicolumn{1}{|p{3.81cm}|}{\centering 0101+1110} &  \multicolumn{1}{p{2.41cm}|}{\centering 8} &  \multicolumn{1}{p{2.78cm}|}{\centering 0200} \\
 \cline{1-1} \cline{2-2} \cline{3-3}
 \multicolumn{1}{|p{3.81cm}|}{\centering 0111+1101+1110} &  \multicolumn{1}{p{2.41cm}|}{\centering 7} &  \multicolumn{1}{p{2.78cm}|}{\centering 1011} \\
\hline
 \multicolumn{1}{|p{3.81cm}|}{\centering 0111+1101} &  \multicolumn{1}{p{2.41cm}|}{\centering 5} &  \multicolumn{1}{p{2.78cm}|}{\centering 0002} \\
 \cline{1-1} \cline{2-2} \cline{3-3}
 \multicolumn{1}{|p{3.81cm}|}{\centering 1101+1110} &  \multicolumn{1}{p{2.41cm}|}{\centering 5} &  \multicolumn{1}{p{2.78cm}|}{\centering 2000} \\
 \cline{1-1} \cline{2-2} \cline{3-3}
 \multicolumn{1}{|p{3.81cm}|}{\centering 0111+1110} &  \multicolumn{1}{p{2.41cm}|}{\centering 5} &  \multicolumn{1}{p{2.78cm}|}{\centering 0020} \\
 \cline{1-1} \cline{2-2} \cline{3-3}
 \multicolumn{1}{|p{3.81cm}|}{\centering 1111} &  \multicolumn{1}{p{2.41cm}|}{\centering 4} &  \multicolumn{1}{p{2.78cm}|}{\centering 0100} \\
 \cline{1-1} \cline{2-2} \cline{3-3}
 \multicolumn{1}{|p{3.81cm}|}{\centering 0000} &  \multicolumn{1}{p{2.41cm}|}{\centering 0} &  \multicolumn{1}{p{2.78cm}|}{\centering 0000} \\
\hline
\end{tabular}\end{center}%
Since $\frak u_2$ is a line, $L$ acts by multiplication by a scalar on it.  Hence there are two orbits:~zero and nonzero elements.

\subsection{Type $D_5$ : $SO(5,5)$}

There are essentially 4 different cases here, owing to the fact that the two spinor nodes 4 and 5 are interchanged by a diagram automorphism.  The behavior for the maximal parabolics associated to the terminal nodes 1, 4, and 5 is similar to that described for $D_4$, but the internal Chevalley modules for nodes 2 and 3 are new.   The actions for all nodes are listed in the following table.  The   actions are described by the
 highest weight of the representation of the semisimple group $[L,L]$. For uniformity, the numbering of the fundamental weights $\varpi_1, \varpi_2, \ldots$ corresponds  the $D_5$ Dynkin diagram, not a standard numbering scheme of the individual Levi factors.  We use the same convention for higher rank groups whenever such an ambiguity arises.

\begin{center}\begin{tabular}{|c|c|c|c|}
\hline
 Node &  Type of $[L,L]$ &  $i \, = \, 1$ & $i \, = \, 2$ \\
\hline
 1 &  \multicolumn{1}{l|}{$SO(4,4)$} &  Spin representation &  \\
  &  $\dim \frak u_i$ &  8 &  \\
  &  action &  $\varpi_2$ &  \\
\hline
 2 &  \multicolumn{1}{l|}{$SL(2)\times SO(3,3)$} &  Standard $\otimes$ Standard & Trivial \\
  &  $\dim \frak u_i$ &  12 & 1 \\
  &  action &  $\varpi_1+\varpi_3$ &  \\
\hline
 3 &  \multicolumn{1}{l|}{$SL(3)\times SL(2)\times SL(2)$} &  Tensor product & Standard \\
  &  $\dim \frak u_i$ &  12 & 3 \\
  &  action &  $\varpi_1+\varpi_4+\varpi_5$ & $\varpi_2$ \\
\hline
 4 &  \multicolumn{1}{l|}{$SL(5)$} &  Exterior square &  \\
  &  $\dim \frak u_i$ &  10 &  \\
  &  action &  $\varpi_2$ &  \\
\hline
 5 &  \multicolumn{1}{l|}{$SL(5)$} &  Exterior square &  \\
  &  $\dim \frak u_i$ &  10 &  \\
  &  action &  $\varpi_2$ &  \\
\hline
\end{tabular}\end{center}

\subsubsection{$D_5$ node 1}\label{sec:D5node1}

The parabolic $P=LU$ associated to node 1 has $[L,L]$ of type $SO(4,4)$, and $U$ an 8-dimensional abelian group.   The action on $\frak u = {\frak u}_1$ is the 8-dimensional vector representation  and has 3 distinct complex orbits:
\begin{center}\begin{tabular}{|cp{1.97cm}||p{2.53cm}||}
\hline
 \multicolumn{1}{|c|}{Orbit Basepoint} &  \multicolumn{1}{p{1.97cm}|}{\centering Dimension} &  \multicolumn{1}{p{2.53cm}|}{\centering Coadjoint orbit intersected} \\
\hline
 \multicolumn{1}{|c|}{11101+11110} &  \multicolumn{1}{p{1.97cm}|}{\centering 8} &  \multicolumn{1}{p{2.53cm}|}{\centering 20000} \\
 \hline
 \multicolumn{1}{|c|}{12211} &  \multicolumn{1}{p{1.97cm}|}{\centering 7} &  \multicolumn{1}{p{2.53cm}|}{\centering 01000} \\
 \hline
 \multicolumn{1}{|c|}{00000} &  \multicolumn{1}{p{1.97cm}|}{\centering 0} &  \multicolumn{1}{p{2.53cm}|}{\centering 00000} \\
\hline
\end{tabular}\end{center}

\subsubsection{$D_5$ node 2}\label{sec:D5node2}

The parabolic $P=LU$ associated to node 2 has $[L,L]$ of type $SL(2)\times SL(4)$, and $U$ a 13-dimensional Heisenberg group.  Thus $\frak u={\frak u}_1\oplus {\frak u}_2$ where $\frak u_1$ is 12 dimensional and $\frak u_2$ is 1 dimensional.  The action on ${\frak u}_1$ is the tensor product of the standard representation of $SL(2)$ with the exterior square representation of $SL(4)$ (that occurred in section~\ref{sec:D4}).  It breaks up into 6 distinct complex orbits:

\begin{center}\begin{tabular}{||p{3.72cm}p{2.09cm}||p{2.5cm}||}
\hline
 \multicolumn{1}{|p{3.72cm}|}{\centering Orbit Basepoint} &  \multicolumn{1}{p{2.09cm}|}{\centering Dimension} &  \multicolumn{1}{p{2.5cm}|}{\centering Coadjoint orbit intersected} \\
\hline
 \multicolumn{1}{|p{3.72cm}|}{\centering 01101+11110} &  \multicolumn{1}{p{2.09cm}|}{\centering 12} &  \multicolumn{1}{p{2.5cm}|}{\centering 02000} \\
 \cline{1-1} \cline{2-2} \cline{3-3}
 \multicolumn{1}{|p{3.72cm}|}{\centering 01211+11101+11110} &  \multicolumn{1}{p{2.09cm}|}{\centering 11} &  \multicolumn{1}{p{2.5cm}|}{\centering 10100} \\
  \cline{1-1} \cline{2-2} \cline{3-3}
 \multicolumn{1}{|p{3.72cm}|}{\centering 01211+11111} &  \multicolumn{1}{p{2.09cm}|}{\centering 9} &  \multicolumn{1}{p{2.5cm}|}{\centering 00011} \\
 \cline{1-1} \cline{2-2} \cline{3-3}
 \multicolumn{1}{|p{3.72cm}|}{\centering 11101+11110} &  \multicolumn{1}{p{2.09cm}|}{\centering 7} &  \multicolumn{1}{p{2.5cm}|}{\centering 20000} \\
\cline{1-1} \cline{2-2} \cline{3-3}
 \multicolumn{1}{|p{3.72cm}|}{\centering 11211} &  \multicolumn{1}{p{2.09cm}|}{\centering 6} &  \multicolumn{1}{p{2.5cm}|}{\centering 01000} \\
 \cline{1-1} \cline{2-2} \cline{3-3}
 \multicolumn{1}{|p{3.72cm}|}{\centering 00000} &  \multicolumn{1}{p{2.09cm}|}{\centering 0} &  \multicolumn{1}{p{2.5cm}|}{\centering 00000} \\
\hline
\end{tabular}\end{center}

\noindent As before, $L$ acts on the 1 dimensional subspace $\frak u_2$ with two orbits:~zero and nonzero.

\subsubsection{$D_5$ node 3}\label{sec:D5node3}

The parabolic $P=LU$ associated to node 3 has $[L,L]$ of type $SL(3) \times SL(2) \times SL(2)$, while $U$ is a 15-dimensional two-step nilpotent group such that $\frak u = \frak u_1 \oplus \frak u_2$, with $\frak u_1$ the 12 dimensional representation of $[L,L]$ from the tensor product action of the three Levi factors, and ${\frak u}_2$ a 3 dimensional vector space with a standard $SL(3)$ action.  There are 9 complex orbits of   $L$ on ${\frak u}_1$:
\begin{center}\begin{tabular}{p{4.59cm}p{1.75cm}p{2.38cm}}
 \cline{1-1} \cline{2-2} \cline{3-3}
 \multicolumn{1}{|p{4.59cm}|}{\centering Orbit Basepoint} &  \multicolumn{1}{p{1.75cm}|}{\centering Dimension} &  \multicolumn{1}{p{2.38cm}|}{\centering Coadjoint orbit intersected} \\
\cline{1-1} \cline{2-2} \cline{3-3}
 \multicolumn{1}{|p{4.59cm}|}{\centering 00111+01101+01110+11100} &  \multicolumn{1}{p{1.75cm}|}{\centering 12} &  \multicolumn{1}{p{2.38cm}|}{\centering 00200} \\
 \cline{1-1} \cline{2-2} \cline{3-3}
 \multicolumn{1}{|p{4.59cm}|}{\centering 00111+01101+11110} &  \multicolumn{1}{p{1.75cm}|}{\centering 11} &  \multicolumn{1}{p{2.38cm}|}{\centering 01011} \\
 \cline{1-1} \cline{2-2} \cline{3-3}
 \multicolumn{1}{|p{4.59cm}|}{\centering 01101+11110} &  \multicolumn{1}{p{1.75cm}|}{\centering 10} &  \multicolumn{1}{p{2.38cm}|}{\centering 02000} \\
 \cline{1-1} \cline{2-2} \cline{3-3}
 \multicolumn{1}{|p{4.59cm}|}{\centering 01111+11101+11110} &  \multicolumn{1}{p{1.75cm}|}{\centering 9} &  \multicolumn{1}{p{2.38cm}|}{\centering 10100} \\
 \cline{1-1} \cline{2-2} \cline{3-3}
 \multicolumn{1}{|p{4.59cm}|}{\centering 01111+11101} &  \multicolumn{1}{p{1.75cm}|}{\centering 7} &  \multicolumn{1}{p{2.38cm}|}{\centering 00011} \\
 \cline{1-1} \cline{2-2} \cline{3-3}
 \multicolumn{1}{|p{4.59cm}|}{\centering 01111+11110} &  \multicolumn{1}{p{1.75cm}|}{\centering 7} &  \multicolumn{1}{p{2.38cm}|}{\centering 00011} \\
 \cline{1-1} \cline{2-2} \cline{3-3}
 \multicolumn{1}{|p{4.59cm}|}{\centering 11101+11110} &  \multicolumn{1}{p{1.75cm}|}{\centering 6} &  \multicolumn{1}{p{2.38cm}|}{\centering 20000} \\
 \cline{1-1} \cline{2-2} \cline{3-3}
 \multicolumn{1}{|p{4.59cm}|}{\centering 11111} &  \multicolumn{1}{p{1.75cm}|}{\centering 5} &  \multicolumn{1}{p{2.38cm}|}{\centering 01000} \\
 \cline{1-1} \cline{2-2} \cline{3-3}
 \multicolumn{1}{|p{4.59cm}|}{\centering 00000} &  \multicolumn{1}{p{1.75cm}|}{\centering 0} &  \multicolumn{1}{p{2.38cm}|}{\centering 00000} \\
 \cline{1-1} \cline{2-2} \cline{3-3}
\end{tabular}\end{center}

\noindent%
 Since the action on $\frak u_2=\C X_{01211}+\C X_{11211}+\C X_{12211}$ is the standard action of $SL(3)$, it has 2 orbits:~zero and nonzero.  It occurs for the lower rank group $SL(4)$ and node 3.  A representative for the larger orbit is $X_{12211}$, which intersects the minimal coadjoint orbit, 01000.

\subsubsection{$D_5$ nodes 4 and 5}\label{sec:D5node4}

In the case of the spinor nodes 4 and 5, the action breaks up into three orbits, of dimensions 10, 7, and 0. This is the 10 dimensional exterior square representation of $SL(5)$.  These orbits intersect the coadjoint nilpotent orbits with weighted Dynkin diagrams $00011$, $01000$, and $00000$, respectively.

\begin{center}\begin{tabular}{|cp{1.88cm}||p{2.28cm}||}
\hline
 \multicolumn{1}{|c|}{Orbit Basepoint} &  \multicolumn{1}{p{1.88cm}|}{\centering Dimension} &  \multicolumn{1}{p{2.28cm}|}{\centering Coadjoint orbit intersected} \\
\hline
 \multicolumn{1}{|c|}{01211+11111} &  \multicolumn{1}{p{1.88cm}|}{\centering 10} &  \multicolumn{1}{p{2.28cm}|}{\centering 00011} \\
 \cline{1-1} \cline{2-2} \cline{3-3}
 \multicolumn{1}{|c|}{12211} &  \multicolumn{1}{p{1.88cm}|}{\centering 7} &  \multicolumn{1}{p{2.28cm}|}{\centering 01000} \\
 \cline{1-1} \cline{2-2} \cline{3-3}
 \multicolumn{1}{|c|}{00000} &  \multicolumn{1}{p{1.88cm}|}{\centering 0} &  \multicolumn{1}{p{2.28cm}|}{\centering 00000} \\
\hline
\end{tabular}\end{center}

\subsection{Type $D_n$ : $SO(n,n)$}

We limit the discussion here to $n>5$, since the lower rank cases have already been discussed.

\subsubsection{$D_n$ node $k<n-1$}\label{sec:Dnnodesmall}

The Levi component of the standard maximal parabolic subgroup $P_{\a_k}$ is of type $A_{k-1}\times D_{n-k}$. The unipotent radical is nonabelian except for $k=1$.  The action on ${\frak u}_2$ is the symmetric square action of $SL(k)$ on symmetric $k$-tensors and arises for $D_k$ node $k$; its orbits are described in section~\ref{sec:Dnnodebig}.  We thus focus on $\frak u_1$ here, on which $L$ acts by the tensor product of the standard representation of $SL(k)$ with the vector representation of $SO(n-k,n-k)$.

The  theory here strongly resembles the case of $B_n$ from section~\ref{sec:Bn} because Witt's theorem again applies nearly verbatim.  The rank restriction (\ref{Bnineq1}) applies directly, while (\ref{Bnineq2}) and (\ref{Bnineq3}) need only to be adjusted by replacing $2n-2k+1$ by $2n-2k$.  Using the matrices $A_{s,p}$ defined in (\ref{Bnrepresenative}), but instead with a basis of dimension $2n-2k$ of course, we obtain distinct  orbit representatives $A_{s,p}$, where $0\le p\le \min\{k,2n-2k\}$, $0\le p+s\le \min\{k,2n-2k\}$, $2s+p\le k$, and $2(p+s)-(2n-2k)\le p$,
for each of the possible configuration of ranks satisfying the inequalities.  However, there is an additional wrinkle in this case:~lemma~\ref{M12} is an assertion about orbit representatives of $GL(k)\times O(n-k,n-k)$.  The Levi is connected, and thus its action has  a second orbit having $\operatorname{rank}(A)=n-k$ and $A^t JA=0$ besides  the one generated by $A_{n-k,0}$.  A representative for this orbit can be given by replacing $w_1+iw_2$ in (\ref{Bnrepresenative}) with $w_1-iw_2$.  This matrix, when combined with the $A_{s,p}$ just listed, comprise a full set of orbit representatives for the complex Levi action on $\frak u_1$.
  This is the only time this phenomenon comes up directly for actions on $\frak u_1$, though note that the action on $\frak u_2$ in section~\ref{sec:E8node5} is equivalent to a $D_8$ action; the second and third of its 18 dimensional orbits are similarly related.

\subsubsection{$D_n$ node $k=n-1$ or $n$}\label{sec:Dnnodebig}

The two cases here are related by a Dynkin diagram symmetry.  The Levi component has $[L,L]$ of type $SL(n)$
and the unipotent radical is abelian of dimension $\f{n(n-1)}{2}$.
The action of $L$ on $\frak u$ is  the exterior square action of $SL(n)$ on antisymmetric $n$-tensors.  Analogously to the situation in section~\ref{sec:Cnnoden},
 this can be naturally viewed as the action of $g\in SL(n)$ on antisymmetric $n\times n$ matrices $X$ given by $g:X\mapsto gXg^t$.  Lemma~\ref{M12} applies here, and shows that the action
 has $\lfloor \f{n}{2} \rfloor+1$ orbits, given by even rank matrices of the form
 \begin{equation}\label{Dnnodenreps}
    A_p \ \ = \ \ \tthree{}{}{J_p}{}{I_{n-2p}}{}{-J_p}{}{} \ , \ \ 0 \,\le \, 2p \,\le n\,,
 \end{equation}
 where $J_p$ is the reverse $p\times p$ identity matrix.

\begin{center}
\begin{figure}
  \includegraphics{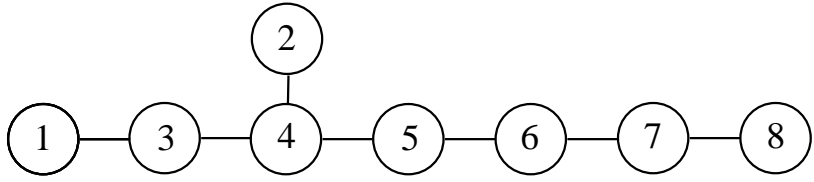}\\
  \caption{$E_n$ Dynkin diagrams.  The diagram for $E_8$ is shown here, while the diagrams for $E_7$ and $E_6$ are given by removing nodes $8$ and then $7$, respectively.}\label{EnDynkins}
\end{figure}
\end{center}

\subsection{Type $E_6$}

Recall from Figure~\ref{EnDynkins} that nodes 5 and 6 are equivalent, respectively, to nodes 3 and 1, so it suffices to discuss nodes 1, 2, 3, and 4 here.  The internal Chevalley modules are described by the table below.   In this and the analogous tables for other groups we will abbreviate the types of some semisimple groups below to their Cartan labels, as well as some of the descriptions of the representations.  The weights will again be listed  using the numbering of the ambient $E_6$ Dynkin diagram.

\begin{center}\begin{tabular}{|c|c|c|c|c|}
\hline
 Node &  Type of $[L,L]$ &  $i\,=\,1$ &  $i\,=\,2$ & $i\,=\,3$ \\
\hline
 1 &  \multicolumn{1}{l|}{$SO(5,5)$} &  Spin representation &   &  \\
  &  $\dim \frak u_i$ &  16 &   &  \\
  &  action &  $\varpi_2$ &   &  \\
\hline
 2 &  \multicolumn{1}{l|}{$SL(6)$} &  Exterior cube &   &  \\
  &  $\dim \frak u_i$ &  20 &  1 &  \\
  &  action &  $\varpi_4$ &   &  \\
\hline
 3 &  \multicolumn{1}{l|}{$SL(2) \times SL(5)$} &  Standard $\otimes$ Exterior square &   &  \\
  &  $\dim \frak u_i$ &  20 &  5 &  \\
  &  action &  $\varpi_1+\varpi_5$ & $\varpi_2$   &  \\
\hline
 4 &  \multicolumn{1}{l|}{$A_2\times A_1 \times A_2$} &  Tensor product &  &  \\
  &  $\dim \frak u_i$ &  18 &  9 & 2 \\
  &  action &  $\varpi_1+\varpi_2+\varpi_6$ & $\varpi_3+\varpi_5$  & $\varpi_2$ \\
\hline
 5 &  \multicolumn{1}{l|}{$SL(5)\times SL(2)$} &  Standard $\otimes$ Exterior square &   &  \\
  &  $\dim \frak u_i$ &  20 &  5 &  \\
  &  action &  $\varpi_3+\varpi_6$ &  $\varpi_2$ &  \\
\hline
 6 &  \multicolumn{1}{l|}{$SO(5,5)$} &  Spin &   &  \\
  &  $\dim \frak u_i$ &  16 &   &  \\
  &  action &  $\varpi_2$ &   &  \\
\hline
\end{tabular}\end{center}

\subsubsection{$E_6$ Node 1}\label{sec:E6node1}

Here $P=LU$ has $[L,L]$ of type $SO(5,5)$ and $U$ is 16 dimensional and abelian.  The action of $L$ on $\frak u = \frak u_1$ is the 16-dimensional spin representation.  It has 3 orbits:

\begin{center}\begin{tabular}{|c|p{1.81cm}||p{2.34cm}||}
\hline
 Orbit Basepoint &  \multicolumn{1}{p{1.81cm}|}{\centering Dimension} &  \multicolumn{1}{p{2.34cm}|}{\centering Coadjoint orbit intersected} \\
\hline
 \multicolumn{1}{|c|}{111221+112211} &  \multicolumn{1}{p{1.81cm}|}{\centering 16} &  \multicolumn{1}{p{2.34cm}|}{\centering 100001} \\
 \cline{1-1} \cline{2-2} \cline{3-3}
 \multicolumn{1}{|c|}{122321} &  \multicolumn{1}{p{1.81cm}|}{\centering 11} &  \multicolumn{1}{p{2.34cm}|}{\centering 010000} \\
 \cline{1-1} \cline{2-2} \cline{3-3}
 000000 &  \multicolumn{1}{p{1.81cm}|}{\centering 0} &  \multicolumn{1}{p{2.34cm}|}{\centering 000000} \\
\hline
\end{tabular}\end{center}

\subsubsection{$E_6$ Node 2}\label{sec:E6node2}

 In this case $[L,L]$ is of type $SL(6)$ and $U$ is a 21-dimensional Heisenberg group.  The action on the 20-dimensional space ${\frak u_1}$ is  the exterior cube representation, and breaks up into 5 orbits:
\begin{center}\begin{tabular}{||p{4.34cm}p{1.94cm}||p{2.25cm}||}
\hline
 \multicolumn{1}{|p{4.34cm}|}{\raggedright \centering Orbit Basepoint} &  \multicolumn{1}{p{1.94cm}|}{\centering Dimension} &  \multicolumn{1}{p{2.25cm}|}{\centering Coadjoint orbit intersected} \\
\hline
 \multicolumn{1}{|p{4.34cm}|}{\raggedright \centering 010111+112210} &  \multicolumn{1}{p{1.94cm}|}{\centering 20} &  \multicolumn{1}{p{2.25cm}|}{\centering 020000} \\
 \cline{1-1} \cline{2-2} \cline{3-3}
 \multicolumn{1}{|p{4.34cm}|}{\raggedright \centering 011221+111211+112210} &  \multicolumn{1}{p{1.94cm}|}{\centering 19} &  \multicolumn{1}{p{2.25cm}|}{\centering 000100} \\
 \cline{1-1} \cline{2-2} \cline{3-3}
 \multicolumn{1}{|p{4.34cm}|}{\raggedright \centering 111221+112211} &  \multicolumn{1}{p{1.94cm}|}{\centering 15} &  \multicolumn{1}{p{2.25cm}|}{\centering 100001} \\
 \cline{1-1} \cline{2-2} \cline{3-3}
 \multicolumn{1}{|p{4.34cm}|}{\raggedright \centering 112321} &  \multicolumn{1}{p{1.94cm}|}{\centering 10} &  \multicolumn{1}{p{2.25cm}|}{\centering 010000} \\
 \cline{1-1} \cline{2-2} \cline{3-3}
 \multicolumn{1}{|p{4.34cm}|}{\raggedright \centering 000000} &  \multicolumn{1}{p{1.94cm}|}{\centering 0} &  \multicolumn{1}{p{2.25cm}|}{\centering 000000} \\
\hline
\end{tabular}\end{center}

The action on the one dimensional piece $\frak u_2$ has two orbits:~zero and nonzero.

\subsubsection{$E_6$ Node 3}\label{sec:E6node3}

In this case $[L,L]$ is of type  $SL(2)\times SL(5)$ and $U$ is a 25-dimensional two-step unipotent group.  The Lie algebra decomposes as
${\frak u}=\frak u_1\oplus \frak u_2$ with $\dim \frak u_1=20$ and $\dim\frak u_2=5$.  The action on $\frak u_1$ is the tensor product of the standard representation of $SL(2)$ with the 10 dimensional exterior square representation of $SL(5)$, and has 8 orbits:
\begin{center}\begin{tabular}{p{5.28cm}p{1.94cm}p{2.38cm}}
 \cline{1-1} \cline{2-2} \cline{3-3}
 \multicolumn{1}{|p{5.28cm}|}{\centering Orbit Basepoint} &  \multicolumn{1}{p{1.94cm}|}{\centering Dimension} &  \multicolumn{1}{p{2.38cm}|}{\centering Coadjoint orbit intersected} \\
\cline{1-1} \cline{2-2} \cline{3-3}
 \multicolumn{1}{|p{5.28cm}|}{\centering 011111+011210+101111+111110} &  \multicolumn{1}{p{1.94cm}|}{\centering 20} &  \multicolumn{1}{p{2.38cm}|}{\centering 001010} \\
 \cline{1-1} \cline{2-2} \cline{3-3}
 \multicolumn{1}{|p{5.28cm}|}{\centering 011111+101111+111210} &  \multicolumn{1}{p{1.94cm}|}{\centering 18} &  \multicolumn{1}{p{2.38cm}|}{\centering 110001} \\
 \cline{1-1} \cline{2-2} \cline{3-3}
 \multicolumn{1}{|p{5.28cm}|}{\centering 011111+111210} &  \multicolumn{1}{p{1.94cm}|}{\centering 16} &  \multicolumn{1}{p{2.38cm}|}{\centering 020000} \\
 \cline{1-1} \cline{2-2} \cline{3-3}
 \multicolumn{1}{|p{5.28cm}|}{\centering 011221+111111+111210} &  \multicolumn{1}{p{1.94cm}|}{\centering 15} &  \multicolumn{1}{p{2.38cm}|}{\centering 000100} \\
 \cline{1-1} \cline{2-2} \cline{3-3}
 \multicolumn{1}{|p{5.28cm}|}{\centering 011221+111211} &  \multicolumn{1}{p{1.94cm}|}{\centering 12} &  \multicolumn{1}{p{2.38cm}|}{\centering 100001} \\
 \cline{1-1} \cline{2-2} \cline{3-3}
 \multicolumn{1}{|p{5.28cm}|}{\centering 111111+111210} &  \multicolumn{1}{p{1.94cm}|}{\centering 11} &  \multicolumn{1}{p{2.38cm}|}{\centering 100001} \\
 \cline{1-1} \cline{2-2} \cline{3-3}
 \multicolumn{1}{|p{5.28cm}|}{\centering 111221} &  \multicolumn{1}{p{1.94cm}|}{\centering 8} &  \multicolumn{1}{p{2.38cm}|}{\centering 010000} \\
 \cline{1-1} \cline{2-2} \cline{3-3}
 \multicolumn{1}{|p{5.28cm}|}{\centering 000000} &  \multicolumn{1}{p{1.94cm}|}{\centering 0} &  \multicolumn{1}{p{2.38cm}|}{\centering 000000} \\
 \cline{1-1} \cline{2-2} \cline{3-3}
\end{tabular}\end{center}

The action on $\frak u_2$ is the 5-dimensional action of $SL(5)$, and breaks up into 2 orbits: zero and nonzero.  A representative for the big orbit is the highest root $122321$, which lies in the minimal coadjoint nilpotent orbit $010000$.

\subsubsection{$E_6$ Node 4}\label{sec:E6node4}

This is the first case with a 3-step nilpotent group.  We have $P=LU$ where $[L,L]$ is of type $SL(3)\times SL(3)\times SL(2)$, and $U$ is 29 dimensional  with
\begin{equation}\label{e6node4}
    \frak u \ \ = \ \ \frak u_1 \,\oplus\,\frak u_2 \,\oplus\,\frak u_3\, , \ \ \dim\,\frak u_1 \,=\,18\,,
    \ \dim\,\frak u_2 \,=\,9\,, \ \text{and}\,\dim\,\frak u_3 \,=\,2\,.
\end{equation}
The action on the 18 dimensional piece $\frak u_1$ is the tensor product of standard representations of the three factors, and has 18 orbits:

\begin{center}\begin{tabular}{|cp{1.84cm}||p{1.81cm}||}
\hline
 \multicolumn{1}{|c|}{Orbit Basepoint} &  \multicolumn{1}{p{1.84cm}|}{\centering Dimension} &  \multicolumn{1}{p{1.81cm}|}{\centering Coadjoint orbit intersected} \\
\hline
 \multicolumn{1}{|c|}{000111+010111+011110+101100} &  \multicolumn{1}{p{1.84cm}|}{\centering 18} &  \multicolumn{1}{p{1.81cm}|}{\centering 000200} \\
 \cline{1-1} \cline{2-2} \cline{3-3}
 \multicolumn{1}{|c|}{001111+010110+101110+111100} &  \multicolumn{1}{p{1.84cm}|}{\centering 17} &  \multicolumn{1}{p{1.81cm}|}{\centering 011010} \\
 \cline{1-1} \cline{2-2} \cline{3-3}
 \multicolumn{1}{|c|}{001111+010111+011110+101110+111100} &  \multicolumn{1}{p{1.84cm}|}{\centering 16} &  \multicolumn{1}{p{1.81cm}|}{\centering 100101} \\
 \cline{1-1} \cline{2-2} \cline{3-3}
 \multicolumn{1}{|c|}{010110+011100+101111} &  \multicolumn{1}{p{1.84cm}|}{\centering 15} &  \multicolumn{1}{p{1.81cm}|}{\centering 120001} \\
 \cline{1-1} \cline{2-2} \cline{3-3}
 \multicolumn{1}{|c|}{001111+010111+101110+111100} &  \multicolumn{1}{p{1.84cm}|}{\centering 14} &  \multicolumn{1}{p{1.81cm}|}{\centering 200002} \\
 \cline{1-1} \cline{2-2} \cline{3-3}
 \multicolumn{1}{|c|}{001111+010111+011110+101110} &  \multicolumn{1}{p{1.84cm}|}{\centering 14} &  \multicolumn{1}{p{1.81cm}|}{\centering 001010} \\
 \cline{1-1} \cline{2-2} \cline{3-3}
 \multicolumn{1}{|c|}{001111+011110+101110+111100} &  \multicolumn{1}{p{1.84cm}|}{\centering 14} &  \multicolumn{1}{p{1.81cm}|}{\centering 001010} \\
 \cline{1-1} \cline{2-2} \cline{3-3}
 \multicolumn{1}{|c|}{010111+011110+101111+111100} &  \multicolumn{1}{p{1.84cm}|}{\centering 14} &  \multicolumn{1}{p{1.81cm}|}{\centering 001010} \\
 \cline{1-1} \cline{2-2} \cline{3-3}
 \multicolumn{1}{|c|}{001111+010111+111110} &  \multicolumn{1}{p{1.84cm}|}{\centering 13} &  \multicolumn{1}{p{1.81cm}|}{\centering 110001} \\
 \cline{1-1} \cline{2-2} \cline{3-3}
 \multicolumn{1}{|c|}{011111+101110+111100} &  \multicolumn{1}{p{1.84cm}|}{\centering 13} &  \multicolumn{1}{p{1.81cm}|}{\centering 110001} \\
 \cline{1-1} \cline{2-2} \cline{3-3}
 \multicolumn{1}{|c|}{001111+111110} &  \multicolumn{1}{p{1.84cm}|}{\centering 12} &  \multicolumn{1}{p{1.81cm}|}{\centering 020000} \\
 \cline{1-1} \cline{2-2} \cline{3-3}
 \multicolumn{1}{|c|}{011111+101111+111110} &  \multicolumn{1}{p{1.84cm}|}{\centering 11} &  \multicolumn{1}{p{1.81cm}|}{\centering 000100} \\
 \cline{1-1} \cline{2-2} \cline{3-3}
 \multicolumn{1}{|c|}{010111+011110+111100} &  \multicolumn{1}{p{1.84cm}|}{\centering 10} &  \multicolumn{1}{p{1.81cm}|}{\centering 000100} \\
 \cline{1-1} \cline{2-2} \cline{3-3}
 \multicolumn{1}{|c|}{011111+111110} &  \multicolumn{1}{p{1.84cm}|}{\centering 9} &  \multicolumn{1}{p{1.81cm}|}{\centering 100001} \\
 \cline{1-1} \cline{2-2} \cline{3-3}
 \multicolumn{1}{|c|}{011111+101111} &  \multicolumn{1}{p{1.84cm}|}{\centering 8} &  \multicolumn{1}{p{1.81cm}|}{\centering 100001} \\
 \cline{1-1} \cline{2-2} \cline{3-3}
 \multicolumn{1}{|c|}{101111+111110} &  \multicolumn{1}{p{1.84cm}|}{\centering 8} &  \multicolumn{1}{p{1.81cm}|}{\centering 100001} \\
 \cline{1-1} \cline{2-2} \cline{3-3}
 \multicolumn{1}{|c|}{111111} &  \multicolumn{1}{p{1.84cm}|}{\centering 6} &  \multicolumn{1}{p{1.81cm}|}{\centering 010000} \\
 \cline{1-1} \cline{2-2} \cline{3-3}
 \multicolumn{1}{|c|}{000000} &  \multicolumn{1}{p{1.84cm}|}{\centering 0} &  \multicolumn{1}{p{1.81cm}|}{\centering 000000} \\
\hline
\end{tabular}\end{center}

The action on the 9-dimensional piece $\frak u_2$ occurs for $SL(6)$, node 3, and has 4 orbits that are parameterized by rank.     The nontrivial orbits there have dimensions 9, 8, and 5, with basepoints 00111 + 01110 + 11100, 01111 + 11110, and 11111.  Hence the orbits on $\frak u_2$ are given as follows:

\begin{center}\begin{tabular}{||p{3.94cm}p{1.97cm}||p{2.31cm}||}
\hline
 \multicolumn{1}{|p{3.94cm}|}{\centering Orbit Basepoint} &  \multicolumn{1}{p{1.97cm}|}{\centering Dimension} &  \multicolumn{1}{p{2.31cm}|}{\centering Coadjoint orbit intersected} \\
\hline
 \multicolumn{1}{|p{3.94cm}|}{\centering 112210+111211+011221} &  \multicolumn{1}{p{1.97cm}|}{\centering 9} &  \multicolumn{1}{p{2.31cm}|}{\centering 000100} \\
 \cline{1-1} \cline{2-2} \cline{3-3}
 \multicolumn{1}{|p{3.94cm}|}{\centering 111221+112211} &  \multicolumn{1}{p{1.97cm}|}{\centering 8} &  \multicolumn{1}{p{2.31cm}|}{\centering 100001} \\
 \cline{1-1} \cline{2-2} \cline{3-3}
 \multicolumn{1}{|p{3.94cm}|}{\centering 112221} &  \multicolumn{1}{p{1.97cm}|}{\centering 5} &  \multicolumn{1}{p{2.31cm}|}{\centering 010000} \\
 \cline{1-1} \cline{2-2} \cline{3-3}
 \multicolumn{1}{|p{3.94cm}|}{\centering 000000} &  \multicolumn{1}{p{1.97cm}|}{\centering 0} &  \multicolumn{1}{p{2.31cm}|}{\centering 000000} \\
\hline
\end{tabular}\end{center}

The action on $\frak u_3$ is the standard action of $SL(2)$, and breaks up into 2 orbits: zero and nonzero.  A representative for the big orbit is the highest root $122321$, which lies in the minimal coadjoint nilpotent orbit $010000$.

\subsection{Type $E_7$}
The internal Chevalley actions for maximal parabolics are given as follows, with the same labeling conventions used for $D_5$ and $E_6$.
\begin{center}\begin{tabular}{|c|c|c|c|c|c|}
\hline
 Node &  Type of $[L,L]$ &  $i\,=\,1$ &  $i\,=\,2$ &  $i\,=\,3$ & $i\,=\,4$ \\
\hline
 1 &  \multicolumn{1}{l|}{$SO(6,6)$} &  Spin &   &   &  \\
  &  $\dim \frak u_i$ &  32 &  1 &   &  \\
  &  action &  $\varpi_3$ &  $0$ &   &  \\
\hline
 2 &  \multicolumn{1}{l|}{$SL(7)$} & Ext. cube   &   &   &  \\
  &  $\dim \frak u_i$ &  35 &  7 &   &  \\
  &  action &  $\varpi_5$ &  $\varpi_1$ &   &  \\
\hline
 3 &  \multicolumn{1}{l|}{$SL(2) \times SL(6)$} &  Stan. $\otimes$ Ext. sq. &   &   &  \\
  &  $\dim \frak u_i$ &  30 &  15 &  2 &  \\
  &  action &  $\varpi_1+\varpi_6$ &  $\varpi_4$ &  $\varpi_1$ &  \\
\hline
 4 &  \multicolumn{1}{l|}{$A_2\times A_1\times A_3$} &  Tensor &   &   &  \\
  &  $\dim \frak u_i$ &  24 &  18 &  8 & 3 \\
  &  action &  $\varpi_1+\varpi_2+\varpi_7$ &  $\varpi_3+\varpi_6$ &  $\varpi_2+\varpi_5$ & $\varpi_1$ \\
\hline
 5 &  \multicolumn{1}{l|}{$SL(5)\times SL(3)$} &  Ext. sq. $\otimes$ Stan. &   &   &  \\
  &  $\dim \frak u_i$ &  30 &  15 &  5 &  \\
  &  action &  $\varpi_3+\varpi_7$ &  $\varpi_2+\varpi_6$ &  $\varpi_1$ &  \\
\hline
 6 &  \multicolumn{1}{l|}{$D_5\times A_1$} &  Spin $\otimes$ Stan. &  Vector &   &  \\
  &  $\dim \frak u_i$ &  32 &  10 &   &  \\
  &  action &  $\varpi_2+\varpi_7$ &  $\varpi_1$ &   &  \\
\hline
 7 &  \multicolumn{1}{l|}{$E_6$} &  Stan.  &   &   &  \\
  &  $\dim \frak u_i$ &  27 &   &   &  \\
  &  action &  $\varpi_1$ &   &   &  \\
\hline
\end{tabular}\end{center}

\subsubsection{$E_7$ Node 1}\label{sec:E7node1}

In this case $P=LU$ where $U$ is a 33-dimensional Heisenberg group and $\frak u=\frak u_1\oplus \frak u_2$, with $\dim\frak u_1=32$ and $\dim\frak u_2=1$.  The semisimple part $[L,L]$ of $L$ has  type $SO(6,6)$, and acts on $\frak u_1$ by the spin representation with 5 orbits:

\begin{center}\begin{tabular}{||p{4.56cm}||p{1.78cm}||p{2.25cm}||}
\hline
 \multicolumn{1}{|p{4.56cm}|}{\centering Orbit Basepoint} &  \multicolumn{1}{p{1.78cm}|}{\centering Dimension} &  \multicolumn{1}{p{2.25cm}|}{\centering Coadjoint orbit intersected} \\
\hline
 \multicolumn{1}{|p{4.56cm}|}{\centering 1011111+1223210} &  \multicolumn{1}{p{1.78cm}|}{\centering 32} &  \multicolumn{1}{p{2.25cm}|}{\centering 2000000} \\
 \cline{1-1} \cline{2-2} \cline{3-3}
 \multicolumn{1}{|p{4.56cm}|}{\centering 1122221+1123211+1223210} &  \multicolumn{1}{p{1.78cm}|}{\centering 31} &  \multicolumn{1}{p{2.25cm}|}{\centering 0010000} \\
 \cline{1-1} \cline{2-2} \cline{3-3}
 \multicolumn{1}{|p{4.56cm}|}{\centering 1123321+1223221} &  \multicolumn{1}{p{1.78cm}|}{\centering 25} &  \multicolumn{1}{p{2.25cm}|}{\centering 0000010} \\
 \cline{1-1} \cline{2-2} \cline{3-3}
 \multicolumn{1}{|p{4.56cm}|}{\centering 1234321} &  \multicolumn{1}{p{1.78cm}|}{\centering 16} &  \multicolumn{1}{p{2.25cm}|}{\centering 1000000} \\
 \cline{1-1} \cline{2-2} \cline{3-3}
 \multicolumn{1}{|p{4.56cm}|}{\centering 0000000} &  \multicolumn{1}{p{1.78cm}|}{\centering 0} &  \multicolumn{1}{p{2.25cm}|}{\centering 0000000} \\
\hline
\end{tabular}\end{center}
The action on the one-dimensional piece $\frak u_2$ has two orbits:~zero and nonzero.

\subsubsection{$E_7$ Node 2}\label{sec:E7node2}

In this case $P=LU$ where $U$ is a 35-dimensional 2-step nilpotent group and $\frak u=\frak u_1\oplus \frak u_2$, with $\dim\frak u_1=35$ and $\dim\frak u_2=7$.  The semisimple part $[L,L]$ of $L$ has type $SL(7)$, and acts on $\frak u_1$ as it does on antisymmetric 3-tensors:

\begin{center}\begin{tabular}{||p{7.3cm}p{1.67cm}p{1.75cm}||}
\hline
 \multicolumn{1}{|p{7.3cm}|}{\centering Orbit Basepoint} &  \multicolumn{1}{p{1.67cm}|}{\centering Dimension} &  \multicolumn{1}{p{1.75cm}|}{\centering Coadjoint orbit intersected} \\
\hline
 \multicolumn{1}{|p{7.3cm}|}{\centering 0112111+0112210+1111111+1112110+1122100} &  \multicolumn{1}{p{1.67cm}|}{\centering 35} &  \multicolumn{1}{p{1.75cm}|}{\centering 0200000} \\
 \cline{1-1} \cline{2-2} \cline{3-3}
 \multicolumn{1}{|p{7.3cm}|}{\centering 0112211+1112111+1112210+1122110} &  \multicolumn{1}{p{1.67cm}|}{\centering 34} &  \multicolumn{1}{p{1.75cm}|}{\centering 0001000} \\
 \cline{1-1} \cline{2-2} \cline{3-3}
 \multicolumn{1}{|p{7.3cm}|}{\centering 0112221+1111111+1123210} &  \multicolumn{1}{p{1.67cm}|}{\centering 31} &  \multicolumn{1}{p{1.75cm}|}{\centering 1000010} \\
 \cline{1-1} \cline{2-2} \cline{3-3}
 \multicolumn{1}{|p{7.3cm}|}{\centering 0112221+1112211+1122111+1123210} &  \multicolumn{1}{p{1.67cm}|}{\centering 28} &  \multicolumn{1}{p{1.75cm}|}{\centering 0100001} \\
 \cline{1-1} \cline{2-2} \cline{3-3}
 \multicolumn{1}{|p{7.3cm}|}{\centering 1111111+1123210} &  \multicolumn{1}{p{1.67cm}|}{\centering 26} &  \multicolumn{1}{p{1.75cm}|}{\centering 2000000} \\
 \cline{1-1} \cline{2-2} \cline{3-3}
 \multicolumn{1}{|p{7.3cm}|}{\centering 1112221+1122211+1123210} &  \multicolumn{1}{p{1.67cm}|}{\centering 25} &  \multicolumn{1}{p{1.75cm}|}{\centering 0010000} \\
 \cline{1-1} \cline{2-2} \cline{3-3}
 \multicolumn{1}{|p{7.3cm}|}{\centering 0112221+1112211+1122111} &  \multicolumn{1}{p{1.67cm}|}{\centering 21} &  \multicolumn{1}{p{1.75cm}|}{\centering 0000002} \\
 \cline{1-1} \cline{2-2} \cline{3-3}
 \multicolumn{1}{|p{7.3cm}|}{\centering 1122221+1123211} &  \multicolumn{1}{p{1.67cm}|}{\centering 20} &  \multicolumn{1}{p{1.75cm}|}{\centering 0000010} \\
 \cline{1-1} \cline{2-2} \cline{3-3}
 \multicolumn{1}{|p{7.3cm}|}{\centering 1123321} &  \multicolumn{1}{p{1.67cm}|}{\centering 13} &  \multicolumn{1}{p{1.75cm}|}{\centering 1000000} \\
 \cline{1-1} \cline{2-2} \cline{3-3}
 \multicolumn{1}{|p{7.3cm}|}{\centering 0000000} &  \multicolumn{1}{p{1.67cm}|}{\centering 0} &  \multicolumn{1}{p{1.75cm}|}{\centering 0000000} \\
\hline
\end{tabular}\end{center}

The action on the 7-dimensional piece $\frak u_2$ is the standard action of $GL(7)$, and has 2 orbits:~zero and nonzero.

\subsubsection{$E_7$ Node 3}\label{sec:E7node3}

In this case $P=LU$ where $U$ is a 47-dimensional 3-step nilpotent group and $\frak u=\frak u_1\oplus \frak u_2\oplus \frak u_3$, with $\dim\frak u_1=30$, $\dim\frak u_2=15$,  and $\dim\frak u_3=2$.  The semisimple part $[L,L]$ of $L$ is of type $SL(2)\times SL(6)$.  Its action on $\frak u_1$ is the tensor product of the standard action of the $SL(2)$ with the 15-dimensional action of the $SL(6)$ on antisymmetric 2-tensors:

\begin{center}\begin{tabular}{p{7.13cm}p{1.75cm}p{1.84cm}}
 \cline{1-1} \cline{2-2} \cline{3-3}
 \multicolumn{1}{|p{7.13cm}|}{\centering Orbit Basepoint} &  \multicolumn{1}{p{1.75cm}|}{\centering Dimension} &  \multicolumn{1}{p{1.84cm}|}{\centering Coadjoint orbit intersected} \\
\cline{1-1} \cline{2-2} \cline{3-3}
 \multicolumn{1}{|p{7.13cm}|}{\centering 0011111+0111110+1011111+1112100} &  \multicolumn{1}{p{1.75cm}|}{\centering 30} &  \multicolumn{1}{p{1.84cm}|}{\centering 0020000} \\
 \cline{1-1} \cline{2-2} \cline{3-3}
 \multicolumn{1}{|p{7.13cm}|}{\centering 0112111+0112210+1011110+1111100} &  \multicolumn{1}{p{1.75cm}|}{\centering 29} &  \multicolumn{1}{p{1.84cm}|}{\centering 1001000} \\
 \cline{1-1} \cline{2-2} \cline{3-3}
 \multicolumn{1}{|p{7.13cm}|}{\centering 0112111+0112210+1011111+1111110+1112100} &  \multicolumn{1}{p{1.75cm}|}{\centering 28} &  \multicolumn{1}{p{1.84cm}|}{\centering 0010010} \\
 \cline{1-1} \cline{2-2} \cline{3-3}
 \multicolumn{1}{|p{7.13cm}|}{\centering 0111111+0112210+1011111+1112110} &  \multicolumn{1}{p{1.75cm}|}{\centering 26} &  \multicolumn{1}{p{1.84cm}|}{\centering 0000020} \\
 \cline{1-1} \cline{2-2} \cline{3-3}
 \multicolumn{1}{|p{7.13cm}|}{\centering 0112221+1011100+1111000} &  \multicolumn{1}{p{1.75cm}|}{\centering 25} &  \multicolumn{1}{p{1.84cm}|}{\centering 2000010} \\
 \cline{1-1} \cline{2-2} \cline{3-3}
 \multicolumn{1}{|p{7.13cm}|}{\centering 0112111+0112210+1111111+1112110} &  \multicolumn{1}{p{1.75cm}|}{\centering 25} &  \multicolumn{1}{p{1.84cm}|}{\centering 0001000} \\
 \cline{1-1} \cline{2-2} \cline{3-3}
 \multicolumn{1}{|p{7.13cm}|}{\centering 0112221+1011111+1111110+1112100} &  \multicolumn{1}{p{1.75cm}|}{\centering 24} &  \multicolumn{1}{p{1.84cm}|}{\centering 0001000} \\
 \cline{1-1} \cline{2-2} \cline{3-3}
 \multicolumn{1}{|p{7.13cm}|}{\centering 0112111+1111111+1112210} &  \multicolumn{1}{p{1.75cm}|}{\centering 23} &  \multicolumn{1}{p{1.84cm}|}{\centering 1000010} \\
 \cline{1-1} \cline{2-2} \cline{3-3}
 \multicolumn{1}{|p{7.13cm}|}{\centering 0112111+1112210} &  \multicolumn{1}{p{1.75cm}|}{\centering 20} &  \multicolumn{1}{p{1.84cm}|}{\centering 2000000} \\
 \cline{1-1} \cline{2-2} \cline{3-3}
 \multicolumn{1}{|p{7.13cm}|}{\centering 0112221+1112111+1112210} &  \multicolumn{1}{p{1.75cm}|}{\centering 19} &  \multicolumn{1}{p{1.84cm}|}{\centering 0010000} \\
 \cline{1-1} \cline{2-2} \cline{3-3}
 \multicolumn{1}{|p{7.13cm}|}{\centering 1011111+1111110+1112100} &  \multicolumn{1}{p{1.75cm}|}{\centering 16} &  \multicolumn{1}{p{1.84cm}|}{\centering 0010000} \\
 \cline{1-1} \cline{2-2} \cline{3-3}
 \multicolumn{1}{|p{7.13cm}|}{\centering 1112111+1112210} &  \multicolumn{1}{p{1.75cm}|}{\centering 15} &  \multicolumn{1}{p{1.84cm}|}{\centering 0000010} \\
 \cline{1-1} \cline{2-2} \cline{3-3}
 \multicolumn{1}{|p{7.13cm}|}{\centering 0112221+1112211} &  \multicolumn{1}{p{1.75cm}|}{\centering 15} &  \multicolumn{1}{p{1.84cm}|}{\centering 0000010} \\
 \cline{1-1} \cline{2-2} \cline{3-3}
 \multicolumn{1}{|p{7.13cm}|}{\centering 1112221} &  \multicolumn{1}{p{1.75cm}|}{\centering 10} &  \multicolumn{1}{p{1.84cm}|}{\centering 1000000} \\
 \cline{1-1} \cline{2-2} \cline{3-3}
 \multicolumn{1}{|p{7.13cm}|}{\centering 0000000} &  \multicolumn{1}{p{1.75cm}|}{\centering 0} &  \multicolumn{1}{p{1.84cm}|}{\centering 0000000} \\
 \cline{1-1} \cline{2-2} \cline{3-3}
\end{tabular}\end{center}

The action on the 15 dimensional $\frak u_2$ is the exterior square action of $GL(6)$, which arises for $SO(6,6)$, node 6.  This latter action has 4 orbits, of dimensions 15, 14, 9, and 0, and basepoints for the nontrivial orbits there are 001211 + 011111 + 111101, 012211 + 111211,  and 122211,  respectively, and correspond to the following orbits here:
\begin{center}\begin{tabular}{||p{4.59cm}p{1.84cm}||p{2.34cm}||}
\hline
 \multicolumn{1}{|p{4.59cm}|}{\centering Orbit Basepoint} &  \multicolumn{1}{p{1.84cm}|}{\centering Dimension} &  \multicolumn{1}{p{2.34cm}|}{\centering Coadjoint orbit intersected} \\
\hline
 \multicolumn{1}{|p{4.59cm}|}{\centering 1122221+1123211+1223210} &  \multicolumn{1}{p{1.84cm}|}{\centering 15} &  \multicolumn{1}{p{2.34cm}|}{\centering 0010000} \\
 \cline{1-1} \cline{2-2} \cline{3-3}
 \multicolumn{1}{|p{4.59cm}|}{\centering 1123321+1223221} &  \multicolumn{1}{p{1.84cm}|}{\centering 14} &  \multicolumn{1}{p{2.34cm}|}{\centering 0000010} \\
 \cline{1-1} \cline{2-2} \cline{3-3}
 \multicolumn{1}{|p{4.59cm}|}{\centering 1224321} &  \multicolumn{1}{p{1.84cm}|}{\centering 9} &  \multicolumn{1}{p{2.34cm}|}{\centering 1000000} \\
 \cline{1-1} \cline{2-2} \cline{3-3}
 \multicolumn{1}{|p{4.59cm}|}{\centering 0000000} &  \multicolumn{1}{p{1.84cm}|}{\centering 0} &  \multicolumn{1}{p{2.34cm}|}{\centering 0000000} \\
\hline
\end{tabular}\end{center}

The action on $\frak u_3$ is the standard action of $SL(2)$, and has 2 orbits:~zero and non-zero.  A representative for the big orbit is the highest root $2234321$, which lies in the minimal coadjoint nilpotent orbit $1000000$.

\subsubsection{$E_7$ Node 4}\label{sec:E7node4}

In this case $P=LU$ where $U$ is a 53-dimensional 4-step nilpotent group and $\frak u=\frak u_1\oplus \frak u_2\oplus \frak u_3\oplus\frak u_4$, with $\dim\frak u_1=24$, $\dim\frak u_2=18$,  $\dim\frak u_3=8$, and $\dim\frak u_4=3$.  The semisimple part $[L,L]$ of $L$ is of type $SL(3)\times SL(2)\times SL(4)$.  Its action on $\frak u_1$ is the tensor product of the standard representations of its three factors.

\begin{center}\begin{tabular}{||p{6.95cm}p{1.81cm}||p{1.91cm}||}
\hline
 \multicolumn{1}{|p{6.95cm}|}{\centering Orbit Basepoint} &  \multicolumn{1}{p{1.81cm}|}{\centering Dimension} &  \multicolumn{1}{p{1.91cm}|}{\centering Coadjoint orbit intersected} \\
\hline
 \multicolumn{1}{|p{6.95cm}|}{\centering $\nrel{0001111+0011110+0101110+}{+0111100+1011100+1111000}$} &  \multicolumn{1}{p{1.81cm}|}{\centering 24} &  \multicolumn{1}{p{1.91cm}|}{\centering 0002000} \\
 \cline{1-1} \cline{2-2} \cline{3-3}
 \multicolumn{1}{|p{6.95cm}|}{\centering $\nrel{0011111+0101110+0111100+}{+1011100+1111000}$} &  \multicolumn{1}{p{1.81cm}|}{\centering 23} &  \multicolumn{1}{p{1.91cm}|}{\centering 1001010} \\
 \cline{1-1} \cline{2-2} \cline{3-3}
 \multicolumn{1}{|p{6.95cm}|}{\centering 0001111+0111110+1011100+1111000} &  \multicolumn{1}{p{1.81cm}|}{\centering 22} &  \multicolumn{1}{p{1.91cm}|}{\centering 2000020} \\
 \cline{1-1} \cline{2-2} \cline{3-3}
 \multicolumn{1}{|p{6.95cm}|}{\centering 0001111+0101111+0111110+1011100} &  \multicolumn{1}{p{1.81cm}|}{\centering 21} &  \multicolumn{1}{p{1.91cm}|}{\centering 0020000} \\
 \cline{1-1} \cline{2-2} \cline{3-3}
 \multicolumn{1}{|p{6.95cm}|}{\centering $\nrel{0011111+0101111+0111100+}{+1011110+1111000}$} &  \multicolumn{1}{p{1.81cm}|}{\centering 21} &  \multicolumn{1}{p{1.91cm}|}{\centering 0001010} \\
 \cline{1-1} \cline{2-2} \cline{3-3}
 \multicolumn{1}{|p{6.95cm}|}{\centering 0011111+0101110+1011110+1111100} &  \multicolumn{1}{p{1.81cm}|}{\centering 20} &  \multicolumn{1}{p{1.91cm}|}{\centering 1001000} \\
 \cline{1-1} \cline{2-2} \cline{3-3}
 \multicolumn{1}{|p{6.95cm}|}{\centering 0101110+0111100+1011111+1111000} &  \multicolumn{1}{p{1.81cm}|}{\centering 19} &  \multicolumn{1}{p{1.91cm}|}{\centering 1001000} \\
 \cline{1-1} \cline{2-2} \cline{3-3}
 \multicolumn{1}{|p{6.95cm}|}{\centering $\nrel{0011111+0101111+0111110+}{+1011110+1111100}$} &  \multicolumn{1}{p{1.81cm}|}{\centering 19} &  \multicolumn{1}{p{1.91cm}|}{\centering 0010010} \\
 \cline{1-1} \cline{2-2} \cline{3-3}
 \multicolumn{1}{|p{6.95cm}|}{\centering 0011111+0111110+1011100+1111000} &  \multicolumn{1}{p{1.81cm}|}{\centering 18} &  \multicolumn{1}{p{1.91cm}|}{\centering 0000020} \\
 \cline{1-1} \cline{2-2} \cline{3-3}
 \multicolumn{1}{|p{6.95cm}|}{\centering 0101110+0111100+1011111} &  \multicolumn{1}{p{1.81cm}|}{\centering 18} &  \multicolumn{1}{p{1.91cm}|}{\centering 2000010} \\
 \cline{1-1} \cline{2-2} \cline{3-3}
 \multicolumn{1}{|p{6.95cm}|}{\centering 0011111+0101111+1011110+1111100} &  \multicolumn{1}{p{1.81cm}|}{\centering 17} &  \multicolumn{1}{p{1.91cm}|}{\centering 0000020} \\
 \cline{1-1} \cline{2-2} \cline{3-3}
 \multicolumn{1}{|p{6.95cm}|}{\centering 0011111+0111110+1011110+1111100} &  \multicolumn{1}{p{1.81cm}|}{\centering 17} &  \multicolumn{1}{p{1.91cm}|}{\centering 0001000} \\
 \cline{1-1} \cline{2-2} \cline{3-3}
 \multicolumn{1}{|p{6.95cm}|}{\centering 0101111+0111110+1011111+1111100} &  \multicolumn{1}{p{1.81cm}|}{\centering 17} &  \multicolumn{1}{p{1.91cm}|}{\centering 0001000} \\
 \cline{1-1} \cline{2-2} \cline{3-3}
 \multicolumn{1}{|p{6.95cm}|}{\centering 0011111+0101111+0111110+1011110} &  \multicolumn{1}{p{1.81cm}|}{\centering 16} &  \multicolumn{1}{p{1.91cm}|}{\centering 0001000} \\
 \cline{1-1} \cline{2-2} \cline{3-3}
 \multicolumn{1}{|p{6.95cm}|}{\centering 0111110+1011111+1111100} &  \multicolumn{1}{p{1.81cm}|}{\centering 16} &  \multicolumn{1}{p{1.91cm}|}{\centering 1000010} \\
 \cline{1-1} \cline{2-2} \cline{3-3}
 \multicolumn{1}{|p{6.95cm}|}{\centering 0011111+0101111+1111110} &  \multicolumn{1}{p{1.81cm}|}{\centering 15} &  \multicolumn{1}{p{1.91cm}|}{\centering 1000010} \\
 \cline{1-1} \cline{2-2} \cline{3-3}
 \multicolumn{1}{|p{6.95cm}|}{\centering 0011111+1111110} &  \multicolumn{1}{p{1.81cm}|}{\centering 14} &  \multicolumn{1}{p{1.91cm}|}{\centering 2000000} \\
 \cline{1-1} \cline{2-2} \cline{3-3}
 \multicolumn{1}{|p{6.95cm}|}{\centering 0101111+0111110+1111100} &  \multicolumn{1}{p{1.81cm}|}{\centering 13} &  \multicolumn{1}{p{1.91cm}|}{\centering 0010000} \\
 \cline{1-1} \cline{2-2} \cline{3-3}
 \multicolumn{1}{|p{6.95cm}|}{\centering 0111111+1011111+1111110} &  \multicolumn{1}{p{1.81cm}|}{\centering 13} &  \multicolumn{1}{p{1.91cm}|}{\centering 0010000} \\
 \cline{1-1} \cline{2-2} \cline{3-3}
 \multicolumn{1}{|p{6.95cm}|}{\centering 0111111+1111110} &  \multicolumn{1}{p{1.81cm}|}{\centering 11} &  \multicolumn{1}{p{1.91cm}|}{\centering 0000010} \\
 \cline{1-1} \cline{2-2} \cline{3-3}
 \multicolumn{1}{|p{6.95cm}|}{\centering 1011111+1111110} &  \multicolumn{1}{p{1.81cm}|}{\centering 10} &  \multicolumn{1}{p{1.91cm}|}{\centering 0000010} \\
 \cline{1-1} \cline{2-2} \cline{3-3}
 \multicolumn{1}{|p{6.95cm}|}{\centering 0111111+1011111} &  \multicolumn{1}{p{1.81cm}|}{\centering 9} &  \multicolumn{1}{p{1.91cm}|}{\centering 0000010} \\
 \cline{1-1} \cline{2-2} \cline{3-3}
 \multicolumn{1}{|p{6.95cm}|}{\centering 1111111} &  \multicolumn{1}{p{1.81cm}|}{\centering 7} &  \multicolumn{1}{p{1.91cm}|}{\centering 1000000} \\
 \cline{1-1} \cline{2-2} \cline{3-3}
 \multicolumn{1}{|p{6.95cm}|}{\centering 0000000} &  \multicolumn{1}{p{1.81cm}|}{\centering 0} &  \multicolumn{1}{p{1.91cm}|}{\centering 0000000} \\
 \cline{1-1} \cline{2-2} \cline{3-3}
\end{tabular}\end{center}

The action on the 18-dimensional $\frak u_2$ is the tensor product action of the standard action of $SL(3)$ factor with the exterior square representation of $SL(4)$ factor.  It arises for $SO(6,6)$, node 3, and has 11 orbits.
The nontrivial ones have dimensions 18, 17, 15, 14, 13, 12, 12, 11, 8, 7, with respective basepoints 001111 + 011101 + 011110 + 111100, 001211 + 011101 + 111110, 001211 + 011111 + 111101 + 111110, 011101 + 111110, 011211 + 111101 + 111110, 001211 + 011111 + 111101, 001211 + 011111 + 111110, 011211 + 111111, 111101 + 111110, 111211 there.  The basepoints and orbits here on $\frak u_2$  are thus given by the following table:
\begin{center}\begin{tabular}{||p{6.09cm}p{1.88cm}||p{2.22cm}||}
\hline
 \multicolumn{1}{|p{6.09cm}|}{\centering Orbit Basepoint} &  \multicolumn{1}{p{1.88cm}|}{\centering Dimension} &  \multicolumn{1}{p{2.22cm}|}{\centering Coadjoint orbit intersected} \\
\hline
 \multicolumn{1}{|p{6.09cm}|}{\centering 0112211+1112111+1112210+1122110} &  \multicolumn{1}{p{1.88cm}|}{\centering 18} &  \multicolumn{1}{p{2.22cm}|}{\centering 0001000} \\
 \cline{1-1} \cline{2-2} \cline{3-3}
 \multicolumn{1}{|p{6.09cm}|}{\centering 1112221+1112111+1122210} &  \multicolumn{1}{p{1.88cm}|}{\centering 17} &  \multicolumn{1}{p{2.22cm}|}{\centering 2000000} \\
 \cline{1-1} \cline{2-2} \cline{3-3}
 \multicolumn{1}{|p{6.09cm}|}{\centering 0112221+1112211+1122111+1122210} &  \multicolumn{1}{p{1.88cm}|}{\centering 15} &  \multicolumn{1}{p{2.22cm}|}{\centering 0100001} \\
 \cline{1-1} \cline{2-2} \cline{3-3}
 \multicolumn{1}{|p{6.09cm}|}{\centering 1112111+1122210} &  \multicolumn{1}{p{1.88cm}|}{\centering 14} &  \multicolumn{1}{p{2.22cm}|}{\centering 2000000} \\
 \cline{1-1} \cline{2-2} \cline{3-3}
 \multicolumn{1}{|p{6.09cm}|}{\centering 1112221+1122111+1122210} &  \multicolumn{1}{p{1.88cm}|}{\centering 13} &  \multicolumn{1}{p{2.22cm}|}{\centering 0010000} \\
 \cline{1-1} \cline{2-2} \cline{3-3}
 \multicolumn{1}{|p{6.09cm}|}{\centering 0112221+1112211+1122111} &  \multicolumn{1}{p{1.88cm}|}{\centering 12} &  \multicolumn{1}{p{2.22cm}|}{\centering 0000002} \\
 \cline{1-1} \cline{2-2} \cline{3-3}
 \multicolumn{1}{|p{6.09cm}|}{\centering 0112221+1112211+1122210} &  \multicolumn{1}{p{1.88cm}|}{\centering 12} &  \multicolumn{1}{p{2.22cm}|}{\centering 0010000} \\
 \cline{1-1} \cline{2-2} \cline{3-3}
 \multicolumn{1}{|p{6.09cm}|}{\centering 1112221+1122211} &  \multicolumn{1}{p{1.88cm}|}{\centering 11} &  \multicolumn{1}{p{2.22cm}|}{\centering 0000010} \\
 \cline{1-1} \cline{2-2} \cline{3-3}
 \multicolumn{1}{|p{6.09cm}|}{\centering 1122111+1122210} &  \multicolumn{1}{p{1.88cm}|}{\centering 8} &  \multicolumn{1}{p{2.22cm}|}{\centering 0000010} \\
 \cline{1-1} \cline{2-2} \cline{3-3}
 \multicolumn{1}{|p{6.09cm}|}{\centering 1122221} &  \multicolumn{1}{p{1.88cm}|}{\centering 7} &  \multicolumn{1}{p{2.22cm}|}{\centering 1000000} \\
 \cline{1-1} \cline{2-2} \cline{3-3}
 \multicolumn{1}{|p{6.09cm}|}{\centering 0000000} &  \multicolumn{1}{p{1.88cm}|}{\centering 0} &  \multicolumn{1}{p{2.22cm}|}{\centering 0000000} \\
\hline
\end{tabular}\end{center}

The action on the 8 dimensional $\frak u_3$ is the tensor product of the standard representations of the $SL(2)$ and $SL(4)$ factors, and arises for $SL(6)$, node 2.  It thus has 3 orbits, classified by rank.  These have dimensions 8, 5, and 0, with basepoints 01111+11110, 11111, and 00000 there, respectively.  The orbits here on $\frak u_3$ are given as follows:

\begin{center}\begin{tabular}{|cp{1.91cm}||p{2.25cm}||}
\hline
 \multicolumn{1}{|c|}{Orbit Basepoint} &  \multicolumn{1}{p{1.91cm}|}{\centering Dimension} &  \multicolumn{1}{p{2.25cm}|}{\centering Coadjoint orbit intersected} \\
\hline
 \multicolumn{1}{|c|}{1123321+1223221} &  \multicolumn{1}{p{1.91cm}|}{\centering 8} &  \multicolumn{1}{p{2.25cm}|}{\centering 0000010} \\
 \cline{1-1} \cline{2-2} \cline{3-3}
 \multicolumn{1}{|c|}{1223321} &  \multicolumn{1}{p{1.91cm}|}{\centering 5} &  \multicolumn{1}{p{2.25cm}|}{\centering 1000000} \\
 \cline{1-1} \cline{2-2} \cline{3-3}
 \multicolumn{1}{|c|}{0000000} &  \multicolumn{1}{p{1.91cm}|}{\centering 0} &  \multicolumn{1}{p{2.25cm}|}{\centering 0000000} \\
\hline
\end{tabular}\end{center}

The action on  $\frak u_4$ is the standard representation of $GL(3)$ and has 2 orbits:~zero and nonzero.  A representative for the big orbit is the highest root $2234321$, which lies in the minimal coadjoint nilpotent orbit $1000000$.

\subsubsection{$E_7$ Node 5}\label{sec:E7node5}

In this case $P=LU$ where $U$ is a 50-dimensional 3-step nilpotent group and $\frak u=\frak u_1\oplus \frak u_2\oplus \frak u_3$, with $\dim\frak u_1=30$, $\dim\frak u_2=15$,  and $\dim\frak u_3=5$.  The semisimple part $[L,L]$ of $L$ is of type $SL(5)\times SL(3)$.  Its action on $\frak u_1$ is the tensor product of the exterior square representation of the $SL(5)$ factor with the standard representation of the $SL(3)$ factor, and has the following orbits:

\begin{center}\begin{tabular}{||p{7.37cm}p{1.72cm}p{1.63cm}}
 \cline{1-1} \cline{2-2} \cline{3-3}
 \multicolumn{1}{|p{7.37cm}|}{\centering Orbit Basepoint} &  \multicolumn{1}{p{1.72cm}|}{\centering Dimension} &  \multicolumn{1}{p{1.63cm}|}{\centering Coadjoint orbit intersected} \\
\cline{1-1} \cline{2-2} \cline{3-3}
 \multicolumn{1}{|p{7.37cm}|}{\centering $\nrel{0011111+0101111+0111110+}{+0112100+1011110+1111100}$} &  \multicolumn{1}{p{1.72cm}|}{\centering 30} &  \multicolumn{1}{p{1.63cm}|}{\centering 0000200} \\
 \cline{1-1} \cline{2-2} \cline{3-3}
 \multicolumn{1}{|p{7.37cm}|}{\centering 0011111+0101111+0111110+1011110+1112100} &  \multicolumn{1}{p{1.72cm}|}{\centering 29} &  \multicolumn{1}{p{1.63cm}|}{\centering 0001010} \\
 \cline{1-1} \cline{2-2} \cline{3-3}
 \multicolumn{1}{|p{7.37cm}|}{\centering 0011111+0101111+0111110+1011111+1112100} &  \multicolumn{1}{p{1.72cm}|}{\centering 28} &  \multicolumn{1}{p{1.63cm}|}{\centering 0110001} \\
 \cline{1-1} \cline{2-2} \cline{3-3}
 \multicolumn{1}{|p{7.37cm}|}{\centering 0011111+0111110+1011111+1112100} &  \multicolumn{1}{p{1.72cm}|}{\centering 27} &  \multicolumn{1}{p{1.63cm}|}{\centering 0020000} \\
 \cline{1-1} \cline{2-2} \cline{3-3}
 \multicolumn{1}{|p{7.37cm}|}{\centering 0011111+0101111+0112110+1011110+1122100} &  \multicolumn{1}{p{1.72cm}|}{\centering 27} &  \multicolumn{1}{p{1.63cm}|}{\centering 1000101} \\
 \cline{1-1} \cline{2-2} \cline{3-3}
 \multicolumn{1}{|p{7.37cm}|}{\centering 0111111+0112110+1011110+1112100} &  \multicolumn{1}{p{1.72cm}|}{\centering 26} &  \multicolumn{1}{p{1.63cm}|}{\centering 1001000} \\
 \cline{1-1} \cline{2-2} \cline{3-3}
 \multicolumn{1}{|p{7.37cm}|}{\centering 0101111+0112110+1011111+1111110+1122100} &  \multicolumn{1}{p{1.72cm}|}{\centering 25} &  \multicolumn{1}{p{1.63cm}|}{\centering 0010010} \\
 \cline{1-1} \cline{2-2} \cline{3-3}
 \multicolumn{1}{|p{7.37cm}|}{\centering 0011111+0101111+1011110+1122100} &  \multicolumn{1}{p{1.72cm}|}{\centering 24} &  \multicolumn{1}{p{1.63cm}|}{\centering 2000002} \\
 \cline{1-1} \cline{2-2} \cline{3-3}
 \multicolumn{1}{|p{7.37cm}|}{\centering 0112111+1011110+1111100} &  \multicolumn{1}{p{1.72cm}|}{\centering 23} &  \multicolumn{1}{p{1.63cm}|}{\centering 2000010} \\
 \cline{1-1} \cline{2-2} \cline{3-3}
 \multicolumn{1}{|p{7.37cm}|}{\centering 0011111+0101111+1112110+1122100} &  \multicolumn{1}{p{1.72cm}|}{\centering 23} &  \multicolumn{1}{p{1.63cm}|}{\centering 0000020} \\
 \cline{1-1} \cline{2-2} \cline{3-3}
 \multicolumn{1}{|p{7.37cm}|}{\centering 0111111+0112110+1011111+1111110+1122100} &  \multicolumn{1}{p{1.72cm}|}{\centering 23} &  \multicolumn{1}{p{1.63cm}|}{\centering 0200000} \\
 \cline{1-1} \cline{2-2} \cline{3-3}
 \multicolumn{1}{|p{7.37cm}|}{\centering 0101111+0112110+1111110+1122100} &  \multicolumn{1}{p{1.72cm}|}{\centering 22} &  \multicolumn{1}{p{1.63cm}|}{\centering 0001000} \\
 \cline{1-1} \cline{2-2} \cline{3-3}
 \multicolumn{1}{|p{7.37cm}|}{\centering 0111111+0112110+1011111+1111110} &  \multicolumn{1}{p{1.72cm}|}{\centering 22} &  \multicolumn{1}{p{1.63cm}|}{\centering 0001000} \\
 \cline{1-1} \cline{2-2} \cline{3-3}
 \multicolumn{1}{|p{7.37cm}|}{\centering 0111111+1011111+1112110+1122100} &  \multicolumn{1}{p{1.72cm}|}{\centering 22} &  \multicolumn{1}{p{1.63cm}|}{\centering 0001000} \\
 \cline{1-1} \cline{2-2} \cline{3-3}
 \multicolumn{1}{|p{7.37cm}|}{\centering 0111111+1112110+1122100} &  \multicolumn{1}{p{1.72cm}|}{\centering 21} &  \multicolumn{1}{p{1.63cm}|}{\centering 1000010} \\
 \cline{1-1} \cline{2-2} \cline{3-3}
 \multicolumn{1}{|p{7.37cm}|}{\centering 0101111+1011111+1122110} &  \multicolumn{1}{p{1.72cm}|}{\centering 20} &  \multicolumn{1}{p{1.63cm}|}{\centering 1000010} \\
 \cline{1-1} \cline{2-2} \cline{3-3}
 \multicolumn{1}{|p{7.37cm}|}{\centering 0112111+1111111+1112110+1122100} &  \multicolumn{1}{p{1.72cm}|}{\centering 19} &  \multicolumn{1}{p{1.63cm}|}{\centering 0100001} \\
 \cline{1-1} \cline{2-2} \cline{3-3}
 \multicolumn{1}{|p{7.37cm}|}{\centering 0101111+1122110} &  \multicolumn{1}{p{1.72cm}|}{\centering 18} &  \multicolumn{1}{p{1.63cm}|}{\centering 2000000} \\
 \cline{1-1} \cline{2-2} \cline{3-3}
 \multicolumn{1}{|p{7.37cm}|}{\centering 0112111+1111111+1122110} &  \multicolumn{1}{p{1.72cm}|}{\centering 17} &  \multicolumn{1}{p{1.63cm}|}{\centering 0010000} \\
 \cline{1-1} \cline{2-2} \cline{3-3}
 \multicolumn{1}{|p{7.37cm}|}{\centering 1111111+1112110+1122100} &  \multicolumn{1}{p{1.72cm}|}{\centering 16} &  \multicolumn{1}{p{1.63cm}|}{\centering 0010000} \\
 \cline{1-1} \cline{2-2} \cline{3-3}
 \multicolumn{1}{|p{7.37cm}|}{\centering 0112111+1112110+1122100} &  \multicolumn{1}{p{1.72cm}|}{\centering 15} &  \multicolumn{1}{p{1.63cm}|}{\centering 0000002} \\
 \cline{1-1} \cline{2-2} \cline{3-3}
 \multicolumn{1}{|p{7.37cm}|}{\centering 1112111+1122110} &  \multicolumn{1}{p{1.72cm}|}{\centering 14} &  \multicolumn{1}{p{1.63cm}|}{\centering 0000010} \\
 \cline{1-1} \cline{2-2} \cline{3-3}
 \multicolumn{1}{|p{7.37cm}|}{\centering 0112111+1111111} &  \multicolumn{1}{p{1.72cm}|}{\centering 12} &  \multicolumn{1}{p{1.63cm}|}{\centering 0000010} \\
 \cline{1-1} \cline{2-2} \cline{3-3}
 \multicolumn{1}{|p{7.37cm}|}{\centering 1122111} &  \multicolumn{1}{p{1.72cm}|}{\centering 9} &  \multicolumn{1}{p{1.63cm}|}{\centering 1000000} \\
 \cline{1-1} \cline{2-2} \cline{3-3}
 \multicolumn{1}{|p{7.37cm}|}{\centering 0000000} &  \multicolumn{1}{p{1.72cm}|}{\centering 0} &  \multicolumn{1}{p{1.63cm}|}{\centering 0000000} \\
\hline
\end{tabular}\end{center}

The action on the 5-dimensional piece $\frak u_2$ is the tensor product of the standard representations of the two factors, and occurs for $SL(8)$, node 3.  It has 4 orbits, classified by rank, having dimensions 15, 12, 7, and 0 with respective basepoints 0011111 + 0111110 + 1111100, 0111111 + 1111110, 1111111, and 0000000 there.  Thus the orbits here on $\frak u_2$ are given by
\begin{center}\begin{tabular}{||p{4.41cm}p{1.81cm}||p{1.91cm}||}
\hline
 \multicolumn{1}{|p{4.41cm}|}{\centering Orbit Basepoint} &  \multicolumn{1}{p{1.81cm}|}{\centering Dimension} &  \multicolumn{1}{p{1.91cm}|}{\centering Coadjoint orbit intersected} \\
\hline
 \multicolumn{1}{|p{4.41cm}|}{\centering 1223210+1123211+1122221} &  \multicolumn{1}{p{1.81cm}|}{\centering 15} &  \multicolumn{1}{p{1.91cm}|}{\centering 0010000} \\
 \cline{1-1} \cline{2-2} \cline{3-3}
 \multicolumn{1}{|p{4.41cm}|}{\centering 1223211+1123221} &  \multicolumn{1}{p{1.81cm}|}{\centering 12} &  \multicolumn{1}{p{1.91cm}|}{\centering 0000010} \\
 \cline{1-1} \cline{2-2} \cline{3-3}
 \multicolumn{1}{|p{4.41cm}|}{\centering 1223221} &  \multicolumn{1}{p{1.81cm}|}{\centering 7} &  \multicolumn{1}{p{1.91cm}|}{\centering 1000000} \\
 \cline{1-1} \cline{2-2} \cline{3-3}
 \multicolumn{1}{|p{4.41cm}|}{\centering 0000000} &  \multicolumn{1}{p{1.81cm}|}{\centering 0} &  \multicolumn{1}{p{1.91cm}|}{\centering 0000000} \\
\hline
\end{tabular}\end{center}

The action on the 15-dimensional piece $\frak u_3$ is the standard action of $SL(5)$, and has 2 orbits:~zero and nonzero.  A representative for the big orbit is the highest root $2234321$, which lies in the minimal coadjoint nilpotent orbit $1000000$.

\subsubsection{$E_7$ Node 6}\label{sec:E7node6}

In this case $P=LU$ where $U$ is a 42-dimensional 2-step nilpotent group and $\frak u=\frak u_1\oplus \frak u_2$, with $\dim\frak u_1=32$   and $\dim\frak u_2=10$.  The semisimple part $[L,L]$ of $L$ is of type $SO(5,5)\times SL(2)$.  Its action on $\frak u_1$ is the tensor product of the spin representation of the $SO(5,5)$ factor with the standard representation of the $SL(2)$ factor, and has the following orbits:

\begin{center}\begin{tabular}{|cp{1.84cm}p{1.97cm}}
 \cline{1-1} \cline{2-2} \cline{3-3}
 \multicolumn{1}{|c|}{Orbit Basepoint} &  \multicolumn{1}{p{1.84cm}|}{\centering Dimension} &  \multicolumn{1}{p{1.97cm}|}{\centering Coadjoint orbit intersected} \\
\cline{1-1} \cline{2-2} \cline{3-3}
 \multicolumn{1}{|c|}{0011111+0101111+1112210+1122110} &  \multicolumn{1}{p{1.84cm}|}{\centering 32} &  \multicolumn{1}{p{1.97cm}|}{\centering 0000020} \\
 \cline{1-1} \cline{2-2} \cline{3-3}
 \multicolumn{1}{|c|}{0112211+1112111+1112210+1122110} &  \multicolumn{1}{p{1.84cm}|}{\centering 31} &  \multicolumn{1}{p{1.97cm}|}{\centering 0001000} \\
 \cline{1-1} \cline{2-2} \cline{3-3}
 \multicolumn{1}{|c|}{0112211+1011111+1223210} &  \multicolumn{1}{p{1.84cm}|}{\centering 28} &  \multicolumn{1}{p{1.97cm}|}{\centering 1000010} \\
 \cline{1-1} \cline{2-2} \cline{3-3}
 \multicolumn{1}{|c|}{1011111+1223210} &  \multicolumn{1}{p{1.84cm}|}{\centering 24} &  \multicolumn{1}{p{1.97cm}|}{\centering 2000000} \\
 \cline{1-1} \cline{2-2} \cline{3-3}
 \multicolumn{1}{|c|}{1112211+1122111+1223210} &  \multicolumn{1}{p{1.84cm}|}{\centering 23} &  \multicolumn{1}{p{1.97cm}|}{\centering 0010000} \\
 \cline{1-1} \cline{2-2} \cline{3-3}
 \multicolumn{1}{|c|}{1123211+1223210} &  \multicolumn{1}{p{1.84cm}|}{\centering 19} &  \multicolumn{1}{p{1.97cm}|}{\centering 0000010} \\
 \cline{1-1} \cline{2-2} \cline{3-3}
 \multicolumn{1}{|c|}{1112211+1122111} &  \multicolumn{1}{p{1.84cm}|}{\centering 17} &  \multicolumn{1}{p{1.97cm}|}{\centering 0000010} \\
 \cline{1-1} \cline{2-2} \cline{3-3}
 \multicolumn{1}{|c|}{1223211} &  \multicolumn{1}{p{1.84cm}|}{\centering 12} &  \multicolumn{1}{p{1.97cm}|}{\centering 1000000} \\
 \cline{1-1} \cline{2-2} \cline{3-3}
 \multicolumn{1}{|c|}{0000000} &  \multicolumn{1}{p{1.84cm}|}{\centering 0} &  \multicolumn{1}{p{1.97cm}|}{\centering 0000000} \\
\hline
\end{tabular}\end{center}

The action on $\frak u_2$ is the 10-dimensional vector realization of $SO(5,5)$, and occurs  for $SO(6,6)$, node 1.  It has 2 nontrivial  orbits, of dimensions 10 and 9 with basepoints 111101 + 111110 and 122211, respectively in $\frak{so}(6,6)$.  The orbits are given as follows:
\begin{center}\begin{tabular}{|cp{1.84cm}p{1.97cm}}
 \cline{1-1} \cline{2-2} \cline{3-3}
 \multicolumn{1}{|c|}{0112221+ 2234321} &  \multicolumn{1}{p{1.84cm}|}{\centering 10} &  \multicolumn{1}{p{1.97cm}|}{\centering 0000010} \\
 \cline{1-1} \cline{2-2} \cline{3-3}
 \multicolumn{1}{|c|}{2234321} &  \multicolumn{1}{p{1.84cm}|}{\centering 9} &  \multicolumn{1}{p{1.97cm}|}{\centering 1000000} \\
 \cline{1-1} \cline{2-2} \cline{3-3}
 \multicolumn{1}{|c|}{0000000} &  \multicolumn{1}{p{1.84cm}|}{\centering 0} &  \multicolumn{1}{p{1.97cm}|}{\centering 0000000} \\
\hline
\end{tabular}\end{center}

\subsubsection{$E_7$ Node 7}\label{sec:E7node7}

This is the only situation where $P$ has an abelian unipotent radical $U$, which in this case is 27 dimensional.  The semisimple part $[L,L]$ of $L$ is of type $E_6$, and acts on $\frak u=\frak u_1$ by the minimal, 27-dimensional representation.  It has 3 orbits:
\begin{center}\begin{tabular}{|cp{1.81cm}p{2.44cm}}
 \cline{1-1} \cline{2-2} \cline{3-3}
 \multicolumn{1}{|c|}{Orbit Basepoint} &  \multicolumn{1}{p{1.81cm}|}{\centering Dimension} &  \multicolumn{1}{p{2.44cm}|}{\centering Coadjoint orbit intersected} \\
\cline{1-1} \cline{2-2} \cline{3-3}
 \multicolumn{1}{|c|}{0112221+1112211+1122111} &  \multicolumn{1}{p{1.81cm}|}{\centering 27} &  \multicolumn{1}{p{2.44cm}|}{\centering 0000002} \\
 \cline{1-1} \cline{2-2} \cline{3-3}
 \multicolumn{1}{|c|}{1123321+1223221} &  \multicolumn{1}{p{1.81cm}|}{\centering 26} &  \multicolumn{1}{p{2.44cm}|}{\centering 0000010} \\
 \cline{1-1} \cline{2-2} \cline{3-3}
 \multicolumn{1}{|c|}{2234321} &  \multicolumn{1}{p{1.81cm}|}{\centering 17} &  \multicolumn{1}{p{2.44cm}|}{\centering 1000000} \\
 \cline{1-1} \cline{2-2} \cline{3-3}
 \multicolumn{1}{|c|}{0000000} &  \multicolumn{1}{p{1.81cm}|}{\centering 0} &  \multicolumn{1}{p{2.44cm}|}{\centering 0000000} \\
\hline
\end{tabular}\end{center}

\subsection{Type $E_8$}
The following table lists the internal Chevalley modules, and also for which smaller groups and parabolics the higher graded ones also occur (aside from the standard actions of $SL(n)$).  The minimal representations of $E_6$ and $E_7$ are written as {\bf 27} and {\bf 56}, respectively.  The same labeling conventions used for $D_5$, $E_6$, and $E_7$ remain in effect here.
\begin{sidewaystable}
\begin{center}\begin{tabular}{||p{1.09cm}p{2.66cm}||p{3.32cm}||p{1.72cm}||p{1.75cm}||p{2.06cm}||p{1.56cm}||p{1.03cm}||}
\hline
 \multicolumn{1}{|p{1.09cm}|}{\centering Node} &  \multicolumn{1}{p{2.66cm}|}{\centering Type of  $[L,L]$} &  \multicolumn{1}{p{3.32cm}|}{\centering $i\,=\,1$} &  \multicolumn{1}{p{1.72cm}|}{\centering $i\,=\,2$} &  \multicolumn{1}{p{1.75cm}|}{\centering $i\,=\,3$} &  \multicolumn{1}{p{2.06cm}|}{\centering $i\,=\,4$} &  \multicolumn{1}{p{1.56cm}|}{\centering $i\,=\,5$} &  \multicolumn{1}{p{1.03cm}|}{\centering $i\,=\,6$} \\
\hline
 \multicolumn{1}{|p{1.09cm}|}{\centering 1} &  \multicolumn{1}{p{2.66cm}|}{\raggedright $SO(7,7)$} &  \multicolumn{1}{p{3.32cm}|}{\centering Spin} &  \multicolumn{1}{p{1.72cm}|}{\centering $D_8$ node 1} &  \multicolumn{1}{p{1.75cm}|}{\centering} &  \multicolumn{1}{p{2.06cm}|}{\centering} &  \multicolumn{1}{p{1.56cm}|}{\centering} &  \multicolumn{1}{p{1.03cm}|}{\centering} \\
 \multicolumn{1}{|p{1.09cm}|}{\centering} &  \multicolumn{1}{p{2.66cm}|}{\centering $\dim \frak u_i$} &  \multicolumn{1}{p{3.32cm}|}{\centering 64} &  \multicolumn{1}{p{1.72cm}|}{\centering 14} &  \multicolumn{1}{p{1.75cm}|}{\centering} &  \multicolumn{1}{p{2.06cm}|}{\centering} &  \multicolumn{1}{p{1.56cm}|}{\centering} &  \multicolumn{1}{p{1.03cm}|}{\centering} \\
 \multicolumn{1}{|p{1.09cm}|}{\centering} &  \multicolumn{1}{p{2.66cm}|}{\centering action} &  \multicolumn{1}{p{3.32cm}|}{\centering $\varpi_2$} &  \multicolumn{1}{p{1.72cm}|}{\centering $\varpi_8$} &  \multicolumn{1}{p{1.75cm}|}{\centering} &  \multicolumn{1}{p{2.06cm}|}{\centering} &  \multicolumn{1}{p{1.56cm}|}{\centering} &  \multicolumn{1}{p{1.03cm}|}{\centering} \\
\hline
 \multicolumn{1}{|p{1.09cm}|}{\centering 2} &  \multicolumn{1}{p{2.66cm}|}{\raggedright $SL(8)$} &  \multicolumn{1}{p{3.32cm}|}{\centering Ext. cube} &  \multicolumn{1}{p{1.72cm}|}{\centering $D_8$ node 8} &  \multicolumn{1}{p{1.75cm}|}{\centering} &  \multicolumn{1}{p{2.06cm}|}{\centering} &  \multicolumn{1}{p{1.56cm}|}{\centering} &  \multicolumn{1}{p{1.03cm}|}{\centering} \\
 \multicolumn{1}{|p{1.09cm}|}{\centering} &  \multicolumn{1}{p{2.66cm}|}{\centering $\dim \frak u_i$} &  \multicolumn{1}{p{3.32cm}|}{\centering 56} &  \multicolumn{1}{p{1.72cm}|}{\centering 28} &  \multicolumn{1}{p{1.75cm}|}{\centering 8} &  \multicolumn{1}{p{2.06cm}|}{\centering} &  \multicolumn{1}{p{1.56cm}|}{\centering} &  \multicolumn{1}{p{1.03cm}|}{\centering} \\
 \multicolumn{1}{|p{1.09cm}|}{\centering} &  \multicolumn{1}{p{2.66cm}|}{\centering action} &  \multicolumn{1}{p{3.32cm}|}{\centering $\varpi_6$} &  \multicolumn{1}{p{1.72cm}|}{\centering $\varpi_3$} &  \multicolumn{1}{p{1.75cm}|}{\centering $\varpi_8$} &  \multicolumn{1}{p{2.06cm}|}{\centering} &  \multicolumn{1}{p{1.56cm}|}{\centering} &  \multicolumn{1}{p{1.03cm}|}{\centering} \\
\hline
 \multicolumn{1}{|p{1.09cm}|}{\centering 3} &  \multicolumn{1}{p{2.66cm}|}{\raggedright $SL(2) \times SL(7)$} &  \multicolumn{1}{p{3.32cm}|}{\centering Stan.$\,\otimes$ Ext. sq.} &  \multicolumn{1}{p{1.72cm}|}{\centering $E_7$ node 2} &  \multicolumn{1}{p{1.75cm}|}{\centering $A_8$ node 7} &  \multicolumn{1}{p{2.06cm}|}{\centering} &  \multicolumn{1}{p{1.56cm}|}{\centering} &  \multicolumn{1}{p{1.03cm}|}{\centering} \\
 \multicolumn{1}{|p{1.09cm}|}{\centering} &  \multicolumn{1}{p{2.66cm}|}{\centering $\dim \frak u_i$} &  \multicolumn{1}{p{3.32cm}|}{\centering 42} &  \multicolumn{1}{p{1.72cm}|}{\centering 35} &  \multicolumn{1}{p{1.75cm}|}{\centering 14} &  \multicolumn{1}{p{2.06cm}|}{\centering 7} &  \multicolumn{1}{p{1.56cm}|}{\centering} &  \multicolumn{1}{p{1.03cm}|}{\centering} \\
 \multicolumn{1}{|p{1.09cm}|}{\centering} &  \multicolumn{1}{p{2.66cm}|}{\centering action} &  \multicolumn{1}{p{3.32cm}|}{\centering $\varpi_1+\varpi_7$} &  \multicolumn{1}{p{1.72cm}|}{\centering $\varpi_5$} &  \multicolumn{1}{p{1.75cm}|}{\centering $\varpi_1+\varpi_2$} &  \multicolumn{1}{p{2.06cm}|}{\centering $\varpi_8$} &  \multicolumn{1}{p{1.56cm}|}{\centering} &  \multicolumn{1}{p{1.03cm}|}{\centering} \\
\hline
 \multicolumn{1}{|p{1.09cm}|}{\centering 4} &  \multicolumn{1}{p{2.66cm}|}{\raggedright $A_2\times A_1 \times A_4$} &  \multicolumn{1}{p{3.32cm}|}{\centering Stan.$\,\otimes\,$Stan.$\,\otimes\,$Stan.} &  \multicolumn{1}{p{1.72cm}|}{\centering $E_7$ node 5} &  \multicolumn{1}{p{1.75cm}|}{\centering $E_6$ node 3} &  \multicolumn{1}{p{2.06cm}|}{\centering $A_7$ node 3} &  \multicolumn{1}{p{1.56cm}|}{\centering $A_4$ node 2} &  \multicolumn{1}{p{1.03cm}|}{\centering} \\
 \multicolumn{1}{|p{1.09cm}|}{\centering} &  \multicolumn{1}{p{2.66cm}|}{\centering $\dim \frak u_i$} &  \multicolumn{1}{p{3.32cm}|}{\centering 30} &  \multicolumn{1}{p{1.72cm}|}{\centering 30} &  \multicolumn{1}{p{1.75cm}|}{\centering 20} &  \multicolumn{1}{p{2.06cm}|}{\centering 15} &  \multicolumn{1}{p{1.56cm}|}{\centering 6} &  \multicolumn{1}{p{1.03cm}|}{\centering 5} \\
 \multicolumn{1}{|p{1.09cm}|}{\centering} &  \multicolumn{1}{p{2.66cm}|}{\centering action} &  \multicolumn{1}{p{3.32cm}|}{\centering $\varpi_1+\varpi_2+\varpi_8$} &  \multicolumn{1}{p{1.72cm}|}{\centering $\varpi_3+\varpi_7$} &  \multicolumn{1}{p{1.75cm}|}{\centering $\varpi_2+\varpi_6$} &  \multicolumn{1}{p{2.06cm}|}{\centering $\varpi_1+\varpi_5$} &  \multicolumn{1}{p{1.56cm}|}{\centering $\varpi_2+\varpi_3$} &  \multicolumn{1}{p{1.03cm}|}{\centering $\varpi_8$} \\
\hline
 \multicolumn{1}{|p{1.09cm}|}{\centering 5} &  \multicolumn{1}{p{2.66cm}|}{\raggedright $SL(5)\times SL(4)$} &  \multicolumn{1}{p{3.32cm}|}{\centering Ext. sq.$\,\otimes$ Stan. } &  \multicolumn{1}{p{1.72cm}|}{\centering $D_8$ node 5} &  \multicolumn{1}{p{1.75cm}|}{\centering $A_8$ node 4} &  \multicolumn{1}{p{2.06cm}|}{\centering $D_5$ node 5} &  \multicolumn{1}{p{1.56cm}|}{\centering} &  \multicolumn{1}{p{1.03cm}|}{\centering} \\
 \multicolumn{1}{|p{1.09cm}|}{\centering} &  \multicolumn{1}{p{2.66cm}|}{\centering $\dim \frak u_i$} &  \multicolumn{1}{p{3.32cm}|}{\centering 40} &  \multicolumn{1}{p{1.72cm}|}{\centering 30} &  \multicolumn{1}{p{1.75cm}|}{\centering 20} &  \multicolumn{1}{p{2.06cm}|}{\centering 10} &  \multicolumn{1}{p{1.56cm}|}{\centering 4} &  \multicolumn{1}{p{1.03cm}|}{\centering} \\
 \multicolumn{1}{|p{1.09cm}|}{\centering} &  \multicolumn{1}{p{2.66cm}|}{\centering action} &  \multicolumn{1}{p{3.32cm}|}{\centering $\varpi_3+\varpi_8$} &  \multicolumn{1}{p{1.72cm}|}{\centering $\varpi_2+\varpi_7$} &  \multicolumn{1}{p{1.75cm}|}{\centering $\varpi_1+\varpi_6$} &  \multicolumn{1}{p{2.06cm}|}{\centering $\varpi_4$} &  \multicolumn{1}{p{1.56cm}|}{\centering $\varpi_8$} &  \multicolumn{1}{p{1.03cm}|}{\centering} \\
\hline
 \multicolumn{1}{|p{1.09cm}|}{\centering 6} &  \multicolumn{1}{p{2.66cm}|}{\raggedright $SO(5,5)\times SL(3)$} &  \multicolumn{1}{p{3.32cm}|}{\centering Spin$\,\otimes$ Stan.} &  \multicolumn{1}{p{1.72cm}|}{\centering $D_8$ node 3} &  \multicolumn{1}{p{1.75cm}|}{\centering $E_6$ node 1} &  \multicolumn{1}{p{2.06cm}|}{\centering} &  \multicolumn{1}{p{1.56cm}|}{\centering} &  \multicolumn{1}{p{1.03cm}|}{\centering} \\
 \multicolumn{1}{|p{1.09cm}|}{\centering} &  \multicolumn{1}{p{2.66cm}|}{\centering $\dim \frak u_i$} &  \multicolumn{1}{p{3.32cm}|}{\centering 48} &  \multicolumn{1}{p{1.72cm}|}{\centering 30} &  \multicolumn{1}{p{1.75cm}|}{\centering 16} &  \multicolumn{1}{p{2.06cm}|}{\centering 3} &  \multicolumn{1}{p{1.56cm}|}{\centering} &  \multicolumn{1}{p{1.03cm}|}{\centering} \\
 \multicolumn{1}{|p{1.09cm}|}{\centering} &  \multicolumn{1}{p{2.66cm}|}{\centering action} &  \multicolumn{1}{p{3.32cm}|}{\centering $\varpi_2+\varpi_8$} &  \multicolumn{1}{p{1.72cm}|}{\centering $\varpi_1+\varpi_7$} &  \multicolumn{1}{p{1.75cm}|}{\centering $\varpi_5$} &  \multicolumn{1}{p{2.06cm}|}{\centering $\varpi_8$} &  \multicolumn{1}{p{1.56cm}|}{\centering} &  \multicolumn{1}{p{1.03cm}|}{\centering} \\
\hline
 \multicolumn{1}{|p{1.09cm}|}{\centering 7} &  \multicolumn{1}{p{2.66cm}|}{\raggedright $E_6\times SL(2)$} &  \multicolumn{1}{p{3.32cm}|}{\centering {\bf 27} $\otimes$ Stan.} &  \multicolumn{1}{p{1.72cm}|}{\centering $E_7$ node 7} &  \multicolumn{1}{p{1.75cm}|}{\centering} &  \multicolumn{1}{p{2.06cm}|}{\centering} &  \multicolumn{1}{p{1.56cm}|}{\centering} &  \multicolumn{1}{p{1.03cm}|}{\centering} \\
 \multicolumn{1}{|p{1.09cm}|}{\centering} &  \multicolumn{1}{p{2.66cm}|}{\centering $\dim \frak u_i$} &  \multicolumn{1}{p{3.32cm}|}{\centering 54} &  \multicolumn{1}{p{1.72cm}|}{\centering 27} &  \multicolumn{1}{p{1.75cm}|}{\centering 2} &  \multicolumn{1}{p{2.06cm}|}{\centering} &  \multicolumn{1}{p{1.56cm}|}{\centering} &  \multicolumn{1}{p{1.03cm}|}{\centering} \\
 \multicolumn{1}{|p{1.09cm}|}{\centering} &  \multicolumn{1}{p{2.66cm}|}{\centering action} &  \multicolumn{1}{p{3.32cm}|}{\centering $\varpi_1+\varpi_8$} &  \multicolumn{1}{p{1.72cm}|}{\centering $\varpi_6$} &  \multicolumn{1}{p{1.75cm}|}{\centering $\varpi_8$} &  \multicolumn{1}{p{2.06cm}|}{\centering} &  \multicolumn{1}{p{1.56cm}|}{\centering} &  \multicolumn{1}{p{1.03cm}|}{\centering} \\
\hline
 \multicolumn{1}{|p{1.09cm}|}{\centering 8} &  \multicolumn{1}{p{2.66cm}|}{\raggedright $E_7$} &  \multicolumn{1}{p{3.32cm}|}{\centering {\bf 56}} &  \multicolumn{1}{p{1.72cm}|}{\centering} &  \multicolumn{1}{p{1.75cm}|}{\centering} &  \multicolumn{1}{p{2.06cm}|}{\centering} &  \multicolumn{1}{p{1.56cm}|}{\centering} &  \multicolumn{1}{p{1.03cm}|}{\centering} \\
 \multicolumn{1}{|p{1.09cm}|}{\centering} &  \multicolumn{1}{p{2.66cm}|}{\centering $\dim \frak u_i$} &  \multicolumn{1}{p{3.32cm}|}{\centering 56} &  \multicolumn{1}{p{1.72cm}|}{\centering 1} &  \multicolumn{1}{p{1.75cm}|}{\centering} &  \multicolumn{1}{p{2.06cm}|}{\centering} &  \multicolumn{1}{p{1.56cm}|}{\centering} &  \multicolumn{1}{p{1.03cm}|}{\centering} \\
 \multicolumn{1}{|p{1.09cm}|}{\centering} &  \multicolumn{1}{p{2.66cm}|}{\centering action} &  \multicolumn{1}{p{3.32cm}|}{\centering $\varpi_7$} &  \multicolumn{1}{p{1.72cm}|}{\centering} &  \multicolumn{1}{p{1.75cm}|}{\centering} &  \multicolumn{1}{p{2.06cm}|}{\centering} &  \multicolumn{1}{p{1.56cm}|}{\centering} &  \multicolumn{1}{p{1.03cm}|}{\centering} \\
\hline
\end{tabular}\end{center}
\end{sidewaystable}

\subsubsection{$E_8$ Node 1}

In this case $P=LU$ where $U$ is a 78-dimensional 2-step nilpotent group and $\frak u=\frak u_1\oplus \frak u_2$, with $\dim\frak u_1=64$   and $\dim\frak u_2=14$.  The semisimple part $[L,L]$ of $L$ is of type $SO(7,7)$, which acts on $\frak u_1$ by the spin representation  with  the following orbits:

\begin{center}\begin{tabular}{p{7.2cm}p{1.63cm}p{1.84cm}}
 \cline{1-1} \cline{2-2} \cline{3-3}
 \multicolumn{1}{|p{7.2cm}|}{\centering Orbit Basepoint} &  \multicolumn{1}{p{1.63cm}|}{\centering Dimension} &  \multicolumn{1}{p{1.84cm}|}{\centering Coadjoint orbit intersected} \\
\cline{1-1} \cline{2-2} \cline{3-3}
 \multicolumn{1}{|p{7.2cm}|}{\centering 11122111+11221111+11233210+12232210} &  \multicolumn{1}{p{1.63cm}|}{\centering 64} &  \multicolumn{1}{p{1.84cm}|}{\centering 20000000} \\
 \cline{1-1} \cline{2-2} \cline{3-3}
 \multicolumn{1}{|p{7.2cm}|}{\centering $\nrel{11222221+11232211+11233210+}{+12232111+12232210}$} &  \multicolumn{1}{p{1.63cm}|}{\centering 63} &  \multicolumn{1}{p{1.84cm}|}{\centering 00100000} \\
 \cline{1-1} \cline{2-2} \cline{3-3}
 \multicolumn{1}{|p{7.2cm}|}{\centering 11122221+11233211+12232211+12343210} &  \multicolumn{1}{p{1.63cm}|}{\centering 59} &  \multicolumn{1}{p{1.84cm}|}{\centering 00000100} \\
 \cline{1-1} \cline{2-2} \cline{3-3}
 \multicolumn{1}{|p{7.2cm}|}{\centering 11222221+12243211+12343210} &  \multicolumn{1}{p{1.63cm}|}{\centering 54} &  \multicolumn{1}{p{1.84cm}|}{\centering 10000001} \\
 \cline{1-1} \cline{2-2} \cline{3-3}
 \multicolumn{1}{|p{7.2cm}|}{\centering 11233321+12233221+12243211+12343210} &  \multicolumn{1}{p{1.63cm}|}{\centering 50} &  \multicolumn{1}{p{1.84cm}|}{\centering 01000000} \\
 \cline{1-1} \cline{2-2} \cline{3-3}
 \multicolumn{1}{|p{7.2cm}|}{\centering 11122221+12343211} &  \multicolumn{1}{p{1.63cm}|}{\centering 44} &  \multicolumn{1}{p{1.84cm}|}{\centering 00000002} \\
 \cline{1-1} \cline{2-2} \cline{3-3}
 \multicolumn{1}{|p{7.2cm}|}{\centering 12233321+12243221+12343211} &  \multicolumn{1}{p{1.63cm}|}{\centering 43} &  \multicolumn{1}{p{1.84cm}|}{\centering 00000010} \\
 \cline{1-1} \cline{2-2} \cline{3-3}
 \multicolumn{1}{|p{7.2cm}|}{\centering 12244321+12343321} &  \multicolumn{1}{p{1.63cm}|}{\centering 35} &  \multicolumn{1}{p{1.84cm}|}{\centering 10000000} \\
 \cline{1-1} \cline{2-2} \cline{3-3}
 \multicolumn{1}{|p{7.2cm}|}{\centering 13354321} &  \multicolumn{1}{p{1.63cm}|}{\centering 22} &  \multicolumn{1}{p{1.84cm}|}{\centering 00000001} \\
 \cline{1-1} \cline{2-2} \cline{3-3}
 \multicolumn{1}{|p{7.2cm}|}{\centering 00000000} &  \multicolumn{1}{p{1.63cm}|}{\centering 0} &  \multicolumn{1}{p{1.84cm}|}{\centering 00000000} \\
 \cline{1-1} \cline{2-2} \cline{3-3}
\end{tabular}\end{center}

The action on the 14-dimensional $\frak u_2$ is the vector representation of $SO(7,7)$, and occurs for $SO(8,8)$, node 1.  It has 3 orbits there, of dimensions 14, 13, and 0 and respective basepoints 11111101 + 11111110, 12222211, and 00000000 (in ${\frak{so}}(8,8)$).  This translates into the following orbits in  $\frak u_2$:
\begin{center}\begin{tabular}{|cp{1.66cm}||p{2.56cm}||}
\hline
 \multicolumn{1}{|c|}{Orbit Basepoint} &  \multicolumn{1}{p{1.66cm}|}{\centering Dimension} &  \multicolumn{1}{p{2.56cm}|}{\centering Coadjoint orbit intersected} \\
\hline
 \multicolumn{1}{|c|}{23354321+22454321} &  \multicolumn{1}{p{1.66cm}|}{\centering 14} &  \multicolumn{1}{p{2.56cm}|}{\centering 10000000} \\
 \cline{1-1} \cline{2-2} \cline{3-3}
 \multicolumn{1}{|c|}{23465432} &  \multicolumn{1}{p{1.66cm}|}{\centering 13} &  \multicolumn{1}{p{2.56cm}|}{\centering 00000001} \\
 \cline{1-1} \cline{2-2} \cline{3-3}
 \multicolumn{1}{|c|}{00000000} &  \multicolumn{1}{p{1.66cm}|}{\centering 0} &  \multicolumn{1}{p{2.56cm}|}{\centering 00000000} \\
\hline
\end{tabular}\end{center}

\subsubsection{$E_8$ Node 2}

In this case $P=LU$ where $U$ is a 92-dimensional 3-step nilpotent group and $\frak u=\frak u_1\oplus \frak u_2\oplus \frak u_3$, with $\dim\frak u_1=56$, $\dim\frak u_2=28$, and  $\dim\frak u_3=8$.  The semisimple part $[L,L]$ of $L$ is of type $SL(8)$, which acts on $\frak u_1$ as it does on antisymmetric 3 tensors.  It has the following orbits:
\begin{center}\begin{tabular}{p{7.2cm}p{1.53cm}p{1.75cm}}
 \cline{1-1} \cline{2-2} \cline{3-3}
 \multicolumn{1}{|p{7.2cm}|}{\centering Orbit Basepoint} &  \multicolumn{1}{p{1.53cm}|}{\centering Dimension} &  \multicolumn{1}{p{1.75cm}|}{\centering Coadjoint orbit intersected} \\
\cline{1-1} \cline{2-2} \cline{3-3}
 \multicolumn{1}{|p{7.2cm}|}{\centering $\nrel{01121111+01122111+01122210+}{+11111111+11122110+11221100}$} &  \multicolumn{1}{p{1.53cm}|}{\centering 56} &  \multicolumn{1}{p{1.75cm}|}{\centering 02000000} \\
 \cline{1-1} \cline{2-2} \cline{3-3}
 \multicolumn{1}{|p{7.2cm}|}{\centering $\nrel{01122111+01122210+11121111+}{+11122110+11221110+11222100}$} &  \multicolumn{1}{p{1.53cm}|}{\centering 55} &  \multicolumn{1}{p{1.75cm}|}{\centering 00010000} \\
 \cline{1-1} \cline{2-2} \cline{3-3}
 \multicolumn{1}{|p{7.2cm}|}{\centering $\nrel{01122211+11111111+11122110+}{+11221110+11232100}$} &  \multicolumn{1}{p{1.53cm}|}{\centering 53} &  \multicolumn{1}{p{1.75cm}|}{\centering 10000100} \\
 \cline{1-1} \cline{2-2} \cline{3-3}
 \multicolumn{1}{|p{7.2cm}|}{\centering $\nrel{01122210+01122211+11122111+}{+11221110+11232100}$} &  \multicolumn{1}{p{1.53cm}|}{\centering 52} &  \multicolumn{1}{p{1.75cm}|}{\centering 01000010} \\
 \cline{1-1} \cline{2-2} \cline{3-3}
 \multicolumn{1}{|p{7.2cm}|}{\centering $\nrel{01122221+11111111+11122210+}{+11222110+11232100}$} &  \multicolumn{1}{p{1.53cm}|}{\centering 50} &  \multicolumn{1}{p{1.75cm}|}{\centering 00100001} \\
 \cline{1-1} \cline{2-2} \cline{3-3}
 \multicolumn{1}{|p{7.2cm}|}{\centering 01122210+01122211+11122111+11221110} &  \multicolumn{1}{p{1.53cm}|}{\centering 48} &  \multicolumn{1}{p{1.75cm}|}{\centering 00000020} \\
 \cline{1-1} \cline{2-2} \cline{3-3}
 \multicolumn{1}{|p{7.2cm}|}{\centering $\nrel{01122221+11122111+11122210+}{+11221111+11222110+11232100}$} &  \multicolumn{1}{p{1.53cm}|}{\centering 48} &  \multicolumn{1}{p{1.75cm}|}{\centering 00001000} \\
 \cline{1-1} \cline{2-2} \cline{3-3}
 \multicolumn{1}{|p{7.2cm}|}{\centering 01122211+11111111+11222210+11232110} &  \multicolumn{1}{p{1.53cm}|}{\centering 47} &  \multicolumn{1}{p{1.75cm}|}{\centering 00000101} \\
 \cline{1-1} \cline{2-2} \cline{3-3}
 \multicolumn{1}{|p{7.2cm}|}{\centering $\nrel{01122211+11122111+11221111+}{+11222210+11232110}$} &  \multicolumn{1}{p{1.53cm}|}{\centering 46} &  \multicolumn{1}{p{1.75cm}|}{\centering 10000010} \\
 \cline{1-1} \cline{2-2} \cline{3-3}
 \multicolumn{1}{|p{7.2cm}|}{\centering 01122211+11122111+11222210+11232110} &  \multicolumn{1}{p{1.53cm}|}{\centering 44} &  \multicolumn{1}{p{1.75cm}|}{\centering 20000000} \\
 \cline{1-1} \cline{2-2} \cline{3-3}
 \multicolumn{1}{|p{7.2cm}|}{\centering $\nrel{01122221+11122211+11222111+}{+11222210+11232110}$} &  \multicolumn{1}{p{1.53cm}|}{\centering 43} &  \multicolumn{1}{p{1.75cm}|}{\centering 00100000} \\
 \cline{1-1} \cline{2-2} \cline{3-3}
 \multicolumn{1}{|p{7.2cm}|}{\centering $\nrel{11122111+11122210+11221111+}{+11222110+11232100}$} &  \multicolumn{1}{p{1.53cm}|}{\centering 42} &  \multicolumn{1}{p{1.75cm}|}{\centering 00100000} \\
 \cline{1-1} \cline{2-2} \cline{3-3}
 \multicolumn{1}{|p{7.2cm}|}{\centering 01121111+11111111+11233210} &  \multicolumn{1}{p{1.53cm}|}{\centering 41} &  \multicolumn{1}{p{1.75cm}|}{\centering 10000002} \\
 \cline{1-1} \cline{2-2} \cline{3-3}
 \multicolumn{1}{|p{7.2cm}|}{\centering 11122211+11222111+11222210+11232110} &  \multicolumn{1}{p{1.53cm}|}{\centering 41} &  \multicolumn{1}{p{1.75cm}|}{\centering 00000100} \\
 \cline{1-1} \cline{2-2} \cline{3-3}
 \multicolumn{1}{|p{7.2cm}|}{\centering 01122221+11122211+11221111+11233210} &  \multicolumn{1}{p{1.53cm}|}{\centering 40} &  \multicolumn{1}{p{1.75cm}|}{\centering 00000100} \\
 \cline{1-1} \cline{2-2} \cline{3-3}
 \multicolumn{1}{|p{7.2cm}|}{\centering 11122221+11221111+11233210} &  \multicolumn{1}{p{1.53cm}|}{\centering 38} &  \multicolumn{1}{p{1.75cm}|}{\centering 10000001} \\
 \cline{1-1} \cline{2-2} \cline{3-3}
 \multicolumn{1}{|p{7.2cm}|}{\centering 11122221+11222211+11232111+11233210} &  \multicolumn{1}{p{1.53cm}|}{\centering 35} &  \multicolumn{1}{p{1.75cm}|}{\centering 01000000} \\
 \cline{1-1} \cline{2-2} \cline{3-3}
 \multicolumn{1}{|p{7.2cm}|}{\centering 11221111+11233210} &  \multicolumn{1}{p{1.53cm}|}{\centering 32} &  \multicolumn{1}{p{1.75cm}|}{\centering 00000002} \\
 \cline{1-1} \cline{2-2} \cline{3-3}
 \multicolumn{1}{|p{7.2cm}|}{\centering 11222221+11232211+11233210} &  \multicolumn{1}{p{1.53cm}|}{\centering 31} &  \multicolumn{1}{p{1.75cm}|}{\centering 00000010} \\
 \cline{1-1} \cline{2-2} \cline{3-3}
 \multicolumn{1}{|p{7.2cm}|}{\centering 11122221+11222211+11232111} &  \multicolumn{1}{p{1.53cm}|}{\centering 28} &  \multicolumn{1}{p{1.75cm}|}{\centering 00000010} \\
 \cline{1-1} \cline{2-2} \cline{3-3}
 \multicolumn{1}{|p{7.2cm}|}{\centering 11232221+11233211} &  \multicolumn{1}{p{1.53cm}|}{\centering 25} &  \multicolumn{1}{p{1.75cm}|}{\centering 10000000} \\
 \cline{1-1} \cline{2-2} \cline{3-3}
 \multicolumn{1}{|p{7.2cm}|}{\centering 11233321} &  \multicolumn{1}{p{1.53cm}|}{\centering 16} &  \multicolumn{1}{p{1.75cm}|}{\centering 00000001} \\
 \cline{1-1} \cline{2-2} \cline{3-3}
 \multicolumn{1}{|p{7.2cm}|}{\centering 00000000} &  \multicolumn{1}{p{1.53cm}|}{\centering 0} &  \multicolumn{1}{p{1.75cm}|}{\centering 00000000} \\
 \cline{1-1} \cline{2-2} \cline{3-3}
\end{tabular}\end{center}

The 28 dimensional action of $SL(8)$ on $\frak u_2$ is the exterior square action, which arises for $SO(8,8)$, node 8.  The action there has 5 orbits, of dimensions 28, 27, 22, 13, and 0, with respective basepoints 00012211 + 00111211 + 01111111 + 11111101, 00122211 + 01112211 + 11111211, 01222211 + 11122211, 12222211, and 00000000 (in ${\frak{so}}(8,8)$).  The orbits here are given by
\begin{center}\begin{tabular}{||p{6.72cm}p{1.69cm}||p{1.47cm}||}
\hline
 \multicolumn{1}{|p{6.72cm}|}{\centering Orbit Basepoint} &  \multicolumn{1}{p{1.69cm}|}{\centering Dimension} &  \multicolumn{1}{p{1.47cm}|}{\centering Coadjoint orbit intersected} \\
\hline
 \multicolumn{1}{|p{6.72cm}|}{\centering 12233321+12243221+12343211+22343210} &  \multicolumn{1}{p{1.69cm}|}{\centering 28} &  \multicolumn{1}{p{1.47cm}|}{\centering 01000000} \\
 \cline{1-1} \cline{2-2} \cline{3-3}
 \multicolumn{1}{|p{6.72cm}|}{\centering 12244321+12343321+22343221} &  \multicolumn{1}{p{1.69cm}|}{\centering 27} &  \multicolumn{1}{p{1.47cm}|}{\centering 00000010} \\
 \cline{1-1} \cline{2-2} \cline{3-3}
 \multicolumn{1}{|p{6.72cm}|}{\centering 12354321+22344321} &  \multicolumn{1}{p{1.69cm}|}{\centering 22} &  \multicolumn{1}{p{1.47cm}|}{\centering 10000000} \\
 \cline{1-1} \cline{2-2} \cline{3-3}
 \multicolumn{1}{|p{6.72cm}|}{\centering 22454321} &  \multicolumn{1}{p{1.69cm}|}{\centering 13} &  \multicolumn{1}{p{1.47cm}|}{\centering 00000001} \\
 \cline{1-1} \cline{2-2} \cline{3-3}
 \multicolumn{1}{|p{6.72cm}|}{\centering 00000000} &  \multicolumn{1}{p{1.69cm}|}{\centering 0} &  \multicolumn{1}{p{1.47cm}|}{\centering 00000000} \\
\hline
\end{tabular}\end{center}

The action of $SL(8)$ on $\frak u_3$ is its standard action, and has 2 orbits:~zero and nonzero.

\subsubsection{$E_8$ Node 3}

In this case $P=LU$ where $U$ is a 98-dimensional 4-step nilpotent group and $\frak u=\frak u_1\oplus \frak u_2\oplus \frak u_3\oplus \frak u_4$, with $\dim\frak u_1=42$, $\dim\frak u_2=35$, $\dim\frak u_3=14$, and  $\dim\frak u_4=7$.  The semisimple part $[L,L]$ of $L$ is of type $SL(2)\times SL(7)$, which acts on $\frak u_1$ as the tensor product of the standard representation of the $SL(2)$ factor with the exterior square representation of the $SL(7)$ factor. It has the following orbits:

\begin{center}\begin{tabular}{p{7.01cm}p{1.56cm}p{1.91cm}}
 \cline{1-1} \cline{2-2} \cline{3-3}
 \multicolumn{1}{|p{7.01cm}|}{\centering Orbit Basepoint} &  \multicolumn{1}{p{1.56cm}|}{\centering Dimension} &  \multicolumn{1}{p{1.91cm}|}{\centering Coadjoint orbit intersected} \\
\cline{1-1} \cline{2-2} \cline{3-3}
 \multicolumn{1}{|p{7.01cm}|}{\centering $\nrel{01111111+01121110+01122100+}{+10111111+11111110+11121100}$} &  \multicolumn{1}{p{1.56cm}|}{\centering 42} &  \multicolumn{1}{p{1.91cm}|}{\centering 00000200} \\
 \cline{1-1} \cline{2-2} \cline{3-3}
 \multicolumn{1}{|p{7.01cm}|}{\centering $\nrel{01111111+01121110+10111111+}{+11111110+11122100}$} &  \multicolumn{1}{p{1.56cm}|}{\centering 40} &  \multicolumn{1}{p{1.91cm}|}{\centering 10000101} \\
 \cline{1-1} \cline{2-2} \cline{3-3}
 \multicolumn{1}{|p{7.01cm}|}{\centering 01111111+01121110+10111111+11122100} &  \multicolumn{1}{p{1.56cm}|}{\centering 38} &  \multicolumn{1}{p{1.91cm}|}{\centering 20000002} \\
 \cline{1-1} \cline{2-2} \cline{3-3}
 \multicolumn{1}{|p{7.01cm}|}{\centering $\nrel{01122111+01122210+10111111+}{+11111110+11121100}$} &  \multicolumn{1}{p{1.56cm}|}{\centering 37} &  \multicolumn{1}{p{1.91cm}|}{\centering 10000100} \\
 \cline{1-1} \cline{2-2} \cline{3-3}
 \multicolumn{1}{|p{7.01cm}|}{\centering 01111111+01121110+11111111+11122100} &  \multicolumn{1}{p{1.56cm}|}{\centering 36} &  \multicolumn{1}{p{1.91cm}|}{\centering 00000020} \\
 \cline{1-1} \cline{2-2} \cline{3-3}
 \multicolumn{1}{|p{7.01cm}|}{\centering 01122111+01122210+11111110+11121100} &  \multicolumn{1}{p{1.56cm}|}{\centering 35} &  \multicolumn{1}{p{1.91cm}|}{\centering 00000101} \\
 \cline{1-1} \cline{2-2} \cline{3-3}
 \multicolumn{1}{|p{7.01cm}|}{\centering $\nrel{01122111+01122210+11111111+}{+11121110+11122100}$} &  \multicolumn{1}{p{1.56cm}|}{\centering 34} &  \multicolumn{1}{p{1.91cm}|}{\centering 10000010} \\
 \cline{1-1} \cline{2-2} \cline{3-3}
 \multicolumn{1}{|p{7.01cm}|}{\centering 01122221+10111111+11111100+11121000} &  \multicolumn{1}{p{1.56cm}|}{\centering 33} &  \multicolumn{1}{p{1.91cm}|}{\centering 00000101} \\
 \cline{1-1} \cline{2-2} \cline{3-3}
 \multicolumn{1}{|p{7.01cm}|}{\centering 01121111+01122210+11111111+11122110} &  \multicolumn{1}{p{1.56cm}|}{\centering 32} &  \multicolumn{1}{p{1.91cm}|}{\centering 20000000} \\
 \cline{1-1} \cline{2-2} \cline{3-3}
 \multicolumn{1}{|p{7.01cm}|}{\centering 01122221+11111100+11121000} &  \multicolumn{1}{p{1.56cm}|}{\centering 31} &  \multicolumn{1}{p{1.91cm}|}{\centering 10000002} \\
 \cline{1-1} \cline{2-2} \cline{3-3}
 \multicolumn{1}{|p{7.01cm}|}{\centering 01122111+01122210+11121111+11122110} &  \multicolumn{1}{p{1.56cm}|}{\centering 30} &  \multicolumn{1}{p{1.91cm}|}{\centering 00000100} \\
 \cline{1-1} \cline{2-2} \cline{3-3}
 \multicolumn{1}{|p{7.01cm}|}{\centering 01122221+11111111+11121110+11122100} &  \multicolumn{1}{p{1.56cm}|}{\centering 30} &  \multicolumn{1}{p{1.91cm}|}{\centering 00000100} \\
 \cline{1-1} \cline{2-2} \cline{3-3}
 \multicolumn{1}{|p{7.01cm}|}{\centering 01122111+11121111+11122210} &  \multicolumn{1}{p{1.56cm}|}{\centering 28} &  \multicolumn{1}{p{1.91cm}|}{\centering 10000001} \\
 \cline{1-1} \cline{2-2} \cline{3-3}
 \multicolumn{1}{|p{7.01cm}|}{\centering 01122111+11122210} &  \multicolumn{1}{p{1.56cm}|}{\centering 24} &  \multicolumn{1}{p{1.91cm}|}{\centering 00000002} \\
 \cline{1-1} \cline{2-2} \cline{3-3}
 \multicolumn{1}{|p{7.01cm}|}{\centering 01122221+11122111+11122210} &  \multicolumn{1}{p{1.56cm}|}{\centering 23} &  \multicolumn{1}{p{1.91cm}|}{\centering 00000010} \\
 \cline{1-1} \cline{2-2} \cline{3-3}
 \multicolumn{1}{|p{7.01cm}|}{\centering 11111111+11121110+11122100} &  \multicolumn{1}{p{1.56cm}|}{\centering 22} &  \multicolumn{1}{p{1.91cm}|}{\centering 00000010} \\
 \cline{1-1} \cline{2-2} \cline{3-3}
 \multicolumn{1}{|p{7.01cm}|}{\centering 11122111+11122210} &  \multicolumn{1}{p{1.56cm}|}{\centering 19} &  \multicolumn{1}{p{1.91cm}|}{\centering 10000000} \\
 \cline{1-1} \cline{2-2} \cline{3-3}
 \multicolumn{1}{|p{7.01cm}|}{\centering 01122221+11122211} &  \multicolumn{1}{p{1.56cm}|}{\centering 18} &  \multicolumn{1}{p{1.91cm}|}{\centering 10000000} \\
 \cline{1-1} \cline{2-2} \cline{3-3}
 \multicolumn{1}{|p{7.01cm}|}{\centering 11122221} &  \multicolumn{1}{p{1.56cm}|}{\centering 12} &  \multicolumn{1}{p{1.91cm}|}{\centering 00000001} \\
 \cline{1-1} \cline{2-2} \cline{3-3}
 \multicolumn{1}{|p{7.01cm}|}{\centering 00000000} &  \multicolumn{1}{p{1.56cm}|}{\centering 0} &  \multicolumn{1}{p{1.91cm}|}{\centering 00000000} \\
 \cline{1-1} \cline{2-2} \cline{3-3}
\end{tabular}\end{center}

The action of $SL(7)$ on the 35-dimensional piece $\frak u_2$ is the exterior cube action from $E_7$, node 2, whose orbits were listed in section~\ref{sec:E7node2}. Its orbits are as follows:
\begin{center}\begin{tabular}{||p{7.3cm}p{1.66cm}p{1.75cm}||}
\hline
 \multicolumn{1}{|p{7.3cm}|}{\centering Orbit Basepoint} &  \multicolumn{1}{p{1.66cm}|}{\centering Dimension} &  \multicolumn{1}{p{1.75cm}|}{\centering Coadjoint orbit intersected} \\
\hline
 \multicolumn{1}{|p{7.3cm}|}{\centering $\nrel{12232210+11233210+12232111+}{+11232211+11222221}$} &  \multicolumn{1}{p{1.66cm}|}{\centering 35} &  \multicolumn{1}{p{1.75cm}|}{\centering 00100000} \\
 \cline{1-1} \cline{2-2} \cline{3-3}
 \multicolumn{1}{|p{7.3cm}|}{\centering 12233210+12232211+11233211+11232221} &  \multicolumn{1}{p{1.66cm}|}{\centering 34} &  \multicolumn{1}{p{1.75cm}|}{\centering 00000100} \\
 \cline{1-1} \cline{2-2} \cline{3-3}
 \multicolumn{1}{|p{7.3cm}|}{\centering 12243210+12232111+11233321} &  \multicolumn{1}{p{1.66cm}|}{\centering 31} &  \multicolumn{1}{p{1.75cm}|}{\centering 10000001} \\
 \cline{1-1} \cline{2-2} \cline{3-3}
 \multicolumn{1}{|p{7.3cm}|}{\centering 12243210+12233211+12232221+11233321} &  \multicolumn{1}{p{1.66cm}|}{\centering 28} &  \multicolumn{1}{p{1.75cm}|}{\centering 01000000} \\
 \cline{1-1} \cline{2-2} \cline{3-3}
 \multicolumn{1}{|p{7.3cm}|}{\centering 12232111+11233321} &  \multicolumn{1}{p{1.66cm}|}{\centering 26} &  \multicolumn{1}{p{1.75cm}|}{\centering 00000002} \\
 \cline{1-1} \cline{2-2} \cline{3-3}
 \multicolumn{1}{|p{7.3cm}|}{\centering 12243211+12233221+11233321} &  \multicolumn{1}{p{1.66cm}|}{\centering 25} &  \multicolumn{1}{p{1.75cm}|}{\centering 00000010} \\
 \cline{1-1} \cline{2-2} \cline{3-3}
 \multicolumn{1}{|p{7.3cm}|}{\centering 12243210+12233211+12232221} &  \multicolumn{1}{p{1.66cm}|}{\centering 21} &  \multicolumn{1}{p{1.75cm}|}{\centering 00000010} \\
 \cline{1-1} \cline{2-2} \cline{3-3}
 \multicolumn{1}{|p{7.3cm}|}{\centering 12243221+12233321} &  \multicolumn{1}{p{1.66cm}|}{\centering 20} &  \multicolumn{1}{p{1.75cm}|}{\centering 10000000} \\
 \cline{1-1} \cline{2-2} \cline{3-3}
 \multicolumn{1}{|p{7.3cm}|}{\centering 12244321} &  \multicolumn{1}{p{1.66cm}|}{\centering 13} &  \multicolumn{1}{p{1.75cm}|}{\centering 00000001} \\
 \cline{1-1} \cline{2-2} \cline{3-3}
 \multicolumn{1}{|p{7.3cm}|}{\centering 00000000} &  \multicolumn{1}{p{1.66cm}|}{\centering 0} &  \multicolumn{1}{p{1.75cm}|}{\centering 0000000} \\
\hline
\end{tabular}\end{center}

The 14-dimensional tensor product action of $SL(2)\times SL(7)$ on $\frak u_3$ occurs for $SL(9)$, node 7, and has 3 orbits, classified by rank.  They have dimensions 14, 8, and 0, with respective basepoints 01111111+11111110, 11111111, and 00000000 (in ${\frak{sl}}(9)$).  The orbits here are
\begin{center}\begin{tabular}{||p{3.38cm}p{1.69cm}||p{4.53cm}||}
\hline
 \multicolumn{1}{|p{3.38cm}|}{\centering Orbit Basepoint} &  \multicolumn{1}{p{1.69cm}|}{\centering Dimension} &  \multicolumn{1}{p{4.53cm}|}{\centering Coadjoint orbit intersected} \\
\hline
 \multicolumn{1}{|p{3.38cm}|}{\centering 13354321+22354321} &  \multicolumn{1}{p{1.69cm}|}{\centering 14} &  \multicolumn{1}{p{4.53cm}|}{\centering 10000000} \\
 \cline{1-1} \cline{2-2} \cline{3-3}
 \multicolumn{1}{|p{3.38cm}|}{\centering 23354321} &  \multicolumn{1}{p{1.69cm}|}{\centering 8} &  \multicolumn{1}{p{4.53cm}|}{\centering 00000001} \\
 \cline{1-1} \cline{2-2} \cline{3-3}
 \multicolumn{1}{|p{3.38cm}|}{\centering 00000000} &  \multicolumn{1}{p{1.69cm}|}{\centering 0} &  \multicolumn{1}{p{4.53cm}|}{\centering 00000000} \\
\hline
\end{tabular}\end{center}

The 7-dimensional action of $SL(7)$ on $\frak u_4$ is its standard action, and has 2 nonzero orbits:~zero and nonzero.

\subsubsection{$E_8$ Node 4}

This is the most intricate configuration in that it has the deepest grading.  Here $P=LU$ where $U$ is a 116-dimensional 6-step nilpotent group and $\frak u=\frak u_1\oplus \frak u_2\oplus \frak u_3\oplus \frak u_4\oplus \frak u_5\oplus \frak u_6$, with $\dim\frak u_1=30$, $\dim\frak u_2=30$, $\dim\frak u_3=20$,   $\dim\frak u_4=15$, $\dim\frak u_5=6$, and $\dim\frak u_6=5$.  The semisimple part $[L,L]$ of $L$ is of type $SL(3)\times SL(2)\times SL(5)$, which acts on $\frak u_1$ as the tensor product of the standard representations of each factor.  The orbits are given as follows:

\begin{center}\begin{tabular}{p{7.61cm}p{1.44cm}p{1.66cm}}
 \cline{1-1} \cline{2-2} \cline{3-3}
 \multicolumn{1}{|p{7.61cm}|}{\centering Orbit Basepoint} &  \multicolumn{1}{p{1.44cm}|}{\centering Dimension} &  \multicolumn{1}{p{1.66cm}|}{\centering Coadjoint orbit intersected} \\
\cline{1-1} \cline{2-2} \cline{3-3}
 \multicolumn{1}{|p{7.61cm}|}{\centering $\nrel{00011111+00111110+01011110+}{+01111100+10111000+11110000}$} &  \multicolumn{1}{p{1.44cm}|}{\centering 30} &  \multicolumn{1}{p{1.66cm}|}{\centering 20000200} \\
 \cline{1-1} \cline{2-2} \cline{3-3}
 \multicolumn{1}{|p{7.61cm}|}{\centering $\nrel{00111110+01011111+01111000+}{+10111100+11110000}$} &  \multicolumn{1}{p{1.44cm}|}{\centering 28} &  \multicolumn{1}{p{1.66cm}|}{\centering 20000101} \\
 \cline{1-1} \cline{2-2} \cline{3-3}
 \multicolumn{1}{|p{7.61cm}|}{\centering $\nrel{00011111+00111110+01011110+}{+01111100+10111100+11111000}$} &  \multicolumn{1}{p{1.44cm}|}{\centering 28} &  \multicolumn{1}{p{1.66cm}|}{\centering 00000200} \\
 \cline{1-1} \cline{2-2} \cline{3-3}
 \multicolumn{1}{|p{7.61cm}|}{\centering $\nrel{00111110+01011111+01111100+}{+10111100+11111000}$} &  \multicolumn{1}{p{1.44cm}|}{\centering 27} &  \multicolumn{1}{p{1.66cm}|}{\centering 10000101} \\
 \cline{1-1} \cline{2-2} \cline{3-3}
 \multicolumn{1}{|p{7.61cm}|}{\centering 00011111+01111110+10111100+11111000} &  \multicolumn{1}{p{1.44cm}|}{\centering 26} &  \multicolumn{1}{p{1.66cm}|}{\centering 20000002} \\
 \cline{1-1} \cline{2-2} \cline{3-3}
 \multicolumn{1}{|p{7.61cm}|}{\centering $\nrel{00111111+01011111+01111100+}{+10111110+11111000}$} &  \multicolumn{1}{p{1.44cm}|}{\centering 25} &  \multicolumn{1}{p{1.66cm}|}{\centering 10000100} \\
 \cline{1-1} \cline{2-2} \cline{3-3}
 \multicolumn{1}{|p{7.61cm}|}{\centering 00011111+01011111+01111110+10111100} &  \multicolumn{1}{p{1.44cm}|}{\centering 24} &  \multicolumn{1}{p{1.66cm}|}{\centering 00000020} \\
 \cline{1-1} \cline{2-2} \cline{3-3}
 \multicolumn{1}{|p{7.61cm}|}{\centering 00111111+01011110+10111110+11111100} &  \multicolumn{1}{p{1.44cm}|}{\centering 23} &  \multicolumn{1}{p{1.66cm}|}{\centering 00000101} \\
 \cline{1-1} \cline{2-2} \cline{3-3}
 \multicolumn{1}{|p{7.61cm}|}{\centering 01011110+01111100+10111111+11111000} &  \multicolumn{1}{p{1.44cm}|}{\centering 23} &  \multicolumn{1}{p{1.66cm}|}{\centering 00000101} \\
 \cline{1-1} \cline{2-2} \cline{3-3}
 \multicolumn{1}{|p{7.61cm}|}{\centering 00111111+01111100+10111110+11111000} &  \multicolumn{1}{p{1.44cm}|}{\centering 22} &  \multicolumn{1}{p{1.66cm}|}{\centering 20000000} \\
 \cline{1-1} \cline{2-2} \cline{3-3}
 \multicolumn{1}{|p{7.61cm}|}{\centering $\nrel{00111111+01011111+01111110+}{+10111110+11111100}$} &  \multicolumn{1}{p{1.44cm}|}{\centering 22} &  \multicolumn{1}{p{1.66cm}|}{\centering 10000010} \\
 \cline{1-1} \cline{2-2} \cline{3-3}
 \multicolumn{1}{|p{7.61cm}|}{\centering 01011110+01111100+10111111} &  \multicolumn{1}{p{1.44cm}|}{\centering 21} &  \multicolumn{1}{p{1.66cm}|}{\centering 10000002} \\
 \cline{1-1} \cline{2-2} \cline{3-3}
 \multicolumn{1}{|p{7.61cm}|}{\centering 00111111+01011111+10111110+11111100} &  \multicolumn{1}{p{1.44cm}|}{\centering 20} &  \multicolumn{1}{p{1.66cm}|}{\centering 20000000} \\
 \cline{1-1} \cline{2-2} \cline{3-3}
 \multicolumn{1}{|p{7.61cm}|}{\centering 00111111+01111110+10111110+11111100} &  \multicolumn{1}{p{1.44cm}|}{\centering 20} &  \multicolumn{1}{p{1.66cm}|}{\centering 00000100} \\
 \cline{1-1} \cline{2-2} \cline{3-3}
 \multicolumn{1}{|p{7.61cm}|}{\centering 01011111+01111110+10111111+11111100} &  \multicolumn{1}{p{1.44cm}|}{\centering 20} &  \multicolumn{1}{p{1.66cm}|}{\centering 00000100} \\
 \cline{1-1} \cline{2-2} \cline{3-3}
 \multicolumn{1}{|p{7.61cm}|}{\centering 01111110+10111111+11111100} &  \multicolumn{1}{p{1.44cm}|}{\centering 19} &  \multicolumn{1}{p{1.66cm}|}{\centering 10000001} \\
 \cline{1-1} \cline{2-2} \cline{3-3}
 \multicolumn{1}{|p{7.61cm}|}{\centering 00111111+01011111+01111110+10111110} &  \multicolumn{1}{p{1.44cm}|}{\centering 18} &  \multicolumn{1}{p{1.66cm}|}{\centering 00000100} \\
 \cline{1-1} \cline{2-2} \cline{3-3}
 \multicolumn{1}{|p{7.61cm}|}{\centering 00111111+01011111+11111110} &  \multicolumn{1}{p{1.44cm}|}{\centering 17} &  \multicolumn{1}{p{1.66cm}|}{\centering 10000001} \\
 \cline{1-1} \cline{2-2} \cline{3-3}
 \multicolumn{1}{|p{7.61cm}|}{\centering 00111111+11111110} &  \multicolumn{1}{p{1.44cm}|}{\centering 16} &  \multicolumn{1}{p{1.66cm}|}{\centering 00000002} \\
 \cline{1-1} \cline{2-2} \cline{3-3}
 \multicolumn{1}{|p{7.61cm}|}{\centering 01011111+01111110+11111100} &  \multicolumn{1}{p{1.44cm}|}{\centering 16} &  \multicolumn{1}{p{1.66cm}|}{\centering 00000010} \\
 \cline{1-1} \cline{2-2} \cline{3-3}
 \multicolumn{1}{|p{7.61cm}|}{\centering 01111111+10111111+11111110} &  \multicolumn{1}{p{1.44cm}|}{\centering 15} &  \multicolumn{1}{p{1.66cm}|}{\centering 00000010} \\
 \cline{1-1} \cline{2-2} \cline{3-3}
 \multicolumn{1}{|p{7.61cm}|}{\centering 01111111+11111110} &  \multicolumn{1}{p{1.44cm}|}{\centering 13} &  \multicolumn{1}{p{1.66cm}|}{\centering 10000000} \\
 \cline{1-1} \cline{2-2} \cline{3-3}
  \multicolumn{1}{|p{7.61cm}|}{\centering 10111111+11111110} &  \multicolumn{1}{p{1.44cm}|}{\centering 12} &  \multicolumn{1}{p{1.66cm}|}{\centering 10000000} \\
 \cline{1-1} \cline{2-2} \cline{3-3}
   \multicolumn{1}{|p{7.61cm}|}{\centering 01111111+10111111} &  \multicolumn{1}{p{1.44cm}|}{\centering 10} &  \multicolumn{1}{p{1.66cm}|}{\centering 10000000} \\
 \cline{1-1} \cline{2-2} \cline{3-3}
   \multicolumn{1}{|p{7.61cm}|}{\centering 11111111} &  \multicolumn{1}{p{1.44cm}|}{\centering 8} &  \multicolumn{1}{p{1.66cm}|}{\centering 00000001} \\
 \cline{1-1} \cline{2-2} \cline{3-3}
   \multicolumn{1}{|p{7.61cm}|}{\centering 00000000} &  \multicolumn{1}{p{1.44cm}|}{\centering 0} &  \multicolumn{1}{p{1.66cm}|}{\centering 00000000} \\
 \cline{1-1} \cline{2-2} \cline{3-3}
\end{tabular}\end{center}

The action on $\frak u_2$ is the tensor product of the standard action of $SL(3)$ with the exterior square representation of $SL(5)$.  It occurs for $E_7$, node 5, whose orbits were listed in section~\ref{sec:E7node5}.  The orbits from there translate here to the following:
\begin{center}\begin{tabular}{||p{7.2cm}p{1.56cm}p{1.72cm}}
 \cline{1-1} \cline{2-2} \cline{3-3}
 \multicolumn{1}{|p{7.2cm}|}{\centering Orbit Basepoint} &  \multicolumn{1}{p{1.56cm}|}{\centering Dimension} &  \multicolumn{1}{p{1.72cm}|}{\centering Coadjoint orbit intersected} \\
\cline{1-1} \cline{2-2} \cline{3-3}
 \multicolumn{1}{|p{7.2cm}|}{\centering $\nrel{11221110+11222100+11122110+}{+01122210+11121111+01122111}$} &  \multicolumn{1}{p{1.56cm}|}{\centering 30} &  \multicolumn{1}{p{1.72cm}|}{\centering 00010000} \\
 \cline{1-1} \cline{2-2} \cline{3-3}
 \multicolumn{1}{|p{7.2cm}|}{\centering $\nrel{11221110+11222100+11122110+}{+11121111+01122211}$} &  \multicolumn{1}{p{1.56cm}|}{\centering 29} &  \multicolumn{1}{p{1.72cm}|}{\centering 10000100} \\
 \cline{1-1} \cline{2-2} \cline{3-3}
 \multicolumn{1}{|p{7.2cm}|}{\centering $\nrel{11221110+11222100+11122110+}{+11221111+01122211}$} &  \multicolumn{1}{p{1.56cm}|}{\centering 28} &  \multicolumn{1}{p{1.72cm}|}{\centering 01000010} \\
 \cline{1-1} \cline{2-2} \cline{3-3}
 \multicolumn{1}{|p{7.2cm}|}{\centering 11221110+11122110+11221111+01122211} &  \multicolumn{1}{p{1.56cm}|}{\centering 27} &  \multicolumn{1}{p{1.72cm}|}{\centering 00000020} \\
 \cline{1-1} \cline{2-2} \cline{3-3}
 \multicolumn{1}{|p{7.2cm}|}{\centering $\nrel{11221110+11222100+11122210+}{+11121111+01122221}$} &  \multicolumn{1}{p{1.56cm}|}{\centering 27} &  \multicolumn{1}{p{1.72cm}|}{\centering 00100001} \\
 \cline{1-1} \cline{2-2} \cline{3-3}
 \multicolumn{1}{|p{7.2cm}|}{\centering 11222110+11122210+11121111+01122211} &  \multicolumn{1}{p{1.56cm}|}{\centering 26} &  \multicolumn{1}{p{1.72cm}|}{\centering 00000101} \\
 \cline{1-1} \cline{2-2} \cline{3-3}
 \multicolumn{1}{|p{7.2cm}|}{\centering $\nrel{11222100+11122210+11221111+}{+11122111+01122221}$} &  \multicolumn{1}{p{1.56cm}|}{\centering 25} &  \multicolumn{1}{p{1.72cm}|}{\centering 10000010} \\
 \cline{1-1} \cline{2-2} \cline{3-3}
 \multicolumn{1}{|p{7.2cm}|}{\centering 11221110+11222100+11121111+01122221} &  \multicolumn{1}{p{1.56cm}|}{\centering 24} &  \multicolumn{1}{p{1.72cm}|}{\centering 00000101} \\
 \cline{1-1} \cline{2-2} \cline{3-3}
 \multicolumn{1}{|p{7.2cm}|}{\centering 11222210+11121111+01122111} &  \multicolumn{1}{p{1.56cm}|}{\centering 23} &  \multicolumn{1}{p{1.72cm}|}{\centering 10000002} \\
 \cline{1-1} \cline{2-2} \cline{3-3}
 \multicolumn{1}{|p{7.2cm}|}{\centering 11221110+11222100+11122211+01122221} &  \multicolumn{1}{p{1.56cm}|}{\centering 23} &  \multicolumn{1}{p{1.72cm}|}{\centering 20000000} \\
 \cline{1-1} \cline{2-2} \cline{3-3}
 \multicolumn{1}{|p{7.2cm}|}{\centering $\nrel{11222110+11122210+11221111+}{+11122111+01122221}$} &  \multicolumn{1}{p{1.56cm}|}{\centering 23} &  \multicolumn{1}{p{1.72cm}|}{\centering 00100000} \\
 \cline{1-1} \cline{2-2} \cline{3-3}
 \multicolumn{1}{|p{7.2cm}|}{\centering 11222100+11122210+11122111+01122221} &  \multicolumn{1}{p{1.56cm}|}{\centering 22} &  \multicolumn{1}{p{1.72cm}|}{\centering 00000100} \\
 \cline{1-1} \cline{2-2} \cline{3-3}
 \multicolumn{1}{|p{7.2cm}|}{\centering 11222110+11122210+11221111+11122111} &  \multicolumn{1}{p{1.56cm}|}{\centering 22} &  \multicolumn{1}{p{1.72cm}|}{\centering 00000100} \\
 \cline{1-1} \cline{2-2} \cline{3-3}
 \multicolumn{1}{|p{7.2cm}|}{\centering 11222110+11221111+11122211+01122221} &  \multicolumn{1}{p{1.56cm}|}{\centering 22} &  \multicolumn{1}{p{1.72cm}|}{\centering 00000100} \\
 \cline{1-1} \cline{2-2} \cline{3-3}
 \multicolumn{1}{|p{7.2cm}|}{\centering 11222110+11122211+01122221} &  \multicolumn{1}{p{1.56cm}|}{\centering 21} &  \multicolumn{1}{p{1.72cm}|}{\centering 10000001} \\
 \cline{1-1} \cline{2-2} \cline{3-3}
 \multicolumn{1}{|p{7.2cm}|}{\centering 11222100+11221111+11122221} &  \multicolumn{1}{p{1.56cm}|}{\centering 20} &  \multicolumn{1}{p{1.72cm}|}{\centering 10000001} \\
 \cline{1-1} \cline{2-2} \cline{3-3}
 \multicolumn{1}{|p{7.2cm}|}{\centering 11222210+11222111+11122211+01122221} &  \multicolumn{1}{p{1.56cm}|}{\centering 19} &  \multicolumn{1}{p{1.72cm}|}{\centering 01000000} \\
 \cline{1-1} \cline{2-2} \cline{3-3}
 \multicolumn{1}{|p{7.2cm}|}{\centering 11222100+11122221} &  \multicolumn{1}{p{1.56cm}|}{\centering 18} &  \multicolumn{1}{p{1.72cm}|}{\centering 00000002} \\
 \cline{1-1} \cline{2-2} \cline{3-3}
 \multicolumn{1}{|p{7.2cm}|}{\centering 11222210+11222111+11122221} &  \multicolumn{1}{p{1.56cm}|}{\centering 17} &  \multicolumn{1}{p{1.72cm}|}{\centering 00000010} \\
 \cline{1-1} \cline{2-2} \cline{3-3}
 \multicolumn{1}{|p{7.2cm}|}{\centering 11222111+11122211+01122221} &  \multicolumn{1}{p{1.56cm}|}{\centering 16} &  \multicolumn{1}{p{1.72cm}|}{\centering 00000010} \\
 \cline{1-1} \cline{2-2} \cline{3-3}
 \multicolumn{1}{|p{7.2cm}|}{\centering 11222210+11122211+01122221} &  \multicolumn{1}{p{1.56cm}|}{\centering 15} &  \multicolumn{1}{p{1.72cm}|}{\centering 00000010} \\
 \cline{1-1} \cline{2-2} \cline{3-3}
 \multicolumn{1}{|p{7.2cm}|}{\centering 11222211+11122221} &  \multicolumn{1}{p{1.56cm}|}{\centering 14} &  \multicolumn{1}{p{1.72cm}|}{\centering 10000000} \\
 \cline{1-1} \cline{2-2} \cline{3-3}
 \multicolumn{1}{|p{7.2cm}|}{\centering 11222210+11222111} &  \multicolumn{1}{p{1.56cm}|}{\centering 12} &  \multicolumn{1}{p{1.72cm}|}{\centering 10000000} \\
 \cline{1-1} \cline{2-2} \cline{3-3}
 \multicolumn{1}{|p{7.2cm}|}{\centering 11222221} &  \multicolumn{1}{p{1.56cm}|}{\centering 9} &  \multicolumn{1}{p{1.72cm}|}{\centering 00000001} \\
 \cline{1-1} \cline{2-2} \cline{3-3}
 \multicolumn{1}{|p{7.2cm}|}{\centering 00000000} &  \multicolumn{1}{p{1.56cm}|}{\centering 0} &  \multicolumn{1}{p{1.72cm}|}{\centering 0000000} \\
\hline
\end{tabular}\end{center}

The action on $\frak u_3$ is the tensor product of the standard action of $SL(2)$ with the exterior square representation of $SL(5)$.  It occurs for $E_6$, node 3, whose orbits were listed in section~\ref{sec:E7node3}.  The orbits here are the following:
\begin{center}\begin{tabular}{p{6.41cm}p{1.72cm}p{2cm}}
 \cline{1-1} \cline{2-2} \cline{3-3}
 \multicolumn{1}{|p{6.41cm}|}{\centering Orbit Basepoint} &  \multicolumn{1}{p{1.72cm}|}{\centering Dimension} &  \multicolumn{1}{p{2cm}|}{\centering Coadjoint orbit intersected} \\
\cline{1-1} \cline{2-2} \cline{3-3}
 \multicolumn{1}{|p{6.41cm}|}{\centering 11233211+11232221+12233210+12232211} &  \multicolumn{1}{p{1.72cm}|}{\centering 20} &  \multicolumn{1}{p{2cm}|}{\centering 00000100} \\
 \cline{1-1} \cline{2-2} \cline{3-3}
 \multicolumn{1}{|p{6.41cm}|}{\centering 11233211+12233210+12232221} &  \multicolumn{1}{p{1.72cm}|}{\centering 18} &  \multicolumn{1}{p{2cm}|}{\centering 10000001} \\
 \cline{1-1} \cline{2-2} \cline{3-3}
 \multicolumn{1}{|p{6.41cm}|}{\centering 11233211+12232221} &  \multicolumn{1}{p{1.72cm}|}{\centering 16} &  \multicolumn{1}{p{2cm}|}{\centering 00000002} \\
 \cline{1-1} \cline{2-2} \cline{3-3}
 \multicolumn{1}{|p{6.41cm}|}{\centering 11233321+12233211+12232221} &  \multicolumn{1}{p{1.72cm}|}{\centering 15} &  \multicolumn{1}{p{2cm}|}{\centering 00000010} \\
 \cline{1-1} \cline{2-2} \cline{3-3}
 \multicolumn{1}{|p{6.41cm}|}{\centering 11233321+12233221} &  \multicolumn{1}{p{1.72cm}|}{\centering 12} &  \multicolumn{1}{p{2cm}|}{\centering 10000000} \\
 \cline{1-1} \cline{2-2} \cline{3-3}
 \multicolumn{1}{|p{6.41cm}|}{\centering 12233211+12232221} &  \multicolumn{1}{p{1.72cm}|}{\centering 11} &  \multicolumn{1}{p{2cm}|}{\centering 10000000} \\
 \cline{1-1} \cline{2-2} \cline{3-3}
 \multicolumn{1}{|p{6.41cm}|}{\centering 12233321} &  \multicolumn{1}{p{1.72cm}|}{\centering 8} &  \multicolumn{1}{p{2cm}|}{\centering 00000001} \\
 \cline{1-1} \cline{2-2} \cline{3-3}
 \multicolumn{1}{|p{6.41cm}|}{\centering 00000000} &  \multicolumn{1}{p{1.72cm}|}{\centering 0} &  \multicolumn{1}{p{2cm}|}{\centering 00000000} \\
 \cline{1-1} \cline{2-2} \cline{3-3}
\end{tabular}\end{center}
The action on $\frak u_4$ is the tensor product of the standard actions of $SL(3)$ and $SL(5)$, and arises for $SL(8)$, node 3.  It has 4 orbits, classified by rank.  They have dimensions 15, 12, 7, and 0, with respective basepoints 0011111 + 0111110 + 1111100, 0111111 + 1111110, 1111111, and 0000000 (in $\frak a_7$). It has the following orbits in $\frak u_4$:
\begin{center}\begin{tabular}{|cp{1.66cm}||p{2.66cm}||}
\hline
 \multicolumn{1}{|c|}{Orbit Basepoint} &  \multicolumn{1}{p{1.66cm}|}{\centering Dimension} &  \multicolumn{1}{p{2.66cm}|}{\centering Coadjoint orbit intersected} \\
\hline
 \multicolumn{1}{|c|}{22343221+12343321+12244321} &  \multicolumn{1}{p{1.66cm}|}{\centering 15} &  \multicolumn{1}{p{2.66cm}|}{\centering 00000010} \\
 \cline{1-1} \cline{2-2} \cline{3-3}
 \multicolumn{1}{|c|}{22343321+12344321} &  \multicolumn{1}{p{1.66cm}|}{\centering 12} &  \multicolumn{1}{p{2.66cm}|}{\centering 10000000} \\
 \cline{1-1} \cline{2-2} \cline{3-3}
 \multicolumn{1}{|c|}{22344321} &  \multicolumn{1}{p{1.66cm}|}{\centering 7} &  \multicolumn{1}{p{2.66cm}|}{\centering 00000001} \\
 \cline{1-1} \cline{2-2} \cline{3-3}
 \multicolumn{1}{|c|}{00000000} &  \multicolumn{1}{p{1.66cm}|}{\centering 0} &  \multicolumn{1}{p{2.66cm}|}{\centering 00000000} \\
\hline
\end{tabular}\end{center}

The action of $\frak u_5$ is the tensor product of the standard actions of $SL(3)$ and $SL(2)$, and occurs for $SL(5)$, node 2.  It has 3 orbits, classified by rank.  They have dimensions 6, 4, and 0, with respective basepoints 0111 + 1110, 1111, and 0000 there.  The orbits here are as follows:
\begin{center}\begin{tabular}{|cp{1.66cm}||p{2.78cm}||}
\hline
 \multicolumn{1}{|c|}{Orbit Basepoint} &  \multicolumn{1}{p{1.66cm}|}{\centering Dimension} &  \multicolumn{1}{p{2.78cm}|}{\centering Coadjoint orbit intersected} \\
\hline
 \multicolumn{1}{|c|}{23354321+22454321} &  \multicolumn{1}{p{1.66cm}|}{\centering 6} &  \multicolumn{1}{p{2.78cm}|}{\centering 10000000} \\
 \cline{1-1} \cline{2-2} \cline{3-3}
 \multicolumn{1}{|c|}{23454321} &  \multicolumn{1}{p{1.66cm}|}{\centering 4} &  \multicolumn{1}{p{2.78cm}|}{\centering 00000001} \\
 \cline{1-1} \cline{2-2} \cline{3-3}
 \multicolumn{1}{|c|}{00000000} &  \multicolumn{1}{p{1.66cm}|}{\centering 0} &  \multicolumn{1}{p{2.78cm}|}{\centering 00000000} \\
\hline
\end{tabular}\end{center}

The action of $\frak u_6$ is the standard action of $GL(5)$, and has two orbits:~zero and nonzero.

\subsubsection{$E_8$ Node 5}\label{sec:E8node5}

Here $P=LU$ where $U$ is a 104-dimensional 5-step nilpotent group and $\frak u=\frak u_1\oplus \frak u_2\oplus \frak u_3\oplus \frak u_4\oplus \frak u_5$, with $\dim\frak u_1=40$, $\dim\frak u_2=30$, $\dim\frak u_3=20$,   $\dim\frak u_4=10$, and $\dim\frak u_5=4$.  The semisimple part $[L,L]$ of $L$ is of type $SL(5)\times SL(4)$, which acts on $\frak u_1$ as the tensor product of the exterior square representation of the $SL(5)$ factor with the standard representation of the $SL(4)$ factor.  The orbits are given as follows:

\begin{center}\begin{tabular}{p{7.4cm}p{1.53cm}p{1.88cm}}
 \cline{1-1} \cline{2-2} \cline{3-3}
 \multicolumn{1}{|p{7.4cm}|}{\centering Orbit Basepoint} &  \multicolumn{1}{p{1.5cm}|}{\centering Dimension} &  \multicolumn{1}{p{1.88cm}|}{\centering Coadjoint orbit intersected} \\
\cline{1-1} \cline{2-2} \cline{3-3}
 \multicolumn{1}{|p{7.4cm}|}{\centering $\nrel{00001000+00011110+01011111+01111110+}{+01121100+10111100+11121000+11221111}$} &  \multicolumn{1}{p{1.5cm}|}{\centering 40} &  \multicolumn{1}{p{1.88cm}|}{\centering 00002000} \\
 \cline{1-1} \cline{2-2} \cline{3-3}
 \multicolumn{1}{|p{7.4cm}|}{\centering $\nrel{00011110+00111100+01011111+01111110+}{+10111110+11111100+11221000}$} &  \multicolumn{1}{p{1.5cm}|}{\centering 39} &  \multicolumn{1}{p{1.88cm}|}{\centering 00010100} \\
 \cline{1-1} \cline{2-2} \cline{3-3}
 \multicolumn{1}{|p{7.4cm}|}{\centering $\nrel{00011111+00111110+01011110+01111111+}{+01121100+11111100+11221000}$} &  \multicolumn{1}{p{1.5cm}|}{\centering 38} &  \multicolumn{1}{p{1.88cm}|}{\centering 10001010} \\
 \cline{1-1} \cline{2-2} \cline{3-3}
 \multicolumn{1}{|p{7.4cm}|}{\centering $\nrel{00111111+01011110+01121000+}{+10111110+11111100+11121000}$} &  \multicolumn{1}{p{1.5cm}|}{\centering 38} &  \multicolumn{1}{p{1.88cm}|}{\centering 01100010} \\
 \cline{1-1} \cline{2-2} \cline{3-3}
 \multicolumn{1}{|p{7.4cm}|}{\centering $\nrel{00011111+01011110+01111110+01121100+}{+10111110+11111100+11221000}$} &  \multicolumn{1}{p{1.5cm}|}{\centering 37} &  \multicolumn{1}{p{1.88cm}|}{\centering 00100101} \\
 \cline{1-1} \cline{2-2} \cline{3-3}
 \multicolumn{1}{|p{7.4cm}|}{\centering $\nrel{00111111+01011111+01111110+}{+01121100+10111100+11121000}$} &  \multicolumn{1}{p{1.5cm}|}{\centering 37} &  \multicolumn{1}{p{1.88cm}|}{\centering 10010001} \\
 \cline{1-1} \cline{2-2} \cline{3-3}
 \multicolumn{1}{|p{7.4cm}|}{\centering $\nrel{00011111+00111110+01011110+}{+01111111+11111100+11221000}$} &  \multicolumn{1}{p{1.5cm}|}{\centering 36} &  \multicolumn{1}{p{1.88cm}|}{\centering 20000020} \\
 \cline{1-1} \cline{2-2} \cline{3-3}
 \multicolumn{1}{|p{7.4cm}|}{\centering $\nrel{00011111+01111110+01121100+}{+10111110+11111100+11221000}$} &  \multicolumn{1}{p{1.5cm}|}{\centering 36} &  \multicolumn{1}{p{1.88cm}|}{\centering 02000002} \\
 \cline{1-1} \cline{2-2} \cline{3-3}
 \multicolumn{1}{|p{7.4cm}|}{\centering $\nrel{00111111+01011111+01111110+01121100+}{+10111110+11111100+11121000}$} &  \multicolumn{1}{p{1.5cm}|}{\centering 36} &  \multicolumn{1}{p{1.88cm}|}{\centering 00010010} \\
\hline
 \multicolumn{1}{|p{7.4cm}|}{\centering $\nrel{00111111+01011111+01111110+01121100+}{+10111110+11111100+11221000}$} &  \multicolumn{1}{p{1.5cm}|}{\centering 35} &  \multicolumn{1}{p{1.88cm}|}{\centering 00100100} \\
\hline
 \multicolumn{1}{|p{7.4cm}|}{\centering $\nrel{00011111+01011110+01111110+}{+01121100+11111100+11221000}$} &  \multicolumn{1}{p{1.5cm}|}{\centering 35} &  \multicolumn{1}{p{1.88cm}|}{\centering 00010002} \\
 \cline{1-1} \cline{2-2} \cline{3-3}
 \multicolumn{1}{|p{7.4cm}|}{\centering $\nrel{00111111+01011111+01121110+}{+10111100+11111000}$} &  \multicolumn{1}{p{1.5cm}|}{\centering 35} &  \multicolumn{1}{p{1.88cm}|}{\centering 20000101} \\
 \cline{1-1} \cline{2-2} \cline{3-3}
 \multicolumn{1}{|p{7.4cm}|}{\centering $\nrel{00001111+01121110+10111110+}{+11111100+11121100+11221000}$} &  \multicolumn{1}{p{1.5cm}|}{\centering 35} &  \multicolumn{1}{p{1.88cm}|}{\centering 00010002} \\
 \cline{1-1} \cline{2-2} \cline{3-3}
 \multicolumn{1}{|p{7.4cm}|}{\centering $\nrel{00001111+01121110+11111100+}{+11121100+11221000}$} &  \multicolumn{1}{p{1.5cm}|}{\centering 34} &  \multicolumn{1}{p{1.88cm}|}{\centering 10000102} \\
 \cline{1-1} \cline{2-2} \cline{3-3}
 \multicolumn{1}{|p{7.4cm}|}{\centering $\nrel{00111111+01111110+01121100+}{+10111110+11111100+11121000}$} &  \multicolumn{1}{p{1.5cm}|}{\centering 34} &  \multicolumn{1}{p{1.88cm}|}{\centering 00000200} \\
 \cline{1-1} \cline{2-2} \cline{3-3}
 \multicolumn{1}{|p{7.4cm}|}{\centering $\nrel{00111111+01011111+01121100+}{+10111110+11111100+11221000}$} &  \multicolumn{1}{p{1.5cm}|}{\centering 34} &  \multicolumn{1}{p{1.88cm}|}{\centering 00000200} \\
 \cline{1-1} \cline{2-2} \cline{3-3}
 \multicolumn{1}{|p{7.4cm}|}{\centering $\nrel{00111111+01011111+01111110+}{+01121100+10111110+11221000}$} &  \multicolumn{1}{p{1.5cm}|}{\centering 34} &  \multicolumn{1}{p{1.88cm}|}{\centering 00010001} \\
 \cline{1-1} \cline{2-2} \cline{3-3}
\end{tabular}\end{center}
\begin{center}\begin{tabular}{p{7.31cm}p{1.5cm}p{1.88cm}}
 \cline{1-1} \cline{2-2} \cline{3-3}
 \multicolumn{1}{|p{7.31cm}|}{\centering $\nrel{\text{Orbit Basepoint}}{\text{(Continued)}}$} &  \multicolumn{1}{p{1.5cm}|}{\centering Dimension} &  \multicolumn{1}{p{1.88cm}|}{\centering Coadjoint orbit intersected} \\
\cline{1-1} \cline{2-2} \cline{3-3}
 \multicolumn{1}{|p{7.31cm}|}{\centering $\nrel{01011111+01111110+01121100+}{+10111111+11111100+11121000}$} &  \multicolumn{1}{p{1.5cm}|}{\centering 34} &  \multicolumn{1}{p{1.88cm}|}{\centering 00010001} \\
 \cline{1-1} \cline{2-2} \cline{3-3}
 \multicolumn{1}{|p{7.31cm}|}{\centering $\nrel{00111111+01011111+01121100+}{+11111110+11121000}$} &  \multicolumn{1}{p{1.5cm}|}{\centering 33} &  \multicolumn{1}{p{1.88cm}|}{\centering 10000101} \\
 \cline{1-1} \cline{2-2} \cline{3-3}
 \multicolumn{1}{|p{7.31cm}|}{\centering $\nrel{01111110+01121100+10111111+}{+11111100+11121000}$} &  \multicolumn{1}{p{1.5cm}|}{\centering 33} &  \multicolumn{1}{p{1.88cm}|}{\centering 10000101} \\
 \cline{1-1} \cline{2-2} \cline{3-3}
 \multicolumn{1}{|p{7.31cm}|}{\centering $\nrel{00111111+01011111+01111110+}{+10111110+11121100+11221000}$} &  \multicolumn{1}{p{1.5cm}|}{\centering 33} &  \multicolumn{1}{p{1.88cm}|}{\centering 10001000} \\
 \cline{1-1} \cline{2-2} \cline{3-3}
 \multicolumn{1}{|p{7.31cm}|}{\centering $\nrel{00111111+01011111+01111110+}{+01121100+10111110+11111100}$} &  \multicolumn{1}{p{1.5cm}|}{\centering 33} &  \multicolumn{1}{p{1.88cm}|}{\centering 00010000} \\
 \cline{1-1} \cline{2-2} \cline{3-3}
 \multicolumn{1}{|p{7.31cm}|}{\centering $\nrel{00001111+01121110+11111110+}{+11121100+11221000}$} &  \multicolumn{1}{p{1.5cm}|}{\centering 32} &  \multicolumn{1}{p{1.88cm}|}{\centering 01000012} \\
 \cline{1-1} \cline{2-2} \cline{3-3}
 \multicolumn{1}{|p{7.31cm}|}{\centering 00111111+01121100+11111110+11121000} &  \multicolumn{1}{p{1.5cm}|}{\centering 32} &  \multicolumn{1}{p{1.88cm}|}{\centering 20000002} \\
 \cline{1-1} \cline{2-2} \cline{3-3}
 \multicolumn{1}{|p{7.31cm}|}{\centering $\nrel{00111111+01011111+01111110+}{+10111111+11121100+11221000}$} &  \multicolumn{1}{p{1.5cm}|}{\centering 32} &  \multicolumn{1}{p{1.88cm}|}{\centering 02000000} \\
 \cline{1-1} \cline{2-2} \cline{3-3}
 \multicolumn{1}{|p{7.31cm}|}{\centering $\nrel{00111111+01011111+01111110+}{+10111110+11121100}$} &  \multicolumn{1}{p{1.5cm}|}{\centering 32} &  \multicolumn{1}{p{1.88cm}|}{\centering 10000100} \\
 \cline{1-1} \cline{2-2} \cline{3-3}
 \multicolumn{1}{|p{7.31cm}|}{\centering $\nrel{00111111+01011111+01111110+}{+10111111+11121100}$} &  \multicolumn{1}{p{1.5cm}|}{\centering 31} &  \multicolumn{1}{p{1.88cm}|}{\centering 01000010} \\
 \cline{1-1} \cline{2-2} \cline{3-3}
 \multicolumn{1}{|p{7.31cm}|}{\centering $\nrel{00111111+01111110+10111111+}{+11121100+11221000}$} &  \multicolumn{1}{p{1.5cm}|}{\centering 31} &  \multicolumn{1}{p{1.88cm}|}{\centering 01000010} \\
 \cline{1-1} \cline{2-2} \cline{3-3}
 \multicolumn{1}{|p{7.31cm}|}{\centering $\nrel{00111111+01011111+01121110+}{+11111110+11121100+11221000}$} &  \multicolumn{1}{p{1.5cm}|}{\centering 31} &  \multicolumn{1}{p{1.88cm}|}{\centering 00010000} \\
 \cline{1-1} \cline{2-2} \cline{3-3}
 \multicolumn{1}{|p{7.31cm}|}{\centering $\nrel{01011111+01121110+10111111+}{+11111100+11221000}$} &  \multicolumn{1}{p{1.5cm}|}{\centering 31} &  \multicolumn{1}{p{1.88cm}|}{\centering 10000100} \\
 \cline{1-1} \cline{2-2} \cline{3-3}
 \multicolumn{1}{|p{7.31cm}|}{\centering $\nrel{00111111+01011111+01121110+}{+10111110+11221100}$} &  \multicolumn{1}{p{1.5cm}|}{\centering 30} &  \multicolumn{1}{p{1.88cm}|}{\centering 00100001} \\
 \cline{1-1} \cline{2-2} \cline{3-3}
 \multicolumn{1}{|p{7.31cm}|}{\centering 00111111+01111110+10111111+11121100} &  \multicolumn{1}{p{1.5cm}|}{\centering 30} &  \multicolumn{1}{p{1.88cm}|}{\centering 00000020} \\
 \cline{1-1} \cline{2-2} \cline{3-3}
 \multicolumn{1}{|p{7.31cm}|}{\centering $\nrel{01111111+01121110+10111110+}{+11121100+11221000}$} &  \multicolumn{1}{p{1.5cm}|}{\centering 30} &  \multicolumn{1}{p{1.88cm}|}{\centering 00100001} \\
 \cline{1-1} \cline{2-2} \cline{3-3}
 \multicolumn{1}{|p{7.31cm}|}{\centering 01111111+01121110+10111110+11121100} &  \multicolumn{1}{p{1.5cm}|}{\centering 29} &  \multicolumn{1}{p{1.88cm}|}{\centering 00000101} \\
 \cline{1-1} \cline{2-2} \cline{3-3}
 \multicolumn{1}{|p{7.31cm}|}{\centering $\nrel{00111111+01011111+11111110+}{+11121100+11221000}$} &  \multicolumn{1}{p{1.5cm}|}{\centering 29} &  \multicolumn{1}{p{1.88cm}|}{\centering 10000100} \\
 \cline{1-1} \cline{2-2} \cline{3-3}
 \multicolumn{1}{|p{7.31cm}|}{\centering $\nrel{01111111+01121110+10111111+}{+11111110+11121100+11221000}$} &  \multicolumn{1}{p{1.5cm}|}{\centering 29} &  \multicolumn{1}{p{1.88cm}|}{\centering 00001000} \\
 \cline{1-1} \cline{2-2} \cline{3-3}
 \multicolumn{1}{|p{7.31cm}|}{\centering $\nrel{01011111+01121110+10111111+}{+11111110+11221100}$} &  \multicolumn{1}{p{1.5cm}|}{\centering 28} &  \multicolumn{1}{p{1.88cm}|}{\centering 10000010} \\
 \cline{1-1} \cline{2-2} \cline{3-3}
 \multicolumn{1}{|p{7.31cm}|}{\centering 01011111+01111110+11121100+11221000} &  \multicolumn{1}{p{1.5cm}|}{\centering 28} &  \multicolumn{1}{p{1.88cm}|}{\centering 20000000} \\
 \cline{1-1} \cline{2-2} \cline{3-3}
\end{tabular}\end{center}

\begin{center}\begin{tabular}{p{7.27cm}p{1.5cm}p{1.88cm}}
 \cline{1-1} \cline{2-2} \cline{3-3}
 \multicolumn{1}{|p{7.27cm}|}{\centering $\nrel{\text{Orbit Basepoint}}{\text{(Continued)}}$} &  \multicolumn{1}{p{1.5cm}|}{\centering Dimension} &  \multicolumn{1}{p{1.88cm}|}{\centering Coadjoint orbit intersected} \\
\cline{1-1} \cline{2-2} \cline{3-3}
 \multicolumn{1}{|p{7.27cm}|}{\centering 01121111+10111110+11111100+11221000} &  \multicolumn{1}{p{1.5cm}|}{\centering 28} &  \multicolumn{1}{p{1.88cm}|}{\centering 00000101} \\
 \cline{1-1} \cline{2-2} \cline{3-3}
 \multicolumn{1}{|p{7.27cm}|}{\centering 00001111+01121110+11121100+11221000} &  \multicolumn{1}{p{1.5cm}|}{\centering 28} &  \multicolumn{1}{p{1.88cm}|}{\centering 00000022} \\
 \cline{1-1} \cline{2-2} \cline{3-3}
 \multicolumn{1}{|p{7.27cm}|}{\centering 00111111+01011111+10111110+11221100} &  \multicolumn{1}{p{1.5cm}|}{\centering 27} &  \multicolumn{1}{p{1.88cm}|}{\centering 00000101} \\
 \cline{1-1} \cline{2-2} \cline{3-3}
 \multicolumn{1}{|p{7.27cm}|}{\centering $\nrel{01111111+01121110+10111111+}{+11121100+11221000}$} &  \multicolumn{1}{p{1.5cm}|}{\centering 27} &  \multicolumn{1}{p{1.88cm}|}{\centering 10000010} \\
 \cline{1-1} \cline{2-2} \cline{3-3}
 \multicolumn{1}{|p{7.27cm}|}{\centering $\nrel{01111111+01121110+11111110+}{+11121100+11221000}$} &  \multicolumn{1}{p{1.5cm}|}{\centering 27} &  \multicolumn{1}{p{1.88cm}|}{\centering 00100000} \\
 \cline{1-1} \cline{2-2} \cline{3-3}
 \multicolumn{1}{|p{7.27cm}|}{\centering 00111111+01011111+11121110+11221100} &  \multicolumn{1}{p{1.5cm}|}{\centering 26} &  \multicolumn{1}{p{1.88cm}|}{\centering 20000000} \\
 \cline{1-1} \cline{2-2} \cline{3-3}
 \multicolumn{1}{|p{7.27cm}|}{\centering 01121111+10111110+11111100} &  \multicolumn{1}{p{1.5cm}|}{\centering 26} &  \multicolumn{1}{p{1.88cm}|}{\centering 10000002} \\
 \cline{1-1} \cline{2-2} \cline{3-3}
 \multicolumn{1}{|p{7.27cm}|}{\centering $\nrel{01111111+01121110+10111111+}{+11111110+11221100}$} &  \multicolumn{1}{p{1.5cm}|}{\centering 26} &  \multicolumn{1}{p{1.88cm}|}{\centering 00100000} \\
 \cline{1-1} \cline{2-2} \cline{3-3}
 \multicolumn{1}{|p{7.27cm}|}{\centering $\nrel{01121111+10111111+11111110+}{+11121100+11221000}$} &  \multicolumn{1}{p{1.5cm}|}{\centering 26} &  \multicolumn{1}{p{1.88cm}|}{\centering 00100000} \\
 \cline{1-1} \cline{2-2} \cline{3-3}
 \multicolumn{1}{|p{7.27cm}|}{\centering 01011111+01121110+11111110+11221100} &  \multicolumn{1}{p{1.5cm}|}{\centering 25} &  \multicolumn{1}{p{1.88cm}|}{\centering 00000100} \\
 \cline{1-1} \cline{2-2} \cline{3-3}
 \multicolumn{1}{|p{7.27cm}|}{\centering 01121111+11111110+11121100+11221000} &  \multicolumn{1}{p{1.5cm}|}{\centering 25} &  \multicolumn{1}{p{1.88cm}|}{\centering 00000100} \\
 \cline{1-1} \cline{2-2} \cline{3-3}
 \multicolumn{1}{|p{7.27cm}|}{\centering 01111111+10111111+11121110+11221100} &  \multicolumn{1}{p{1.5cm}|}{\centering 25} &  \multicolumn{1}{p{1.88cm}|}{\centering 00000100} \\
 \cline{1-1} \cline{2-2} \cline{3-3}
 \multicolumn{1}{|p{7.27cm}|}{\centering 01111111+01121110+10111111+11111110} &  \multicolumn{1}{p{1.5cm}|}{\centering 24} &  \multicolumn{1}{p{1.88cm}|}{\centering 00000100} \\
 \cline{1-1} \cline{2-2} \cline{3-3}
 \multicolumn{1}{|p{7.27cm}|}{\centering 01111111+11121110+11221100} &  \multicolumn{1}{p{1.5cm}|}{\centering 24} &  \multicolumn{1}{p{1.88cm}|}{\centering 10000001} \\
 \cline{1-1} \cline{2-2} \cline{3-3}
 \multicolumn{1}{|p{7.27cm}|}{\centering 01011111+10111111+11221110} &  \multicolumn{1}{p{1.5cm}|}{\centering 22} &  \multicolumn{1}{p{1.88cm}|}{\centering 10000001} \\
 \cline{1-1} \cline{2-2} \cline{3-3}
 \multicolumn{1}{|p{7.27cm}|}{\centering 01121111+11111111+11121110+11221100} &  \multicolumn{1}{p{1.5cm}|}{\centering 22} &  \multicolumn{1}{p{1.88cm}|}{\centering 01000000} \\
 \cline{1-1} \cline{2-2} \cline{3-3}
 \multicolumn{1}{|p{7.27cm}|}{\centering 01011111+11221110} &  \multicolumn{1}{p{1.5cm}|}{\centering 20} &  \multicolumn{1}{p{1.88cm}|}{\centering 00000002} \\
 \cline{1-1} \cline{2-2} \cline{3-3}
 \multicolumn{1}{|p{7.27cm}|}{\centering 10111111+11111110+11121100+11221000} &  \multicolumn{1}{p{1.5cm}|}{\centering 20} &  \multicolumn{1}{p{1.88cm}|}{\centering 01000000} \\
 \cline{1-1} \cline{2-2} \cline{3-3}
 \multicolumn{1}{|p{7.27cm}|}{\centering 11111111+11121110+11221100} &  \multicolumn{1}{p{1.5cm}|}{\centering 19} &  \multicolumn{1}{p{1.88cm}|}{\centering 00000010} \\
 \cline{1-1} \cline{2-2} \cline{3-3}
 \multicolumn{1}{|p{7.27cm}|}{\centering 01121111+11111111+11221110} &  \multicolumn{1}{p{1.5cm}|}{\centering 19} &  \multicolumn{1}{p{1.88cm}|}{\centering 00000010} \\
 \cline{1-1} \cline{2-2} \cline{3-3}
 \multicolumn{1}{|p{7.27cm}|}{\centering 01121111+11121110+11221100} &  \multicolumn{1}{p{1.5cm}|}{\centering 18} &  \multicolumn{1}{p{1.88cm}|}{\centering 00000010} \\
 \cline{1-1} \cline{2-2} \cline{3-3}
 \multicolumn{1}{|p{7.27cm}|}{\centering 11121111+11221110} &  \multicolumn{1}{p{1.5cm}|}{\centering 16} &  \multicolumn{1}{p{1.88cm}|}{\centering 10000000} \\
 \cline{1-1} \cline{2-2} \cline{3-3}
 \multicolumn{1}{|p{7.27cm}|}{\centering 01121111+11111111} &  \multicolumn{1}{p{1.5cm}|}{\centering 13} &  \multicolumn{1}{p{1.88cm}|}{\centering 10000000} \\
 \cline{1-1} \cline{2-2} \cline{3-3}
 \multicolumn{1}{|p{7.27cm}|}{\centering 11221111} &  \multicolumn{1}{p{1.5cm}|}{\centering 10} &  \multicolumn{1}{p{1.88cm}|}{\centering 00000001} \\
 \cline{1-1} \cline{2-2} \cline{3-3}
 \multicolumn{1}{|p{7.27cm}|}{\centering 00000000} &  \multicolumn{1}{p{1.5cm}|}{\centering 0} &  \multicolumn{1}{p{1.88cm}|}{\centering 00000000} \\
 \cline{1-1} \cline{2-2} \cline{3-3}
\end{tabular}\end{center}

The 30-dimensional action of $SL(5)\times SL(4)$ on $\frak u_2$ is the tensor product of the standard action on $SL(5)$ with the exterior square action of $SL(4)$, and occurs for $SO(8,8)$, node 5.  It has 16 orbits, of dimensions 30, 29, 28, 27, 25, 24, 23, 21, 18,  18, 18, 17, 15, 10, 9, and 0, with respective basepoints
     00001211 + 00011111 +00111101 + 00111110 + 01111100 + 11111000,
     00001211 + 00011111 + 00111101 +     01111110 + 11111100,
     00011111 + 00111101 + 01111110 + 11111100,
     00011211 + 00111111 + 01111101 +     01111110 + 11111100,
     00011211 + 00111111 + 01111101 +  11111110,
     00111111 + 01111101 + 01111110 + 11111100,
     00111211 + 01111101 + 11111110,
     00111211 + 01111111 + 11111101 + 11111110,
     01111101 + 11111110,
     00111211 + 01111111 + 11111101,
     00111211 + 01111111 + 11111110,
     01111211 + 11111101 + 11111110,
     01111211 + 11111111,
    11111101 + 11111110, and
     11111211,
   00000000 (in ${\frak{so}}(8,8)$).
     The orbits here are given as follows:
\begin{center}\begin{tabular}{||p{7.53cm}p{1.69cm}||p{1.47cm}||}
\hline
 \multicolumn{1}{|p{7.53cm}|}{\centering Orbit Basepoint} &  \multicolumn{1}{p{1.69cm}|}{\centering Dimension} &  \multicolumn{1}{p{1.47cm}|}{\centering Coadjoint orbit intersected} \\
\hline
 \multicolumn{1}{|p{7.53cm}|}{\centering $\nrel{01122221+11122211+11222210+}{+11222111+11232110+12232100}$} &  \multicolumn{1}{p{1.69cm}|}{\centering 30} &  \multicolumn{1}{p{1.47cm}|}{\centering 00001000} \\
 \cline{1-1} \cline{2-2} \cline{3-3}
 \multicolumn{1}{|p{7.53cm}|}{\centering $\nrel{01122221+11122211+11222210+}{+11232111+12232110}$} &  \multicolumn{1}{p{1.69cm}|}{\centering 29} &  \multicolumn{1}{p{1.47cm}|}{\centering 10000010} \\
 \cline{1-1} \cline{2-2} \cline{3-3}
 \multicolumn{1}{|p{7.53cm}|}{\centering 11122211+11222210+11232111+12232110} &  \multicolumn{1}{p{1.69cm}|}{\centering 28} &  \multicolumn{1}{p{1.47cm}|}{\centering 20000000} \\
 \cline{1-1} \cline{2-2} \cline{3-3}
 \multicolumn{1}{|p{7.53cm}|}{\centering $\nrel{11122221+11222211+11232210+}{+11232111+12232110}$} &  \multicolumn{1}{p{1.69cm}|}{\centering 27} &  \multicolumn{1}{p{1.47cm}|}{\centering 00100000} \\
 \cline{1-1} \cline{2-2} \cline{3-3}
 \multicolumn{1}{|p{7.53cm}|}{\centering 11122221+11222211+11232210+12232111} &  \multicolumn{1}{p{1.69cm}|}{\centering 25} &  \multicolumn{1}{p{1.47cm}|}{\centering 00000100} \\
 \cline{1-1} \cline{2-2} \cline{3-3}
 \multicolumn{1}{|p{7.53cm}|}{\centering 11222211+11232210+11232111+12232110} &  \multicolumn{1}{p{1.69cm}|}{\centering 24} &  \multicolumn{1}{p{1.47cm}|}{\centering 00000100} \\
 \cline{1-1} \cline{2-2} \cline{3-3}
 \multicolumn{1}{|p{7.53cm}|}{\centering 11222221+11232210+12232111} &  \multicolumn{1}{p{1.69cm}|}{\centering 23} &  \multicolumn{1}{p{1.47cm}|}{\centering 10000001} \\
 \cline{1-1} \cline{2-2} \cline{3-3}
 \multicolumn{1}{|p{7.53cm}|}{\centering 11222221+11232211+12232210+12232111} &  \multicolumn{1}{p{1.69cm}|}{\centering 21} &  \multicolumn{1}{p{1.47cm}|}{\centering 01000000} \\
 \cline{1-1} \cline{2-2} \cline{3-3}
 \multicolumn{1}{|p{7.53cm}|}{\centering 11232210+12232111} &  \multicolumn{1}{p{1.69cm}|}{\centering 18} &  \multicolumn{1}{p{1.47cm}|}{\centering 00000002} \\
 \cline{1-1} \cline{2-2} \cline{3-3}
 \multicolumn{1}{|p{7.53cm}|}{\centering 11222221+11232211+12232210} &  \multicolumn{1}{p{1.69cm}|}{\centering 18} &  \multicolumn{1}{p{1.47cm}|}{\centering 00000010} \\
 \cline{1-1} \cline{2-2} \cline{3-3}
 \multicolumn{1}{|p{7.53cm}|}{\centering 11222221+11232211+12232111} &  \multicolumn{1}{p{1.69cm}|}{\centering 18} &  \multicolumn{1}{p{1.47cm}|}{\centering 00000010} \\
 \cline{1-1} \cline{2-2} \cline{3-3}
 \multicolumn{1}{|p{7.53cm}|}{\centering 11232221+12232210+12232111} &  \multicolumn{1}{p{1.69cm}|}{\centering 17} &  \multicolumn{1}{p{1.47cm}|}{\centering 00000010} \\
 \cline{1-1} \cline{2-2} \cline{3-3}
 \multicolumn{1}{|p{7.53cm}|}{\centering 11232221+12232211} &  \multicolumn{1}{p{1.69cm}|}{\centering 15} &  \multicolumn{1}{p{1.47cm}|}{\centering 10000000} \\
 \cline{1-1} \cline{2-2} \cline{3-3}
 \multicolumn{1}{|p{7.53cm}|}{\centering 12232210+12232111} &  \multicolumn{1}{p{1.69cm}|}{\centering 10} &  \multicolumn{1}{p{1.47cm}|}{\centering 10000000} \\
 \cline{1-1} \cline{2-2} \cline{3-3}
 \multicolumn{1}{|p{7.53cm}|}{\centering 12232221} &  \multicolumn{1}{p{1.69cm}|}{\centering 9} &  \multicolumn{1}{p{1.47cm}|}{\centering 00000001} \\
 \cline{1-1} \cline{2-2} \cline{3-3}
 \multicolumn{1}{|p{7.53cm}|}{\centering 00000000} &  \multicolumn{1}{p{1.69cm}|}{\centering 0} &  \multicolumn{1}{p{1.47cm}|}{\centering 00000000} \\
\hline
\end{tabular}\end{center}

The action on $\frak u_3$ is the tensor product of the standard representations of $SL(5)$ and $SL(4)$, and arises for $SL(9)$, node 4.  It has 5 orbits, classified by rank, with dimensions 20, 18, 14, 8, and 0, and respective basepoints 00011111 + 00111110 + 01111100 + 11111000, 00111111 + 01111110 + 11111100, 01111111 + 11111110, 11111111, and 00000000, respectively (in $\frak a_8$).  The orbits here are as follows:
\begin{center}\begin{tabular}{||p{6.66cm}p{1.69cm}||p{1.72cm}||}
\hline
 \multicolumn{1}{|p{6.66cm}|}{\centering Orbit Basepoint} &  \multicolumn{1}{p{1.69cm}|}{\centering Dimension} &  \multicolumn{1}{p{1.72cm}|}{\centering Coadjoint orbit intersected} \\
\hline
 \multicolumn{1}{|p{6.66cm}|}{\centering 12233321+12243221+12343211+22343210} &  \multicolumn{1}{p{1.69cm}|}{\centering 20} &  \multicolumn{1}{p{1.72cm}|}{\centering 01000000} \\
 \cline{1-1} \cline{2-2} \cline{3-3}
 \multicolumn{1}{|p{6.66cm}|}{\centering 12243321+12343221+22343211} &  \multicolumn{1}{p{1.69cm}|}{\centering 18} &  \multicolumn{1}{p{1.72cm}|}{\centering 00000010} \\
 \cline{1-1} \cline{2-2} \cline{3-3}
 \multicolumn{1}{|p{6.66cm}|}{\centering 12343321+22343221} &  \multicolumn{1}{p{1.69cm}|}{\centering 14} &  \multicolumn{1}{p{1.72cm}|}{\centering 10000000} \\
 \cline{1-1} \cline{2-2} \cline{3-3}
 \multicolumn{1}{|p{6.66cm}|}{\centering 22343321} &  \multicolumn{1}{p{1.69cm}|}{\centering 8} &  \multicolumn{1}{p{1.72cm}|}{\centering 00000001} \\
 \cline{1-1} \cline{2-2} \cline{3-3}
 \multicolumn{1}{|p{6.66cm}|}{\centering 00000000} &  \multicolumn{1}{p{1.69cm}|}{\centering 0} &  \multicolumn{1}{p{1.72cm}|}{\centering 00000000} \\
\hline
\end{tabular}\end{center}

The action on $\frak u_4$ is the exterior square representation of $SL(5)$, and arises for $SO(5,5)$, node 5 (see section~\ref{sec:D5node4}). It has three orbits, of dimensions 10, 7, and 0, with respective basepoints 01211 + 11111, 12211, and 00000 there.  The orbits here have the following basepoints:
\begin{center}\begin{tabular}{|cp{1.66cm}||p{2.59cm}||}
\hline
 \multicolumn{1}{|c|}{Orbit Basepoint} &  \multicolumn{1}{p{1.66cm}|}{\centering Dimension} &  \multicolumn{1}{p{2.59cm}|}{\centering Coadjoint orbit intersected} \\
\hline
 \multicolumn{1}{|c|}{22454321+23354321} &  \multicolumn{1}{p{1.66cm}|}{\centering 10} &  \multicolumn{1}{p{2.59cm}|}{\centering 10000000} \\
 \cline{1-1} \cline{2-2} \cline{3-3}
 \multicolumn{1}{|c|}{23464321} &  \multicolumn{1}{p{1.66cm}|}{\centering 7} &  \multicolumn{1}{p{2.59cm}|}{\centering 00000001} \\
 \cline{1-1} \cline{2-2} \cline{3-3}
 \multicolumn{1}{|c|}{00000000} &  \multicolumn{1}{p{1.66cm}|}{\centering 0} &  \multicolumn{1}{p{2.59cm}|}{\centering 00000000} \\
\hline
\end{tabular}\end{center}

The action on $\frak u_5$ is the standard representation of $SL(4)$, and has two orbits:~zero and nonzero.

\subsubsection{$E_8$ Node 6}

Here $P=LU$ where $U$ is a 97-dimensional 4-step nilpotent group and $\frak u=\frak u_1\oplus \frak u_2\oplus \frak u_3\oplus \frak u_4$, with $\dim\frak u_1=48$, $\dim\frak u_2=30$, $\dim\frak u_3=16$,   and $\dim\frak u_4=3$.  The semisimple part $[L,L]$ of $L$ is of type $SO(5,5)\times SL(3)$, which acts on $\frak u_1$ as the tensor product of the spin representation of the $SO(5,5)$ factor with the standard representation of the $SL(3)$ factor.  The orbits are given as follows:

\begin{center}\begin{tabular}{p{7.33cm}p{1.63cm}p{1.75cm}}
 \cline{1-1} \cline{2-2} \cline{3-3}
 \multicolumn{1}{|p{7.33cm}|}{\centering Orbit Basepoint} &  \multicolumn{1}{p{1.63cm}|}{\centering Dimension} &  \multicolumn{1}{p{1.75cm}|}{\centering Coadjoint orbit intersected} \\
\cline{1-1} \cline{2-2} \cline{3-3}
 \multicolumn{1}{|p{7.33cm}|}{\centering $\nrel{00111111+01011111+01122110+}{+10111110+11122100+11221100}$} &  \multicolumn{1}{p{1.63cm}|}{\centering 48} &  \multicolumn{1}{p{1.75cm}|}{\centering 00000200} \\
 \cline{1-1} \cline{2-2} \cline{3-3}
 \multicolumn{1}{|p{7.33cm}|}{\centering $\nrel{01011111+01122110+10111111+}{+11121110+11122100+11221100}$} &  \multicolumn{1}{p{1.63cm}|}{\centering 47} &  \multicolumn{1}{p{1.75cm}|}{\centering 00010001} \\
 \cline{1-1} \cline{2-2} \cline{3-3}
 \multicolumn{1}{|p{7.33cm}|}{\centering $\nrel{00111111+01011111+01122110+}{+10111110+12232100}$} &  \multicolumn{1}{p{1.63cm}|}{\centering 45} &  \multicolumn{1}{p{1.75cm}|}{\centering 10000101} \\
 \cline{1-1} \cline{2-2} \cline{3-3}
 \multicolumn{1}{|p{7.33cm}|}{\centering $\nrel{01121111+01122110+10111111+}{+11111110+11122100+11221100}$} &  \multicolumn{1}{p{1.63cm}|}{\centering 45} &  \multicolumn{1}{p{1.75cm}|}{\centering 10001000} \\
 \cline{1-1} \cline{2-2} \cline{3-3}
 \multicolumn{1}{|p{7.33cm}|}{\centering $\nrel{01111111+01121111+01122110+}{+10111111+11121110+11222100}$} &  \multicolumn{1}{p{1.63cm}|}{\centering 43} &  \multicolumn{1}{p{1.75cm}|}{\centering 02000000} \\
 \cline{1-1} \cline{2-2} \cline{3-3}
 \multicolumn{1}{|p{7.33cm}|}{\centering $\nrel{01121111+01122110+11111111+}{+11121110+11122100+11221100}$} &  \multicolumn{1}{p{1.63cm}|}{\centering 43} &  \multicolumn{1}{p{1.75cm}|}{\centering 00010000} \\
 \cline{1-1} \cline{2-2} \cline{3-3}
 \multicolumn{1}{|p{7.33cm}|}{\centering  00111111+01011111+10111110+12232100} &  \multicolumn{1}{p{1.63cm}|}{\centering 42} &  \multicolumn{1}{p{1.75cm}|}{\centering 20000002} \\
 \cline{1-1} \cline{2-2} \cline{3-3}
 \multicolumn{1}{|p{7.33cm}|}{\centering $\nrel{01121111+01122110+11111111+}{+11121110+11222100}$} &  \multicolumn{1}{p{1.63cm}|}{\centering 42} &  \multicolumn{1}{p{1.75cm}|}{\centering 10000100} \\
 \cline{1-1} \cline{2-2} \cline{3-3}
 \multicolumn{1}{|p{7.33cm}|}{\centering $\nrel{01011111+01122110+10111111+}{+11221110+12232100}$} &  \multicolumn{1}{p{1.63cm}|}{\centering 41} &  \multicolumn{1}{p{1.75cm}|}{\centering 10000100} \\
 \cline{1-1} \cline{2-2} \cline{3-3}
 \multicolumn{1}{|p{7.33cm}|}{\centering $\nrel{01121111+01122111+11111111+}{+11122110+11221100}$} &  \multicolumn{1}{p{1.63cm}|}{\centering 41} &  \multicolumn{1}{p{1.75cm}|}{\centering 01000010} \\
 \cline{1-1} \cline{2-2} \cline{3-3}
 \multicolumn{1}{|p{7.33cm}|}{\centering $\nrel{01121111+01122110+11111111+}{+11221110+11232100}$} &  \multicolumn{1}{p{1.63cm}|}{\centering 40} &  \multicolumn{1}{p{1.75cm}|}{\centering 00100001} \\
 \cline{1-1} \cline{2-2} \cline{3-3}
 \multicolumn{1}{|p{7.33cm}|}{\centering 01121111+01122111+11122110+11221100} &  \multicolumn{1}{p{1.63cm}|}{\centering 39} &  \multicolumn{1}{p{1.75cm}|}{\centering 00000020} \\
 \cline{1-1} \cline{2-2} \cline{3-3}
 \multicolumn{1}{|p{7.33cm}|}{\centering 01122110+11121111+11221110+11222100} &  \multicolumn{1}{p{1.63cm}|}{\centering 38} &  \multicolumn{1}{p{1.75cm}|}{\centering 00000101} \\
 \cline{1-1} \cline{2-2} \cline{3-3}
 \multicolumn{1}{|p{7.33cm}|}{\centering 01121111+01122110+11111111+11232100} &  \multicolumn{1}{p{1.63cm}|}{\centering 37} &  \multicolumn{1}{p{1.75cm}|}{\centering 00000101} \\
 \cline{1-1} \cline{2-2} \cline{3-3}
 \multicolumn{1}{|p{7.33cm}|}{\centering $\nrel{01122111+10111111+11122110+}{+11221110+12232100}$} &  \multicolumn{1}{p{1.63cm}|}{\centering 37} &  \multicolumn{1}{p{1.75cm}|}{\centering 10000010} \\
 \cline{1-1} \cline{2-2} \cline{3-3}
 \multicolumn{1}{|p{7.33cm}|}{\centering 01111111+10111111+11232110+12232100} &  \multicolumn{1}{p{1.63cm}|}{\centering 35} &  \multicolumn{1}{p{1.75cm}|}{\centering 20000000} \\
 \cline{1-1} \cline{2-2} \cline{3-3}
 \multicolumn{1}{|p{7.33cm}|}{\centering $\nrel{01122111+11121111+11122110+}{+11221110+12232100}$} &  \multicolumn{1}{p{1.63cm}|}{\centering 35} &  \multicolumn{1}{p{1.75cm}|}{\centering 00100000} \\
 \cline{1-1} \cline{2-2} \cline{3-3}
 \multicolumn{1}{|p{7.33cm}|}{\centering 00111111+01011111+11122110+11221110} &  \multicolumn{1}{p{1.63cm}|}{\centering 34} &  \multicolumn{1}{p{1.75cm}|}{\centering 20000000} \\
 \cline{1-1} \cline{2-2} \cline{3-3}
 \multicolumn{1}{|p{7.33cm}|}{\centering 01122110+11122100+11221111} &  \multicolumn{1}{p{1.63cm}|}{\centering 34} &  \multicolumn{1}{p{1.75cm}|}{\centering 10000002} \\
 \cline{1-1} \cline{2-2} \cline{3-3}
 \multicolumn{1}{|p{7.33cm}|}{\centering 01122111+11121111+11122110+11221110} &  \multicolumn{1}{p{1.63cm}|}{\centering 33} &  \multicolumn{1}{p{1.75cm}|}{\centering 00000100} \\
 \cline{1-1} \cline{2-2} \cline{3-3}
\end{tabular}\end{center}

\begin{center}\begin{tabular}{p{7.31cm}p{1.63cm}p{1.75cm}}
 \cline{1-1} \cline{2-2} \cline{3-3}
 \multicolumn{1}{|p{7.31cm}|}{\centering $\nrel{\text{Orbit Basepoint}}{\text{(Continued)}}$} &  \multicolumn{1}{p{1.63cm}|}{\centering Dimension} &  \multicolumn{1}{p{1.75cm}|}{\centering Coadjoint orbit intersected} \\
\cline{1-1} \cline{2-2} \cline{3-3}
 \multicolumn{1}{|p{7.31cm}|}{\centering 01122111+11111111+11232110+12232100} &  \multicolumn{1}{p{1.63cm}|}{\centering 33} &  \multicolumn{1}{p{1.75cm}|}{\centering 00000100} \\
 \cline{1-1} \cline{2-2} \cline{3-3}
 \multicolumn{1}{|p{7.31cm}|}{\centering 10111111+11122110+11221110+12232100} &  \multicolumn{1}{p{1.63cm}|}{\centering 32} &  \multicolumn{1}{p{1.75cm}|}{\centering 00000100} \\
 \cline{1-1} \cline{2-2} \cline{3-3}
 \multicolumn{1}{|p{7.31cm}|}{\centering 11111111+11232110+12232100} &  \multicolumn{1}{p{1.63cm}|}{\centering 31} &  \multicolumn{1}{p{1.75cm}|}{\centering 10000001} \\
 \cline{1-1} \cline{2-2} \cline{3-3}
 \multicolumn{1}{|p{7.31cm}|}{\centering 01122111+10111111+12232110} &  \multicolumn{1}{p{1.63cm}|}{\centering 30} &  \multicolumn{1}{p{1.75cm}|}{\centering 10000001} \\
 \cline{1-1} \cline{2-2} \cline{3-3}
 \multicolumn{1}{|p{7.31cm}|}{\centering 11122111+11221111+11232110+12232100} &  \multicolumn{1}{p{1.63cm}|}{\centering 29} &  \multicolumn{1}{p{1.75cm}|}{\centering 01000000} \\
 \cline{1-1} \cline{2-2} \cline{3-3}
 \multicolumn{1}{|p{7.31cm}|}{\centering 10111111+12232110} &  \multicolumn{1}{p{1.63cm}|}{\centering 26} &  \multicolumn{1}{p{1.75cm}|}{\centering 00000002} \\
 \cline{1-1} \cline{2-2} \cline{3-3}
 \multicolumn{1}{|p{7.31cm}|}{\centering 11122111+11221111+12232110} &  \multicolumn{1}{p{1.63cm}|}{\centering 25} &  \multicolumn{1}{p{1.75cm}|}{\centering 00000010} \\
 \cline{1-1} \cline{2-2} \cline{3-3}
 \multicolumn{1}{|p{7.31cm}|}{\centering 11222111+11232110+12232100} &  \multicolumn{1}{p{1.63cm}|}{\centering 25} &  \multicolumn{1}{p{1.75cm}|}{\centering 00000010} \\
 \cline{1-1} \cline{2-2} \cline{3-3}
 \multicolumn{1}{|p{7.31cm}|}{\centering 11232111+12232110} &  \multicolumn{1}{p{1.63cm}|}{\centering 21} &  \multicolumn{1}{p{1.75cm}|}{\centering 10000000} \\
 \cline{1-1} \cline{2-2} \cline{3-3}
 \multicolumn{1}{|p{7.31cm}|}{\centering 11122111+11221111} &  \multicolumn{1}{p{1.63cm}|}{\centering 18} &  \multicolumn{1}{p{1.75cm}|}{\centering 10000000} \\
 \cline{1-1} \cline{2-2} \cline{3-3}
 \multicolumn{1}{|p{7.31cm}|}{\centering 12232111} &  \multicolumn{1}{p{1.63cm}|}{\centering 13} &  \multicolumn{1}{p{1.75cm}|}{\centering 00000001} \\
 \cline{1-1} \cline{2-2} \cline{3-3}
 \multicolumn{1}{|p{7.31cm}|}{\centering 00000000} &  \multicolumn{1}{p{1.63cm}|}{\centering 0} &  \multicolumn{1}{p{1.75cm}|}{\centering 00000000} \\
 \cline{1-1} \cline{2-2} \cline{3-3}
\end{tabular}\end{center}

The action on $\frak u_2$ is the tensor product of the 10-dimensional vector representation of the $SO(5,5)$ factor with the standard representation of $SL(3)$.  It occurs for $SO(8,8)$ node 3, and has   orbits of dimensions 30, 29, 27, 24, 22, 21, 19, 12, 11, and 0, with respective basepoints
     00111111 + 01111101 + 01111110 + 11111100,
     00122211 + 01111101 + 11111110,
     00122211 + 01112211 + 11111101 + 11111110,
     00122211 + 01112211 + 11111211,
     01111101 + 11111110,
     01122211 + 11111101 + 11111110,
     01122211 + 11112211,
     11111101 + 11111110,
          11122211, and
     00000000
(in $\frak{so}(8,8)$).
 The orbits here are as follows:
\begin{center}\begin{tabular}{|cp{1.66cm}||p{1.47cm}||}
\hline
 \multicolumn{1}{|c|}{Orbit Basepoint} &  \multicolumn{1}{p{1.66cm}|}{\centering Dimension} &  \multicolumn{1}{p{1.47cm}|}{\centering Coadjoint orbit intersected} \\
\hline
 \multicolumn{1}{|c|}{12233210+11233211+12232211+11232221} &  \multicolumn{1}{p{1.66cm}|}{\centering 30} &  \multicolumn{1}{p{1.47cm}|}{\centering 00000100} \\
 \cline{1-1} \cline{2-2} \cline{3-3}
 \multicolumn{1}{|c|}{22343210+11233211+12232221} &  \multicolumn{1}{p{1.66cm}|}{\centering 29} &  \multicolumn{1}{p{1.47cm}|}{\centering 10000001} \\
 \cline{1-1} \cline{2-2} \cline{3-3}
 \multicolumn{1}{|c|}{22343210+12343211+11233221+12232221} &  \multicolumn{1}{p{1.66cm}|}{\centering 27} &  \multicolumn{1}{p{1.47cm}|}{\centering 01000000} \\
 \cline{1-1} \cline{2-2} \cline{3-3}
 \multicolumn{1}{|c|}{22343210+12343211+12243221} &  \multicolumn{1}{p{1.66cm}|}{\centering 24} &  \multicolumn{1}{p{1.47cm}|}{\centering 00000010} \\
 \cline{1-1} \cline{2-2} \cline{3-3}
 \multicolumn{1}{|c|}{11233211+12232221} &  \multicolumn{1}{p{1.66cm}|}{\centering 22} &  \multicolumn{1}{p{1.47cm}|}{\centering 00000002} \\
 \cline{1-1} \cline{2-2} \cline{3-3}
 \multicolumn{1}{|c|}{22343211+11233221+12232221} &  \multicolumn{1}{p{1.66cm}|}{\centering 21} &  \multicolumn{1}{p{1.47cm}|}{\centering 00000010} \\
 \cline{1-1} \cline{2-2} \cline{3-3}
 \multicolumn{1}{|c|}{22343211+12343221} &  \multicolumn{1}{p{1.66cm}|}{\centering 19} &  \multicolumn{1}{p{1.47cm}|}{\centering 10000000} \\
 \cline{1-1} \cline{2-2} \cline{3-3}
 \multicolumn{1}{|c|}{11233221+12232221} &  \multicolumn{1}{p{1.66cm}|}{\centering 12} &  \multicolumn{1}{p{1.47cm}|}{\centering 10000000} \\
 \cline{1-1} \cline{2-2} \cline{3-3}
 \multicolumn{1}{|c|}{22343221} &  \multicolumn{1}{p{1.66cm}|}{\centering 11} &  \multicolumn{1}{p{1.47cm}|}{\centering 00000001} \\
 \cline{1-1} \cline{2-2} \cline{3-3}
 \multicolumn{1}{|c|}{00000000} &  \multicolumn{1}{p{1.66cm}|}{\centering 0} &  \multicolumn{1}{p{1.47cm}|}{\centering 00000000} \\
\hline
\end{tabular}\end{center}

The action on $\frak u_3$ is the 16 dimensional spin representation of $SO(5,5)$, and occurs earlier for $E_6$, node 1 (see section~\ref{sec:E6node1}).  It has the following orbits:
\begin{center}\begin{tabular}{|c|p{1.75cm}||p{2.94cm}||}
\hline
 Orbit Basepoint &  \multicolumn{1}{p{1.75cm}|}{\centering Dimension} &  \multicolumn{1}{p{2.94cm}|}{\centering Coadjoint orbit intersected} \\
\hline
 \multicolumn{1}{|c|}{22454321+23354321} &  \multicolumn{1}{p{1.75cm}|}{\centering 16} &  \multicolumn{1}{p{2.94cm}|}{\centering 10000000} \\
 \cline{1-1} \cline{2-2} \cline{3-3}
 \multicolumn{1}{|c|}{23465321} &  \multicolumn{1}{p{1.75cm}|}{\centering 11} &  \multicolumn{1}{p{2.94cm}|}{\centering 00000001} \\
 \cline{1-1} \cline{2-2} \cline{3-3}
 00000000 &  \multicolumn{1}{p{1.75cm}|}{\centering 0} &  \multicolumn{1}{p{2.94cm}|}{\centering 00000000} \\
\hline
\end{tabular}\end{center}

The action on $\frak u_4$ is the 3 dimensional standard representation of $GL(3)$, and has two orbits:~zero and non-zero.

\subsubsection{$E_8$ Node 7}

Here $P=LU$ where $U$ is a 83-dimensional 3-step nilpotent group and $\frak u=\frak u_1\oplus \frak u_2\oplus \frak u_3$, with $\dim\frak u_1=54$, $\dim\frak u_2=27$,      and $\dim\frak u_3=2$.  The semisimple part $[L,L]$ of $L$ is of type $E_6\times SL(2)$, which acts on $\frak u_1$ as the tensor product of the (minimal) 27-dimensional representation of the $E_6$ factor with the standard representation of the $SL(2)$ factor.  The orbits are given as follows:

\begin{center}\begin{tabular}{p{7.44cm}p{1.63cm}p{1.69cm}}
 \cline{1-1} \cline{2-2} \cline{3-3}
 \multicolumn{1}{|p{7.44cm}|}{\centering Orbit Basepoint} &  \multicolumn{1}{p{1.63cm}|}{\centering Dimension} &  \multicolumn{1}{p{1.69cm}|}{\centering Coadjoint orbit intersected} \\
\cline{1-1} \cline{2-2} \cline{3-3}
 \multicolumn{1}{|p{7.44cm}|}{\centering 01122210+01122211+11122111+11221110} &  \multicolumn{1}{p{1.63cm}|}{\centering 54} &  \multicolumn{1}{p{1.69cm}|}{\centering 00000020} \\
 \cline{1-1} \cline{2-2} \cline{3-3}
 \multicolumn{1}{|p{7.44cm}|}{\centering 01122111+10111111+11233210+12232210} &  \multicolumn{1}{p{1.63cm}|}{\centering 53} &  \multicolumn{1}{p{1.69cm}|}{\centering 00000101} \\
 \cline{1-1} \cline{2-2} \cline{3-3}
 \multicolumn{1}{|p{7.44cm}|}{\centering $\nrel{01122211+11122111+11221111+}{+11233210+12232210}$} &  \multicolumn{1}{p{1.63cm}|}{\centering 52} &  \multicolumn{1}{p{1.69cm}|}{\centering 10000010} \\
 \cline{1-1} \cline{2-2} \cline{3-3}
 \multicolumn{1}{|p{7.44cm}|}{\centering 11122111+11221111+11233210+12232210} &  \multicolumn{1}{p{1.63cm}|}{\centering 50} &  \multicolumn{1}{p{1.69cm}|}{\centering 20000000} \\
 \cline{1-1} \cline{2-2} \cline{3-3}
 \multicolumn{1}{|p{7.44cm}|}{\centering 11232211+11233210+12232111+12232210} &  \multicolumn{1}{p{1.63cm}|}{\centering 47} &  \multicolumn{1}{p{1.69cm}|}{\centering 00000100} \\
 \cline{1-1} \cline{2-2} \cline{3-3}
 \multicolumn{1}{|p{7.44cm}|}{\centering 00111111+01011111+22343210} &  \multicolumn{1}{p{1.63cm}|}{\centering 45} &  \multicolumn{1}{p{1.69cm}|}{\centering 10000002} \\
 \cline{1-1} \cline{2-2} \cline{3-3}
 \multicolumn{1}{|p{7.44cm}|}{\centering 01122211+11122111+11221111+22343210} &  \multicolumn{1}{p{1.63cm}|}{\centering 44} &  \multicolumn{1}{p{1.69cm}|}{\centering 00000100} \\
 \cline{1-1} \cline{2-2} \cline{3-3}
 \multicolumn{1}{|p{7.44cm}|}{\centering 01122211+12232111+22343210} &  \multicolumn{1}{p{1.63cm}|}{\centering 43} &  \multicolumn{1}{p{1.69cm}|}{\centering 10000001} \\
 \cline{1-1} \cline{2-2} \cline{3-3}
 \multicolumn{1}{|p{7.44cm}|}{\centering 01122211+22343210} &  \multicolumn{1}{p{1.63cm}|}{\centering 36} &  \multicolumn{1}{p{1.69cm}|}{\centering 00000002} \\
 \cline{1-1} \cline{2-2} \cline{3-3}
 \multicolumn{1}{|p{7.44cm}|}{\centering 11233211+12232211+22343210} &  \multicolumn{1}{p{1.63cm}|}{\centering 35} &  \multicolumn{1}{p{1.69cm}|}{\centering 00000010} \\
 \cline{1-1} \cline{2-2} \cline{3-3}
 \multicolumn{1}{|p{7.44cm}|}{\centering 12343211+22343210} &  \multicolumn{1}{p{1.63cm}|}{\centering 29} &  \multicolumn{1}{p{1.69cm}|}{\centering 10000000} \\
 \cline{1-1} \cline{2-2} \cline{3-3}
 \multicolumn{1}{|p{7.44cm}|}{\centering 01122211+11122111+11221111} &  \multicolumn{1}{p{1.63cm}|}{\centering 28} &  \multicolumn{1}{p{1.69cm}|}{\centering 00000010} \\
 \cline{1-1} \cline{2-2} \cline{3-3}
 \multicolumn{1}{|p{7.44cm}|}{\centering 11233211+12232211} &  \multicolumn{1}{p{1.63cm}|}{\centering 27} &  \multicolumn{1}{p{1.69cm}|}{\centering 10000000} \\
 \cline{1-1} \cline{2-2} \cline{3-3}
 \multicolumn{1}{|p{7.44cm}|}{\centering 22343211} &  \multicolumn{1}{p{1.63cm}|}{\centering 18} &  \multicolumn{1}{p{1.69cm}|}{\centering 00000001} \\
 \cline{1-1} \cline{2-2} \cline{3-3}
 \multicolumn{1}{|p{7.44cm}|}{\centering 00000000} &  \multicolumn{1}{p{1.63cm}|}{\centering 0} &  \multicolumn{1}{p{1.69cm}|}{\centering 00000000} \\
 \cline{1-1} \cline{2-2} \cline{3-3}
\end{tabular}\end{center}

The action on $\frak u_2$ is the 27-dimensional representation of $E_6$, which occurs for $E_7$, node $7$ (see section~\ref{sec:E7node7}).  It has the following orbits here:
\begin{center}\begin{tabular}{||p{5.66cm}p{1.78cm}p{3.29cm}}
 \cline{1-1} \cline{2-2} \cline{3-3}
 \multicolumn{1}{|p{5.66cm}|}{\centering Orbit Basepoint} &  \multicolumn{1}{p{1.78cm}|}{\centering Dimension} &  \multicolumn{1}{p{3.29cm}|}{\centering Coadjoint orbit intersected} \\
\cline{1-1} \cline{2-2} \cline{3-3}
 \multicolumn{1}{|p{5.66cm}|}{\centering 22343221+12343321+12244321} &  \multicolumn{1}{p{1.78cm}|}{\centering 27} &  \multicolumn{1}{p{3.29cm}|}{\centering 00000010} \\
 \cline{1-1} \cline{2-2} \cline{3-3}
 \multicolumn{1}{|p{5.66cm}|}{\centering 22454321+23354321} &  \multicolumn{1}{p{1.78cm}|}{\centering 26} &  \multicolumn{1}{p{3.29cm}|}{\centering 10000000} \\
 \cline{1-1} \cline{2-2} \cline{3-3}
 \multicolumn{1}{|p{5.66cm}|}{\centering 23465421} &  \multicolumn{1}{p{1.78cm}|}{\centering 17} &  \multicolumn{1}{p{3.29cm}|}{\centering 00000001} \\
 \cline{1-1} \cline{2-2} \cline{3-3}
 \multicolumn{1}{|p{5.66cm}|}{\centering 00000000} &  \multicolumn{1}{p{1.78cm}|}{\centering 0} &  \multicolumn{1}{p{3.29cm}|}{\centering 00000000} \\
\hline
\end{tabular}\end{center}

The action on $\frak u_3$ is the standard action of $GL(2)$, and has 2-orbits:~zero and nonzero.

\subsubsection{$E_8$ Node 8}

Here $P=LU$ where $U$ is a 57-dimensional Heisenberg group and $\frak u=\frak u_1\oplus \frak u_2$, with $\dim\frak u_1=56$  and $\dim\frak u_2=1$.  The semisimple part $[L,L]$ of $L$ is of type $E_7$, which acts on $\frak u_1$ as its (minimal) 56-dimensional representation.

\begin{center}\begin{tabular}{p{6.19cm}p{1.66cm}p{2.91cm}}
 \cline{1-1} \cline{2-2} \cline{3-3}
 \multicolumn{1}{|p{6.19cm}|}{\centering Orbit Basepoint} &  \multicolumn{1}{p{1.66cm}|}{\centering Dimension} &  \multicolumn{1}{p{2.91cm}|}{\centering Coadjoint orbit intersected} \\
\cline{1-1} \cline{2-2} \cline{3-3}
 \multicolumn{1}{|p{6.19cm}|}{\centering 01122221+22343211} &  \multicolumn{1}{p{1.66cm}|}{\centering 56} &  \multicolumn{1}{p{2.91cm}|}{\centering 00000002} \\
 \cline{1-1} \cline{2-2} \cline{3-3}
 \multicolumn{1}{|p{6.19cm}|}{\centering 12244321+12343321+22343221} &  \multicolumn{1}{p{1.66cm}|}{\centering 55} &  \multicolumn{1}{p{2.91cm}|}{\centering 00000010} \\
 \cline{1-1} \cline{2-2} \cline{3-3}
 \multicolumn{1}{|p{6.19cm}|}{\centering 22454321+23354321} &  \multicolumn{1}{p{1.66cm}|}{\centering 45} &  \multicolumn{1}{p{2.91cm}|}{\centering 10000000} \\
 \cline{1-1} \cline{2-2} \cline{3-3}
 \multicolumn{1}{|p{6.19cm}|}{\centering 23465431} &  \multicolumn{1}{p{1.66cm}|}{\centering 28} &  \multicolumn{1}{p{2.91cm}|}{\centering 00000001} \\
 \cline{1-1} \cline{2-2} \cline{3-3}
 \multicolumn{1}{|p{6.19cm}|}{\centering 00000000} &  \multicolumn{1}{p{1.66cm}|}{\centering 0} &  \multicolumn{1}{p{2.91cm}|}{\centering 00000000} \\
 \cline{1-1} \cline{2-2} \cline{3-3}
\end{tabular}\end{center}

The action on the one-dimensional piece $\frak u_2$ has two orbits:~zero and nonzero.

\subsection{Type $F_4$}

\begin{center}
\begin{figure}
  \includegraphics[scale=.5]{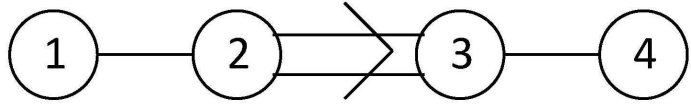}\\
  \caption{$F_4$ Dynkin diagram.}\label{F4Dynkin}
\end{figure}
\end{center}

The following table lists the internal Chevalley modules
that arise for maximal parabolic subgroups of $F_4$.  We also indicate where the higher graded pieces arise earlier, except for those which are the standard actions of $SL(n)$. We write the 3rd fundamental representation of $Sp(6)$ as {\bf 14} (though caution the reader that the 2nd fundamental representation has the same dimension).  The numbering and labeling conventions for the analogous charts for  $D_5$, $E_6$, $E_7$, and $E_8$ remain in effect.

\begin{center}\begin{tabular}{|c|c|c|c|c|c|}
\hline
 Node &  Type of $[L,L]$ &  $i\,=\,1$ &   $i\,=\,2$ &   $i\,=\,3$ &  $i\,=\,4$ \\
\hline
 1 &  \multicolumn{1}{l|}{$Sp(6)$} &  {\bf 14}  & Trivial  &   &  \\
  &  $\dim \frak u_i$ &  14 &  1 &   &  \\
  &  action &  $\varpi_2$ &   &   &  \\
\hline
 2 &  \multicolumn{1}{l|}{$SL(2)\times SL(3)$} &  Standard $\otimes$ Sym$^2$ & $C_3$ node 3  &   &   \\
  &  $\dim \frak u_i$ &  12 &  6 &  2 &  \\
  &  action &  $\varpi_1+2\varpi_4$ & $2\varpi_3    $  & $\varpi$  &  \\
\hline
 3 &  \multicolumn{1}{l|}{$SL(3)\times SL(2)$} &  Tensor product & $B_4$ node 3  &   &   \\
  &  $\dim \frak u_i$ &  6 &  9 &  2 & 3 \\
  &  action &  $\varpi_1+\varpi_4$ & $\varpi_2+2\varpi_4$  & $\varpi_4$  & $\varpi_1$  \\
\hline
 4 &  \multicolumn{1}{l|}{$SO(7)$} &  Spin & $B_4$ node 1  &   &  \\
  &  $\dim \frak u_i$ &  8 &  7 &   &  \\
  &  action &  $\varpi_3$ & $\varpi_1$  &   &  \\
\hline
\end{tabular}\end{center}

\subsubsection{$F_4$ Node 1}

Here $P=LU$ where $U$ is a 15-dimensional Heisenberg group and $\frak u=\frak u_1\oplus \frak u_2$, with $\dim\frak u_1=14$  and $\dim\frak u_2=1$.  The semisimple part $[L,L]$ of $L$ is of type $Sp(6)$ and  acts on $\frak u_1$ as its  14-dimensional representation 3rd fundamental representation (corresponding to the long root).  Its orbits are as follows.
\begin{center}\begin{tabular}{|l|p{2.28cm}||p{4.22cm}||}
\hline
 Orbit Basepoint &  \multicolumn{1}{p{2.28cm}|}{\centering Dimension} &  \multicolumn{1}{p{4.22cm}|}{\centering Coadjoint orbit intersected} \\
\hline
 \multicolumn{1}{|c|}{1122+1220} &  \multicolumn{1}{p{2.28cm}|}{\centering 14} &  \multicolumn{1}{p{4.22cm}|}{\centering 2000} \\
 \cline{1-1} \cline{2-2} \cline{3-3}
 \multicolumn{1}{|c|}{1222+1231} &  \multicolumn{1}{p{2.28cm}|}{\centering 13} &  \multicolumn{1}{p{4.22cm}|}{\centering 0100} \\
 \cline{1-1} \cline{2-2} \cline{3-3}
 \multicolumn{1}{|c|}{1232} &  \multicolumn{1}{p{2.28cm}|}{\centering 10} &  \multicolumn{1}{p{4.22cm}|}{\centering 0001} \\
 \cline{1-1} \cline{2-2} \cline{3-3}
 \multicolumn{1}{|c|}{1342} &  \multicolumn{1}{p{2.28cm}|}{\centering 7} &  \multicolumn{1}{p{4.22cm}|}{\centering 1000} \\
 \cline{1-1} \cline{2-2} \cline{3-3}
 \multicolumn{1}{|c|}{0000}&  \multicolumn{1}{p{2.28cm}|}{\centering 0} &  \multicolumn{1}{p{4.22cm}|}{\centering 0000} \\
\hline
\end{tabular}\end{center}

The action on the one-dimensional piece $\frak u_2$ has two orbits:~zero and nonzero.

\subsubsection{$F_4$ Node 2}

Here $P=LU$ where $U$ is a 20-dimensional 3-step Heisenberg group and $\frak u=\frak u_1\oplus \frak u_2\oplus \frak u_3$, with $\dim\frak u_1=12$, $\dim\frak u_2=6$,  and $\dim\frak u_3=2$.  The semisimple part $[L,L]$ of $L$ is of type $SL(2)\times SL(3)$ and  acts on $\frak u_1$ as the tensor product of the standard representation of the $SL(2)$ factor with the 6-dimensional symmetric square representation of the $SL(3)$ factor.  It has the following orbits:

\begin{center}\begin{tabular}{||p{3.88cm}||p{1.69cm}||p{3.09cm}||}
\hline
 \multicolumn{1}{|p{3.88cm}|}{\centering Orbit Basepoint} &  \multicolumn{1}{p{1.69cm}|}{\centering Dimension} &  \multicolumn{1}{p{3.09cm}|}{\centering Coadjoint orbit intersected} \\
\hline
 \multicolumn{1}{|p{3.88cm}|}{\centering 0100+0121+1111+1120} &  \multicolumn{1}{p{1.69cm}|}{\centering 12} &  \multicolumn{1}{p{3.09cm}|}{\centering 0200} \\
 \cline{1-1} \cline{2-2} \cline{3-3}
 \multicolumn{1}{|p{3.88cm}|}{\centering 0120+0122+1110} &  \multicolumn{1}{p{1.69cm}|}{\centering 11} &  \multicolumn{1}{p{3.09cm}|}{\centering 1010} \\
 \cline{1-1} \cline{2-2} \cline{3-3}
 \multicolumn{1}{|p{3.88cm}|}{\centering 0122+1110} &  \multicolumn{1}{p{1.69cm}|}{\centering 10} &  \multicolumn{1}{p{3.09cm}|}{\centering 2001} \\
 \cline{1-1} \cline{2-2} \cline{3-3}
 \multicolumn{1}{|p{3.88cm}|}{\centering 0121+1111+1120} &  \multicolumn{1}{p{1.69cm}|}{\centering 10} &  \multicolumn{1}{p{3.09cm}|}{\centering 0101} \\
 \cline{1-1} \cline{2-2} \cline{3-3}
 \multicolumn{1}{|p{3.88cm}|}{\centering 0122+1111+1120} &  \multicolumn{1}{p{1.69cm}|}{\centering 9} &  \multicolumn{1}{p{3.09cm}|}{\centering 0010} \\
 \cline{1-1} \cline{2-2} \cline{3-3}
 \multicolumn{1}{|p{3.88cm}|}{\centering 0121+1111} &  \multicolumn{1}{p{1.69cm}|}{\centering 8} &  \multicolumn{1}{p{3.09cm}|}{\centering 0002} \\
 \cline{1-1} \cline{2-2} \cline{3-3}
 \multicolumn{1}{|p{3.88cm}|}{\centering 0122+1120} &  \multicolumn{1}{p{1.69cm}|}{\centering 8} &  \multicolumn{1}{p{3.09cm}|}{\centering 2000} \\
 \cline{1-1} \cline{2-2} \cline{3-3}
 \multicolumn{1}{|p{3.88cm}|}{\centering 1111+1120} &  \multicolumn{1}{p{1.69cm}|}{\centering 7} &  \multicolumn{1}{p{3.09cm}|}{\centering 0100} \\
 \cline{1-1} \cline{2-2} \cline{3-3}
 \multicolumn{1}{|p{3.88cm}|}{\centering 0122+1121} &  \multicolumn{1}{p{1.69cm}|}{\centering 7} &  \multicolumn{1}{p{3.09cm}|}{\centering 0100} \\
 \cline{1-1} \cline{2-2} \cline{3-3}
 \multicolumn{1}{|p{3.88cm}|}{\centering 1121} &  \multicolumn{1}{p{1.69cm}|}{\centering 6} &  \multicolumn{1}{p{3.09cm}|}{\centering 0001} \\
 \cline{1-1} \cline{2-2} \cline{3-3}
 \multicolumn{1}{|p{3.88cm}|}{\centering 1122} &  \multicolumn{1}{p{1.69cm}|}{\centering 4} &  \multicolumn{1}{p{3.09cm}|}{\centering 1000} \\
 \cline{1-1} \cline{2-2} \cline{3-3}
 \multicolumn{1}{|p{3.88cm}|}{\centering 0000} &  \multicolumn{1}{p{1.69cm}|}{\centering 0} &  \multicolumn{1}{p{3.09cm}|}{\centering 0000} \\
\hline
\end{tabular}\end{center}

The symmetric square action of the $SL(3)$ factor   on $\frak u_2$ comes up earlier  for $Sp(6)$, node 3, and has four orbits there of dimensions 6, 5, 3, and 0, with respective basepoints 021 + 111, 121, 221, and 000.  The orbits here are the following:
\begin{center}\begin{tabular}{|l|p{1.78cm}||p{1.94cm}||}
\hline
 Orbit Basepoint &  \multicolumn{1}{p{1.78cm}|}{\centering Dimension} &  \multicolumn{1}{p{1.94cm}|}{\centering Coadjoint orbit intersected} \\
\hline
 \multicolumn{1}{|c|}{1222+1231} &  \multicolumn{1}{p{1.78cm}|}{\centering 6} &  \multicolumn{1}{p{1.94cm}|}{\centering 0100} \\
 \cline{1-1} \cline{2-2} \cline{3-3}
 \multicolumn{1}{|c|}{1232} &  \multicolumn{1}{p{1.78cm}|}{\centering 5} &  \multicolumn{1}{p{1.94cm}|}{\centering 0001} \\
 \cline{1-1} \cline{2-2} \cline{3-3}
 \multicolumn{1}{|c|}{1242} &  \multicolumn{1}{p{1.78cm}|}{\centering 3} &  \multicolumn{1}{p{1.94cm}|}{\centering 1000} \\
 \cline{1-1} \cline{2-2} \cline{3-3}
 \multicolumn{1}{|c|}{0000} &  \multicolumn{1}{p{1.78cm}|}{\centering 0} &  \multicolumn{1}{p{1.94cm}|}{\centering 0000} \\
\hline
\end{tabular}\end{center}
The action on $\frak u_3$ is the standard action of $SL(2)$ has 2 orbits:~zero and nonzero.

\subsubsection{$F_4$ Node 3}

Here $P=LU$ where $U$ is a 20-dimensional 4-step Heisenberg group and $\frak u=\frak u_1\oplus \frak u_2\oplus \frak u_3\oplus \frak u_4$, with $\dim\frak u_1=6$, $\dim\frak u_2=9$,   $\dim\frak u_3=2$, and $\dim\frak u_4=3$.  The semisimple part $[L,L]$ of $L$ is of type $SL(3)\times SL(2)$ and  acts on $\frak u_1$ as the tensor product of the standard representations of the two factors.  It has the following orbits:
\begin{center}\begin{tabular}{|c|p{1.94cm}||p{4.47cm}||}
\hline
 Orbit Basepoint &  \multicolumn{1}{p{1.94cm}|}{\centering Dimension} &  \multicolumn{1}{p{4.47cm}|}{\centering Coadjoint orbit intersected} \\
\hline
 \multicolumn{1}{|c|}{0111+1110} &  \multicolumn{1}{p{1.94cm}|}{\centering 6} &  \multicolumn{1}{p{4.47cm}|}{\centering 0002} \\
 \cline{1-1} \cline{2-2} \cline{3-3}
 \multicolumn{1}{|c|}{1111} &  \multicolumn{1}{p{1.94cm}|}{\centering 4} &  \multicolumn{1}{p{4.47cm}|}{\centering 0001} \\
 \cline{1-1} \cline{2-2} \cline{3-3}
 0000 &  \multicolumn{1}{p{1.94cm}|}{\centering 0} &  \multicolumn{1}{p{4.47cm}|}{\centering 0000} \\
\hline
\end{tabular}\end{center}

The action on $\frak u_2$ is the tensor product of the standard representation of $SL(3)$ with the symmetric square representation of $SL(2)$, and came up earlier for  $SO(5,4)$, node 3.  It has orbits there of dimensions 9, 8, 7, 5, 4, and 0, with respective basepoints 0012 + 0111 + 1110, 0112 + 1110, 0112 + 1111, 1111, 1112, and 0000. It has the following orbits here:
\begin{center}\begin{tabular}{|l|p{2.25cm}||p{4.28cm}||}
\hline
 Orbit Basepoint &  \multicolumn{1}{p{2.25cm}|}{\centering Dimension} &  \multicolumn{1}{p{4.28cm}|}{\centering Coadjoint orbit intersected} \\
\hline
 \multicolumn{1}{|c|}{0122+1121+1220} &  \multicolumn{1}{p{2.25cm}|}{\centering 9} &  \multicolumn{1}{p{4.28cm}|}{\centering 0010} \\
 \cline{1-1} \cline{2-2} \cline{3-3}
 \multicolumn{1}{|c|}{1122+1220} &  \multicolumn{1}{p{2.25cm}|}{\centering 8} &  \multicolumn{1}{p{4.28cm}|}{\centering 2000} \\
 \cline{1-1} \cline{2-2} \cline{3-3}
 \multicolumn{1}{|c|}{1122+1221} &  \multicolumn{1}{p{2.25cm}|}{\centering 7} &  \multicolumn{1}{p{4.28cm}|}{\centering 0100} \\
 \cline{1-1} \cline{2-2} \cline{3-3}
 \multicolumn{1}{|c|}{1221} &  \multicolumn{1}{p{2.25cm}|}{\centering 5} &  \multicolumn{1}{p{4.28cm}|}{\centering 0001} \\
 \cline{1-1} \cline{2-2} \cline{3-3}
 \multicolumn{1}{|c|}{1222} &  \multicolumn{1}{p{2.25cm}|}{\centering 4} &  \multicolumn{1}{p{4.28cm}|}{\centering 1000} \\
 \cline{1-1} \cline{2-2} \cline{3-3}
 \multicolumn{1}{|c|}{0000} &  \multicolumn{1}{p{2.25cm}|}{\centering 0} &  \multicolumn{1}{p{4.28cm}|}{\centering 0000} \\
\hline
\end{tabular}\end{center}
   The standard actions of $GL(2)$ and $GL(3)$, respectively, on $\frak u_3$ and $\frak u_4$ have two orbits each:~zero and nonzero.
\subsubsection{$F_4$ Node 4}

Here $P=LU$ where $U$ is a 15-dimensional 2-step Heisenberg group and $\frak u=\frak u_1\oplus \frak u_2$, with $\dim\frak u_1=6$, $\dim\frak u_2=8$ and  $\dim\frak u_2=7$.  The semisimple part $[L,L]$ of $L$ is of type $SO(7)$ and  acts on $\frak u_1$ as its spin representation, with the following orbits:
\begin{center}\begin{tabular}{|c|p{1.97cm}||p{4.16cm}||}
\hline
 Orbit Basepoint &  \multicolumn{1}{p{1.97cm}|}{\centering Dimension} &  \multicolumn{1}{p{4.16cm}|}{\centering Coadjoint orbit intersected} \\
\hline
 \multicolumn{1}{|c|}{0121+1111} &  \multicolumn{1}{p{1.97cm}|}{\centering 8} &  \multicolumn{1}{p{4.16cm}|}{\centering 0002} \\
 \cline{1-1} \cline{2-2} \cline{3-3}
 \multicolumn{1}{|c|}{1231} &  \multicolumn{1}{p{1.97cm}|}{\centering 7} &  \multicolumn{1}{p{4.16cm}|}{\centering 0001} \\
 \cline{1-1} \cline{2-2} \cline{3-3}
 0000 &  \multicolumn{1}{p{1.97cm}|}{\centering 0} &  \multicolumn{1}{p{4.16cm}|}{\centering 0000} \\
\hline
\end{tabular}\end{center}
The action on $\frak u_2$  is the 7-dimensional vector action, which came up previously for $SO(5,4)$, node 1.  It has orbits there of dimensions 7, 6, and 0, with respective basepoints 1111, 1222, and 0000. Its orbits here are:
\begin{center}\begin{tabular}{||p{2.44cm}||p{1.69cm}||p{4.22cm}||}
\hline
 \multicolumn{1}{|p{2.44cm}|}{\centering Orbit Basepoint} &  \multicolumn{1}{p{1.69cm}|}{\centering Dimension} &  \multicolumn{1}{p{4.22cm}|}{\centering Coadjoint orbit intersected} \\
\hline
 \multicolumn{1}{|p{2.44cm}|}{\centering 1232} &  \multicolumn{1}{p{1.69cm}|}{\centering 7} &  \multicolumn{1}{p{4.22cm}|}{\centering 0001} \\
 \cline{1-1} \cline{2-2} \cline{3-3}
 \multicolumn{1}{|p{2.44cm}|}{\centering 2342} &  \multicolumn{1}{p{1.69cm}|}{\centering 6} &  \multicolumn{1}{p{4.22cm}|}{\centering 1000} \\
 \cline{1-1} \cline{2-2} \cline{3-3}
 \multicolumn{1}{|p{2.44cm}|}{\centering 0000} &  \multicolumn{1}{p{1.69cm}|}{\centering 0} &  \multicolumn{1}{p{4.22cm}|}{\centering 0000} \\
\hline
\end{tabular}\end{center}

\subsection{Type $G_2$}

\begin{center}
  \includegraphics{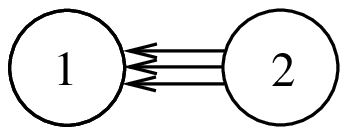} \\
  $G_2$ Dynkin diagram.
\end{center}

The following table lists the internal Chevalley modules
for the two conjugacy classes of maximal parabolic subgroups of $G_2$.
The actions on $\frak u_2$ are trivial, while the action for node 1 on $\frak u_3$ is the   2-dimensional standard representation of $SL(2)$ (it has two orbits:~zero and nonzero).  We list the actions below and the orbits on $\frak u_1$, which are most interesting in the case of the symmetric cube action for node 2.

\begin{center}\begin{tabular}{|c|c|c|c|c|}
\hline
 Node &  Type of $[L,L]$ &  $i\,=\,1$ &  $i\,=\,2$ & $i\,=\,3$ \\
\hline
 1 &  \multicolumn{1}{l|}{$SL(2)$} &  Standard &  Trivial  & Standard  \\
  &  $\dim \frak u_i$ &  2 &  1 & 2 \\
  &  action &  $\varpi$ &   &  ($SL(3)$ node 1) \\
\hline
 2 &  \multicolumn{1}{l|}{$SL(2)$} &  Symmetric cube &  Trivial  &  \\
  &  $\dim \frak u_i$ &  4 &  1 &  \\
  &  action &  $3\varpi$ &   &  \\
\hline
\end{tabular}\end{center}

\begin{center}\begin{tabular}{|cp{2cm}||p{4.22cm}||}
\hline
 \multicolumn{1}{|c|}{\centering Orbit Basepoint} &  \multicolumn{1}{p{2cm}|}{\centering Dimension} &  \multicolumn{1}{p{4.22cm}|}{\centering Coadjoint orbit intersected} \\
\hline
 \multicolumn{2}{|c|}{Node 1 } &  \multicolumn{1}{p{4.22cm}|}{\centering} \\
\hline
 \multicolumn{1}{|c|}{11} &  \multicolumn{1}{p{2cm}|}{\centering 2} &  \multicolumn{1}{p{4.22cm}|}{\centering 10} \\
 \multicolumn{1}{|c|}{00} &  \multicolumn{1}{p{2cm}|}{\centering 0} &  \multicolumn{1}{p{4.22cm}|}{\centering 00} \\
 \multicolumn{1}{|c|}{} &  \multicolumn{1}{p{2cm}|}{\centering} &  \multicolumn{1}{p{4.22cm}|}{\centering} \\
\hline
 \multicolumn{2}{|c|}{Node 2 } &  \multicolumn{1}{p{4.22cm}|}{\centering} \\
\hline
 \multicolumn{1}{|c|}{01+31} &  \multicolumn{1}{p{2cm}|}{\centering 4} &  \multicolumn{1}{p{4.22cm}|}{\centering 02} \\
 \multicolumn{1}{|c|}{21} &  \multicolumn{1}{p{2cm}|}{\centering 3} &  \multicolumn{1}{p{4.22cm}|}{\centering 10} \\
 \multicolumn{1}{|c|}{31} &  \multicolumn{1}{p{2cm}|}{\centering 2} &  \multicolumn{1}{p{4.22cm}|}{\centering 01} \\
 \multicolumn{1}{|c|}{00} &  \multicolumn{1}{p{2cm}|}{\centering 0} &  \multicolumn{1}{p{4.22cm}|}{\centering 00} \\
\hline
\end{tabular}\end{center}

\end{document}